\documentclass{article}

\usepackage{microtype}
\usepackage{graphicx}
\usepackage{subcaption}
\usepackage{booktabs} 

\usepackage{hyperref}

\usepackage[preprint]{icml2026}

\usepackage{amsmath}
\usepackage{amssymb}
\usepackage{mathtools}
\usepackage{amsthm}
\usepackage{multirow, hhline}
\usepackage{pifont}
\newcommand{\cmark}{\ding{51}}%
\newcommand{\xmark}{\ding{55}}%

\usepackage[capitalize,noabbrev]{cleveref}
\usepackage{amsmath}\allowdisplaybreaks
\usepackage{amsfonts,bm}

\def\1{\bf{1}}

\def\va{{\bf{a}}}

\def\fD{{\mathcal{D}}}

\def\fL{{\mathcal{L}}}

\DeclareMathOperator*{\argmax}{arg\,max}
\DeclareMathOperator*{\argmin}{arg\,min}

\usepackage{amsthm}
\usepackage{etoolbox}
\theoremstyle{plain}

\makeatletter
\def\Ddots{\mathinner{\mkern1mu\raise\p@
\vbox{\kern7\p@\hbox{.}}\mkern2mu
\raise4\p@\hbox{.}\mkern2mu\raise7\p@\hbox{.}\mkern1mu}}
\makeatother

\makeatletter
\newcommand*{\rom}[1]{\expandafter\@slowromancap\romannumeral #1@}
\makeatother

\newtheorem{theorem}{Theorem}
\newtheorem{proposition}[theorem]{Proposition}
\newtheorem{lemma}[theorem]{Lemma}
\newtheorem{corollary}[theorem]{Corollary}
\theoremstyle{definition}
\newtheorem{definition}[theorem]{Definition}
\newtheorem{assumption}[theorem]{Assumption}
\theoremstyle{remark}
\newtheorem{remark}[theorem]{Remark}

\def\va {{{\bf a}}}

\def\A{{\bf A}}

\def\C{{\bf C}}

\def\g{{\bf g}}

\def\H{{\bf H}}
\def\I{{\bf I}}

\def\T{{\bf T}}

\def\X{{\bf X}}
\def\x{{\bf x}}

\def\y{{\bf y}}
\def\Z{{\bf Z}}

\def\0{{\bf 0}}
\def\1{{\bf 1}}

\def\OM{{\mathcal O}}

\def\RB{{\mathbb R}}

\def \fx{{\boldsymbol{x}}}
\def \fy{{\boldsymbol{y}}}
\def \fw{{\boldsymbol{w}}}
\def \fs{{\boldsymbol{s}}}
\def \fz{{\boldsymbol{z}}}

\usepackage{enumitem}
\usepackage{hhline}

\makeatletter
\let\c@theorem\relax    
\theoremstyle{plain}
\newtheorem{theorem}{Theorem}[section]
\newtheorem{proposition}[theorem]{Proposition}
\newtheorem{lemma}[theorem]{Lemma}
\newtheorem{corollary}[theorem]{Corollary}
\theoremstyle{definition}
\newtheorem{definition}[theorem]{Definition}
\newtheorem{assumption}[theorem]{Assumption}
\theoremstyle{remark}
\newtheorem{remark}[theorem]{Remark}

\usepackage[textsize=tiny]{todonotes}

\icmltitlerunning{Second-Order Bilevel Optimization}

\begin{document}

\twocolumn[
  \icmltitle{Second-Order Bilevel Optimization with Accelerated Convergence Rates}



  \icmlsetsymbol{equal}{*}

  \begin{icmlauthorlist}
    \icmlauthor{Sheng Yang}{equal,ucr}
    \icmlauthor{Chengchang Liu}{equal,cuhk}
    \icmlauthor{Lesi Chen}{tsinghua}
    \icmlauthor{John C.S. Lui}{cuhk}

  \end{icmlauthorlist}

  \icmlaffiliation{ucr}{Department of Statistics, University of California, Riverside}
  \icmlaffiliation{cuhk}{Department of Computer Science \& Engineering, The Chinese University of Hong Kong}
  \icmlaffiliation{tsinghua}{IIIS, Tsinghua University}
  \icmlcorrespondingauthor{Chengchang Liu}{7liuchengchang@gmail.com}

  \icmlkeywords{Machine Learning, ICML}

  \vskip 0.3in
]



\printAffiliationsAndNotice{}  

\begin{abstract}
This paper studies second-order methods for nonconvex-strongly-convex bilevel optimization. We propose a novel fully second-order bilevel approximation method (FSBA) that achieves an iteration complexity of $\tilde{\mathcal{O}}(\epsilon^{-1.5})$ for finding the $(\epsilon, \OM(\sqrt{\epsilon}))$ second-order stationary point of the hyper-objective function.
Our results demonstrate that second-order methods can achieve an accelerated convergence rate than first-order methods in bilevel optimization.
To address the heavy computational cost associated with the second-order oracle, we introduce a lazy variant of FSBA, called LFSBA, which reuses second-order information across several iterations. We prove that LFSBA exhibits better computational complexity than FSBA by a factor of $\sqrt{d}$, where $d$ is the dimension of the problem. 
We also apply a similar idea to nonconvex strongly-concave minimax optimization and propose the lazy minimax cubic-regularized Newton (LMCN) method with better computational complexity compared to existing second-order methods.
\end{abstract}

\section{Introduction}
In this paper, we consider the following bilevel optimization problem:
\begin{equation}
\label{eq:bilevel}
\begin{split}
    \min _{\boldsymbol{x} \in \mathbb{R}^{d_x}} \varphi(\boldsymbol{x}) &:= f\left(\boldsymbol{x}, \boldsymbol{y}^*(\boldsymbol{x})\right), \\
    \text{where} \quad \boldsymbol{y}^*(\boldsymbol{x}) &:= \argmin _{\boldsymbol{y} \in \mathbb{R}^{d_y}} g(\boldsymbol{x}, \boldsymbol{y}).
\end{split}
\end{equation}
We assume that the lower-level function $g$ is strongly convex with respect to $\fy$.
This formulation is widely applied in various machine learning applications, including but not limited to hyperparameter tuning \cite{franceschi2018bilevel,pedregosa2016hyperparameter}, neural architecture search \cite{liu2018darts,zhang2021idarts,zoph2016neural}, meta-learning \cite{ji2022theoretical, rajeswaran2019meta}, reinforcement learning \cite{hong2023two, konda1999actor}, and adversarial training~\cite{bruckner2011stackelberg,goodfellow2020generative,robey2023adversarial,zhang2022revisiting}.

The strong convexity of lower-level function $g$ with respect to $\fy$ and proper smooth assumptions on $f$ and $g$ ensure the differentiability of $\varphi(\x)$, whose gradient can be expressed as:
\begin{align*}
   \! &\nabla\varphi(\fx) = \nabla_x f(\fx,\fy^*(\fx)) \\
    &~~-\nabla_{xy}^2g(\fx,\fy^*(\fx))\big(\nabla^2_{yy}g(\fx,\fy^*(\fx))\big)^{-1}\nabla_y g(\fx,\fy^*(\fx)).
\end{align*}
Previous second-order methods, such as approximate implicit differentiation (AID) \cite{ghadimi2018approximation,ji2021bilevel,liao2018reviving, lorraine2020optimizing} and iterative differentiation (ITD) \cite{arbel2022amortized, bolte2021nonsmooth,domke2012generic,franceschi2017forward,franceschi2018bilevel}, utilize Hessian-vector products \cite{ji2021bilevel,li2022fully} to estimate the hypergradient $\nabla \varphi(\boldsymbol{x})$.
They perform inexact gradient descent or (perturbed) accelerate gradient descent~\cite{yang2023accelerating,wang2024efficient} to minimize $\varphi(\cdot)$.
The iteration complexities of these methods are consistent with those of first-order methods for non-convex minimization problems: $\OM(\epsilon^{-2})$ for the gradient descent type algorithm when $\nabla\varphi(\cdot)$ is Lipschitz continuous, and $\OM(\epsilon^{-1.75})$ for the accelerated gradient descent methods if $\nabla^2\varphi(\cdot)$ is Lipschitz continuous.

\begin{table*}[t] 
\centering
\caption{Comparison of computational complexities for finding an $\epsilon$-stationary point of the hyper-objective $\varphi(\boldsymbol{x}) := f(\boldsymbol{x}, \boldsymbol{y}^*(\boldsymbol{x}))$ under Assumption \ref{assum:basic}. \textbf{Finding SOSP} indicates whether the algorithm can find an approximate second-order stationary point.}
\label{tab:compare-alg}

\small
\renewcommand{\arraystretch}{1.25} 

\begin{tabular}{@{}llccc@{}} 
\toprule
\textbf{Oracle} & \textbf{Method} & \textbf{Iteration Complexity} & \textbf{Hessian Computations} & \textbf{Find SOSP?} \\ 
\midrule
\midrule

\multirow{5}{*}{1st ($f$, $g$)} 
  & PZOBO \cite{sow2022convergence} & $\tilde{\mathcal{O}}(d_x^2 \kappa^6\epsilon^{-4})$~$^{\text{(a)}}$ & - & \xmark \\ 
  & BOME \cite{liu2022bome} & $\tilde{\mathcal{O}}(\text{poly}(\kappa)\epsilon^{-6})$~$^{\text{(b)}}$ & - & \xmark \\ 
  & F$^2$BA \cite{Kown23Fully} & $\tilde{\mathcal{O}}(\kappa^7\epsilon^{-3})$~$^{\text{(c)}}$ & - & \xmark \\ 
  & F$^2$BA \cite{Chen23Nearoptimal} & $\tilde{\mathcal{O}}(\kappa^4\epsilon^{-2})$ & - & \cmark \\
  & RAF$^2$BA \cite{Chen23Nearoptimal,yang2023accelerating} & $\tilde{\mathcal{O}}(\kappa^{3.75}\epsilon^{-1.75})$ & - & \cmark \\
\midrule

\multirow{6}{*}{1st ($f$) + 2nd ($g$)} 
  & AID \cite{ghadimi2018approximation} & $\mathcal{O}(\kappa^5\epsilon^{-2.5})$  & every iteration~$^\dag$ & \xmark \\ 
  & AID-BiO \cite{ji2021bilevel} &  ${\mathcal{O}}(\kappa^4\epsilon^{-2})$  & every iteration~$^\dag$ & \xmark \\ 
  & ITD-BiO \cite{ji2021bilevel} &  $\tilde{\mathcal{O}}(\kappa^4\epsilon^{-2})$  & -$^*$ & \xmark \\ 
  & iNEON \cite{huang2025efficiently} & $\tilde{\mathcal{O}}(\text{poly}(\kappa)\epsilon^{-2})$ & every iteration~$^\dag$ & \cmark \\ 
  & RAHGD \cite{yang2023accelerating} &  $\tilde{\mathcal{O}}(\kappa^{3.25}\epsilon^{-1.75})$ & every iteration~$^\dag$ & \cmark \\
  & IAPUN \cite{wang2024efficient} &  $\tilde{\mathcal{O}}(\kappa^{4.75}\epsilon^{-1.75})$  & every iteration & \cmark \\
\midrule

\multirow{2}{*}{2nd ($f$, $g$)} 
  & FSBA (Alg.~\ref{alg:FSBA}) &  $\tilde{\mathcal{O}}(\kappa^3\epsilon^{-1.5})$  & every iteration & \cmark \\ 
  & LFSBA (Alg.~\ref{alg:LFSBA}) &  $\tilde{\mathcal{O}}(\kappa^3 m^{0.5}\epsilon^{-1.5})$  & once every $m$ iterations & \cmark 
  \\
\bottomrule
\end{tabular}

\vspace{0.5em}
\begin{minipage}{0.95\linewidth} 
    \scriptsize 
    \textbf{Note:}
    (a) Assumes $\|\nabla^2 g(\boldsymbol{x}, \boldsymbol{y}) - \nabla^2 g(\boldsymbol{x}', \boldsymbol{y}')\|_F^2 \leq \rho_g^2 (\|\boldsymbol{x} - \boldsymbol{x}'\|^2 + \|\boldsymbol{y} - \boldsymbol{y}'\|^2)$, stronger than Assumption 3.1c. 
    (b) Additionally assumes both $|f(\boldsymbol{x}, \boldsymbol{y})|$ and $|g(\boldsymbol{x}, \boldsymbol{y})|$ are upper bounded. 
    (c) Additionally assumes $\nabla^2 f$ is Lipschitz and gradients are bounded. \\
    $^\dag$ Each iteration requires computing $(\nabla_{yy}^2g(\boldsymbol{x},\boldsymbol{y}))^{-1}\boldsymbol{v}$, typically via Conjugate Gradient (CG). \\
    $^*$ The algorithm does not explicitly construct the Hessian, though the analytic form involves 2nd-order derivatives.
\end{minipage}
\end{table*}

If we let the Lagrange function as $\mathcal{L}_\lambda(\boldsymbol{x}, \boldsymbol{y}):=f(\boldsymbol{x}, \boldsymbol{y})+\lambda\left(g(\boldsymbol{x}, \boldsymbol{y})-g\left(\boldsymbol{x}, \fy^*(\boldsymbol{x}\right)\right)$, \citet{Kown23Fully} shows that \eqref{eq:bilevel} can be effectively solved by the following formulation:
\begin{equation}
\label{eq:approximation_L_function}
\begin{split}
    \min _{\boldsymbol{x} \in \mathbb{R}^{d_x}} \mathcal{L}_\lambda^*(\boldsymbol{x}) &:= \mathcal{L}_\lambda\left(\boldsymbol{x}, \boldsymbol{y}_\lambda^*(\boldsymbol{x})\right), \\
    \text{where} \quad \boldsymbol{y}_\lambda^*(\boldsymbol{x}) &:= \argmin _{\boldsymbol{y} \in \mathbb{R}^{d_y}} \mathcal{L}_\lambda(\boldsymbol{x}, \boldsymbol{y}).
\end{split}
\end{equation}
The gradient $\nabla \fL_{\lambda}^*(\x)$ can be expressed as:
\begin{equation}
\label{eq:gradient-lagrange}
\begin{split}
    \nabla \mathcal{L}_\lambda^*(\boldsymbol{x}) &= \nabla_x f(\boldsymbol{x}, \boldsymbol{y}_\lambda^*(\boldsymbol{x})) \\
    &\quad + \lambda \big( \nabla_x g(\boldsymbol{x}, \boldsymbol{y}_\lambda^*(\boldsymbol{x})) - \nabla_x g(\boldsymbol{x}, \boldsymbol{y}^*(\boldsymbol{x})) \big).
\end{split}
\end{equation}
which can be approximated by using only first-order oracles of $f$ and $g$. 
\citet{Kown23Fully} proposed F$^2$BA, which performs the inexact gradient descent on $\fL_{\lambda}^*(\cdot)$.
They showed that F$^2$BA can find the $\epsilon$-stationary point of $\varphi(\cdot)$ within $\OM(\epsilon^{-3})$ calls to the gradient of $f$ and $g$. 
Later, \citet{Chen23Nearoptimal} improved the iteration complexities of F$^2$BA to $\tilde{\OM}(\epsilon^{-2})$ and showed that this complexity could be further enhanced to $\tilde{\OM}(\epsilon^{-1.75})$ if $\nabla^2\varphi(\cdot)$ is Lipschitz continuous. 
It is important to note that these rates (nearly) match those of the second-order methods for bilevel optimization mentioned earlier, while requiring only first-order oracles.
Given these results, we find the advantages of accessing second-order oracles in bilevel optimization appear to be rather limited. 
Thus, it is a natural question to ask: 

\textit{``Can we develop a second-order method with improved iteration complexities that demonstrates the benefits of utilizing second-order oracles in bilevel optimization?''}

To achieve better iteration complexity, a straightforward approach is to employ an inexact Newton-type algorithm that utilizes not only the hypergradient $\nabla\varphi(\fx)$, but also the hyperHessian $\nabla^2\varphi(\fx)$.
However, accessing $\nabla^2\varphi(\fx)$ requires the third-order derivatives of $g$, which is too expensive or even not feasible.
In this work, we adopt the idea of fully first-order bilevel approximation methods to tackle the bilevel optimization problem~\eqref{eq:bilevel} via its approximation~\eqref{eq:approximation_L_function}.
We design a fully second-order bilevel approximation (FSBA) method that estimates both $\nabla \fL_{\lambda}^{*}(\fx)$ and $\nabla^2 \fL_{\lambda}^*(\fx)$, performing inexact cubic-regularized Newton iterations to find the stationary point of $\fL_{\lambda}^*(\cdot)$.
We prove that our proposed FSBA method requires only $\tilde{\OM}(\epsilon^{-1.5})$ oracles to the gradient and Hessian of $f$ and $g$, which surpass the performance of the state-of-the-art second-order methods for bilevel optimization. 
Since nonconvex-strongly-concave bilevel optimization problems subsume nonconvex optimization problems,
the oracle complexity of FSBA is near-optimal, aligning with the lower bound $\Omega(\epsilon^{-1.5})$ by~\citet{carmon2020lower}.

Considering the heavy computational cost of second-order oracles, existing second-order methods often use iterative techniques to compute the product of the Hessian inverse and the gradient, or they apply approximation algorithms to reduce the computational burden \cite{ghadimi2018approximation,yang2025lancbio}. More recently, quasi-Newton methods have been developed to efficiently solve the lower-level problem~\cite{fang2025qnbo}. 
Therefore, it is also crucial to reduce the computational complexity of the FSBA method.

To address this, we leverage the concept of lazy Hessians \cite{doikov2023second,doikov2023first} and propose a lazy fully second-order bilevel approximation method (LFSBA). 
Instead of computing the approximate Hessian at each iteration, as is common in the existing literature, LFSBA estimates $\nabla^2 \fL_{\lambda}^*(\cdot)$ only at specific snapshot points and reuses this approximate Hessian for the next $m$ iterations.
We demonstrate that LFSBA achieves an iteration complexity of $\tilde{\OM}(\sqrt{m}\epsilon^{-1.5})$, providing a better computational complexity than FSBA by properly tuning $m$.  
We compare FSBA and LFSBA with existing first- and second-order methods in Table~\ref{tab:compare-alg}.

\paragraph{Paper organization.}
In Section~\ref{sec:pre}, we introduce the notation and assumptions. 
In Section~\ref{sec:FSBA}, we propose the FSBA method and study its convergence behavior.
In Section~\ref{sec:LFSBA}, we propose the lazy variant of FSBA method (LFSBA) and show that it has a better computational complexity than FSBA. 
We conduct experiments to validate our theoretical results in Section~\ref{sec:exp} and summarize this paper in Section~\ref{sec:conclu}.
All proofs are deferred to the Appendix.

\section{Preliminaries}
\label{sec:pre}
We first introduce some basic notation.
For a twice differentiable function $f(\boldsymbol{x}, \boldsymbol{y})$, its partial gradients with respect to $\boldsymbol{x}$ and $\boldsymbol{y}$ are denoted
by $\nabla_x f(\boldsymbol{x}, \boldsymbol{y})$ and $\nabla_y f(\boldsymbol{x}, \boldsymbol{y})$, respectively. 
The Hessian matrix of $f$ at point $(\boldsymbol{x}, \boldsymbol{y})$ is partitioned as $\nabla^2 f(\boldsymbol{x}, \boldsymbol{y})=\begin{bmatrix}
    \nabla_{x x}^2 f(\boldsymbol{x}, \boldsymbol{y}), \nabla_{x y}^2 f(\boldsymbol{x}, \boldsymbol{y})\\
   \nabla_{y x}^2 f(\boldsymbol{x}, \boldsymbol{y}), \nabla_{y y}^2 f(\boldsymbol{x}, \boldsymbol{y}) 
\end{bmatrix}$.
We use $\|\cdot\|$ to denote the spectral norm of matrices and the Euclidean norm of vectors, respectively. 
Given a real symmetric matrix $\A$, we let $\lambda_{\min }(\A)$ denote its smallest eigenvalue. 
Furthermore, we denote the following quantity, for any 
$\boldsymbol{x} \in \mathbb{R}^{d_x}$, $\xi(\boldsymbol{x}):=\left[-\lambda_{\min }\left(\nabla^2 f(\boldsymbol{x}) \right)\right]_{+}$, where
$[t]_{+}:= \max \{t, 0\}$ denotes the positive part. 
Finally, we use $\mathcal{O}(\cdot)$ to hide only absolute constants that are independent of any problem parameters, $\tilde{\mathcal{O}}(\cdot)$ to additionally hide polylogarithmic factors, and $\Omega(\cdot)$ to denote an asymptotic lower bound up to a constant factor.

We now introduce the following assumptions on the upper-level function $f$ and the
lower-level function $g$. 
\begin{assumption}
\label{assum:basic}  
Suppose that the upper-level function $f$, the lower-level function $g$ and the hyper-objective $\varphi$ satisfy the following conditions:
\begin{enumerate}[label=(\alph*),leftmargin=0.7cm]
\item  $g(\boldsymbol{x}, \boldsymbol{y})$  is three times differentiable and $\mu$-strongly convex with respect to $\boldsymbol{y}$ for any fixed $\boldsymbol{x}$;\vskip-0.14cm
\item $f(\boldsymbol{x}, \boldsymbol{y})$ is twice differentiable and $C$-Lipschitz with respect to $\boldsymbol{x}$ and $\boldsymbol{y}$;\vskip-0.14cm
\item $g(\boldsymbol{x}, \boldsymbol{y})$ and $f(\boldsymbol{x}, \boldsymbol{y})$ are $\ell$-gradient Lipschitz with respect to $\boldsymbol{x}$ and $\boldsymbol{y}$;\vskip-0.14cm
\item $f(\boldsymbol{x}, \boldsymbol{y})$ and $g(\boldsymbol{x}, \boldsymbol{y})$ are $\rho$-Hessian Lipschitz with respect to $\boldsymbol{x}$ and $\boldsymbol{y}$;\vskip-0.14cm
\item $g(\boldsymbol{x}, \boldsymbol{y})$ is $\nu$-third-order derivative Lipschitz with respect to $\boldsymbol{x}$ and $\boldsymbol{y}$;\vskip-0.14cm
\item $\varphi(\boldsymbol{x})$ is lower bounded, i.e. $\varphi^*:=\min _{\boldsymbol{x} \in \mathbb{R}^{d_x}} \varphi(\boldsymbol{x})>-\infty$.
\end{enumerate}
\end{assumption}
These assumptions are common and standard for the nonconvex strongly-convex bilevel optimization~\cite{yang2023accelerating,Chen23Nearoptimal,Kown23Fully, ghadimi2018approximation}. 
We define the condition number as follows.
\begin{definition}
    \label{def:condition-number}
We define the largest smoothness constant $\bar{\ell}:=\max \left\{C, \ell, \nu, \rho \right\}$ and the condition number $\kappa:=\bar{\ell} / \mu$.
\end{definition}
Assumption~\ref{assum:basic} means that $\fy^*(\x)$ is Lipschitz continuous, as presented in the following proposition.
\begin{proposition}[Lemma 2.2, \citet{ghadimi2018approximation}]
\label{prop:condinum_y_star}
Under Assumption~\ref{assum:basic}, $\fy^*(\cdot)$ is $\kappa$-Lipschitz continuous.
\end{proposition}
Furthermore, when $\lambda$ is large enough, we also have that $\fL_{\lambda}^{*}(\cdot)$ is a good proxy of $\varphi(\cdot)$ and that $\nabla \fL_{\lambda}^*(\cdot)$, $\nabla^2 \fL_{\lambda}^*(\cdot)$ are Lipschitz continuous.
\begin{proposition}[Lemma 4.1, Lemma 5.1 in \citet{Chen23Nearoptimal}, and Lemma 3.2 in \citet{chen2025faster}; Lemma 3.1 in \citet{Kown23Fully}]
\label{prop:proxy_property}
Under Assumption~\ref{assum:basic} and let $\lambda \geq 2 \ell / \mu$, we have:
\begin{enumerate}[label=(\alph*),leftmargin=0.5cm]
\item  $\left|\mathcal{L}_\lambda^*(\boldsymbol{x}) - \varphi(\boldsymbol{x})\right| \!=\! \mathcal{O}({\bar{\ell} \kappa^2}/{\lambda})$, $\left\|\nabla \mathcal{L}_\lambda^*(\boldsymbol{x}) - \nabla \varphi(\boldsymbol{x})\right\|\! =\! \mathcal{O}({\bar{\ell} \kappa^3}/{\lambda})$, and $\left\|\nabla^2 \mathcal{L}_\lambda^*(\boldsymbol{x}) - \nabla^2 \varphi(\boldsymbol{x})\right\| = \mathcal{O}({\bar{\ell} \kappa^5}/{\lambda})$ hold for all $\x\in\RB^{d_x}$.
\vskip-0.2cm
\item $\nabla \mathcal{L}_\lambda^*(\boldsymbol{x})$ is $L$-Lipschitz continuous, where $L\!:= \mathcal{O}\!\left(\bar{\ell} \kappa^3\!\right)$. \vskip-0.2cm
\item $\nabla^2\mathcal{L}_\lambda^*(\boldsymbol{x})$ is $\bar{\rho}$-Lipschitz continuous, where $\bar{\rho} \!\!:= \mathcal{O}\!\left(\bar{\ell} \kappa^5\!\right)\!.$
\end{enumerate}
\end{proposition}
In addition, $\fL_{\lambda}(\fx,\cdot)$ is strongly convex in $\fy$ when $\lambda$ is large enough.
\begin{proposition}[Lemma 3.2, \citet{Kown23Fully}]
\label{prop:strong_convex_L} 
Under Assumption \ref{assum:basic}, if $\lambda \geq 2 \ell / \mu$, then $\mathcal{L}_\lambda(\boldsymbol{x}, \cdot)$ is $(\lambda \mu / 2)$-strongly convex, $(1+\lambda)\ell$-smooth.
The condition number of $\fL_{\lambda}(\fx,\cdot)$ is $3\kappa$.
\end{proposition}

Finally, we give the formal definition of the $\epsilon$-first-order stationary points and $(\epsilon, \tau)$-second-order stationary points as follows.
\begin{definition}
\label{def:2}  
We call $\hat{\boldsymbol{x}}$ an $\epsilon$-first-order stationary point (FOSP) of $\varphi(\boldsymbol{x})$ if $\|\nabla \varphi(\hat{\boldsymbol{x}})\| \!\leq \!\epsilon$.
\end{definition}
\begin{definition}
\label{def:3}  
We call $\hat{\boldsymbol{x}}$ an $(\epsilon, \tau)$-second-order stationary point (SOSP) of $\varphi(\boldsymbol{x})$ if $\|\nabla \varphi(\hat{\boldsymbol{x}})\| \leq \epsilon$ and $\lambda_{\min }\left(\nabla^2 \varphi(\hat{\boldsymbol{x}})\right) \geq-\tau$.
\end{definition}

\section{Fully Second-Order Bilevel Approximation Method}
\label{sec:FSBA}
In this section, we introduce our fully second-order bilevel approximation method (FSBA) and present its convergence analysis.
We also introduce an inexact variant of second-order method for practical consideration.
\subsection{The FSBA Method}
Proposition~\ref{prop:proxy_property} shows that $\fL_{\lambda}^*(\cdot)$ defined in \eqref{eq:approximation_L_function} is a good approximation of $\varphi(\cdot)$, and that $(\epsilon,\OM(\sqrt{\epsilon}))$-SOSP $\fL_{\lambda}^*(\x)$ is also an $(\OM(\epsilon),\OM(\sqrt{\epsilon}))$-SOSP of $\varphi(\cdot)$.
Hence, the main idea of our FSBA method is to apply second-order methods to solve the proxy function $\fL_{\lambda}^*(\cdot)$ instead of $\varphi(\cdot)$.

The Hessian of $\fL_{\lambda}^*(\fx)$ can be expressed by
{\small\begin{align*}
    &\nabla^2\fL_{\lambda}^*(\fx) = \nabla^2_{xx} \fL_{\lambda}(\fx,\fy_{\lambda}^*(\fx))\\
    &~ \nabla_{x y}^2 \mathcal{L}_\lambda\big(\boldsymbol{x}, \fy_\lambda^*(\boldsymbol{x})\big)
     \left[\nabla_{y y}^2 \mathcal{L}_\lambda\big(\boldsymbol{x}, \fy_\lambda^*(\boldsymbol{x})\big)\right]^{-1} \nabla_{y x}^2 \mathcal{L}_\lambda\big(\boldsymbol{x}, \fy_\lambda^*(\boldsymbol{x})\big)
\end{align*}}
\!\!where the Hessian block of $\fL_{\lambda}(\x,\y)$ can be computed individually by only second-order information of $f$ and $g$.
Ideally, using $\nabla \fL^*_{\lambda}(\fx)$ and $\nabla^2 \fL^*_{\lambda}(\fx)$ to construct the cubic-regularized Newton (CRN) method on $\fL^{*}(\cdot)$ with the following update direction
\begin{align}
\label{eq:cubic-CRN}
\begin{split}
\!\!\fs^* \!=\! \argmin_{\fs\in\RB^{d_x}}\! \big\{\fs^{\top}\nabla\fL_{\lambda}^*(\fx)\!+\!\frac{\fs^{\top}\!\nabla^2 \fL_{\lambda}^*(\fx)\fs}{2}\!+\!\frac{M\!\|\fs\|^3}{6}\!\big\}
\end{split}
\end{align}
can find an $(\epsilon,\OM(\sqrt{\epsilon}))$-SOSP of $\fL_{\lambda}^*(\cdot)$ within $\OM(\epsilon^{-1.5})$ iterations. 
However, $\nabla \fL_{\lambda}^*(\fx)$ and $\nabla^2\fL_{\lambda}^*(\fx)$ cannot be computed exactly since they contain $\fy_{\lambda}^*(\fx)$ and $\fy^*(\fx)$.
The following lemma shows that $\nabla \fL_{\lambda}^*(\fx)$ and $\nabla^2\fL_{\lambda}^*(\fx)$ can be estimated by introducing additional variables $\fy$ and $\fw$. 
\begin{lemma}
\label{lem:estimators_error} 
Under Assumption~\ref{assum:basic} and let 
{\small\begin{align}
\begin{split}
\label{eq:grad-est}
    \g(\fx;\fy,\fw) : = \nabla_x f\left(\boldsymbol{x}, \boldsymbol{y}\right) \!+\! \lambda\left(\nabla_x g\left(\boldsymbol{x}, \boldsymbol{y}\right) \!-\! \nabla_x g\left(\boldsymbol{x}, \boldsymbol{w}\right)\right)
\end{split}\\
\label{eq:Hess-est}
\begin{split}
   &\H(\fx;\fy,\fw):=  \\
   &~~  \nabla_{x x}^2 f\left(\boldsymbol{x}, \fy\right) +\lambda\big(\nabla_{x x}^2 g\left(\boldsymbol{x}, \fy\right)-\nabla_{x x}^2 g\left(\boldsymbol{x}, \fw\right)\big)\\
   &~~~~-\nabla_{x y}^2 \mathcal{L}_\lambda\left(\boldsymbol{x}, \fy\right)\!\left[\nabla_{y y}^2 \mathcal{L}_\lambda\left(\boldsymbol{x}, \fy\right)\right]^{-1}\!\nabla_{y x}^2 \mathcal{L}_\lambda\left(\boldsymbol{x}, \fy\right)\\
   &~~~~+\lambda \nabla_{x y}^2 g\left(\boldsymbol{x}, \fw\right)\left[\nabla_{y y}^2 g\left(\boldsymbol{x}, \fw\right)\right]^{-1}\nabla_{y x}^2 g\left(\boldsymbol{x}, \fw\right), 
\end{split}
\end{align}
}
\!\!then we have
{\small
\begin{align*}
&\left\|{\nabla} \mathcal{L}_\lambda^*(\fx)\!-\!\g(\fx;\fy,\!\fw)\!\right\| \!\leq\! 2 \lambda \ell \left\|\boldsymbol{y}\!-\!\fy_\lambda^*\left(\boldsymbol{x}\right)\right\|\!+\!\lambda \ell\! \left\|\boldsymbol{w}\!-\!\fy^*\!(\boldsymbol{x})\right\|,
    \\
&\left\|{\nabla}^2 \mathcal{L}_\lambda^*\!\left(\boldsymbol{x}\right)\!-\!\H(\fx;\fy,\!\fw)\!\right\|\! \leq \!C_1\! \left\|\boldsymbol{w}\!-\!\fy^*\!\!\left(\boldsymbol{x}\right) \right\| \!+
\!C_2\!
 \left\|\boldsymbol{y} \!-\!\fy_\lambda^*\!\left(\boldsymbol{x}\right)\right\|,
\end{align*}
}
\!\!where $C_1 := \mathcal{\OM}(\lambda \bar{\ell} +\bar{\ell} \kappa^2)$ and $C_2 := \mathcal{\OM}( \lambda\bar{\ell} \kappa^2 )$.
\end{lemma}
Since $\fL_{\lambda}(\fx,\cdot)$ and $g(\fx,\cdot)$ are strongly convex according to Proposition~\ref{prop:strong_convex_L} and Assumption~\ref{assum:basic} (a),
it will be easy to find $\fy\approx \fy_{\lambda}^*(\fx)$ and $\fw\approx \fy^*(\fx)$
by applying the proper first- or second-order method for strongly convex optimization, i.e., the accelerated gradient descent method (AGD, Algorithm~\ref{alg:AGD}), on $\fL_{\lambda}(\fx,\cdot)$ and $g(\fx,\cdot)$, respectively.
This leads to $\g(\fx;\fy,\fw)\approx\nabla\fL_{\lambda}^*(\fx)$ and $\H(\fx;\fy,\fw)\approx\nabla^2\fL_{\lambda}^*(\fx)$ by Lemma~\ref{lem:estimators_error}.
Replacing $\nabla \fL^*_{\lambda}(\fx)$ and  $\nabla^2 \fL^*_{\lambda}(\fx)$ by $\g(\fx;\fy,\fw)$ and $\H(\fx;\fy,\fw)$ in the CRN update~\eqref{eq:cubic-CRN} leads to our FSBA method, as presented in Algorithm~\ref{alg:FSBA}.

\begin{algorithm}
\caption{ AGD $\left(h(\cdot), \boldsymbol{z}_0, K, \eta, \theta\right)$}
\label{alg:AGD}
\begin{algorithmic}[1]
    \STATE \textbf{Input:} $\tilde{\boldsymbol{z}}_0=\boldsymbol{z}_0$
    \FOR{$k=0$ to $K-1$}
         \STATE  $\fz_{k+1}=\tilde{\boldsymbol{z}}_k-\eta \nabla h\left(\tilde{\boldsymbol{z}}_k\right)$
    \STATE $\tilde{\boldsymbol{z}}_{k+1}=\boldsymbol{z}_{t+1}+\theta\left(\boldsymbol{z}_{k+1}-\boldsymbol{z}_k\right)$
    \ENDFOR
    \STATE \textbf{Output:} $\fz_K$
\end{algorithmic}
\end{algorithm}

\begin{algorithm}
\caption{Fully Second-Order Bilevel Approximation method (FSBA)}
\label{alg:FSBA}
\begin{algorithmic}[1]
\STATE \textbf{Input:} $\boldsymbol{x}_0 \in \mathbb{R}^{d_x}, \boldsymbol{y}_{-1}=\mathbf{0}$, $\boldsymbol{w}_{-1}=\boldsymbol{0}$, $T$, $\ell_1$, $\ell_2$, $\kappa_1$, $\kappa_2$, $\epsilon$, $M$, 
$
\left\{K_t^1\right\}_{t=0}^T, \left\{K_t^2\right\}_{t=0}^T$
\\[0.1cm]
\FOR{$t = 0, 1, \cdots, T-1$}
\STATE $\boldsymbol{w}_t=\operatorname{AGD}\left(g\left(\boldsymbol{x}_t, \cdot\right), \boldsymbol{w}_{t-1}, K_t^1, \frac{1}{\ell_1}, \frac{\sqrt{\kappa_1}-1}{\sqrt{\kappa_1}+1}\right)$
\STATE $\boldsymbol{y}_t=\operatorname{AGD}\left(\mathcal{L}_\lambda(\boldsymbol{x}_t, \cdot), \boldsymbol{y}_{t-1}, K_t^2, \frac{1}{\ell_2}, \frac{\sqrt{\kappa_2}-1}{\sqrt{\kappa_2}+1}\right)$
\\[0.1cm]
\STATE Compute $\g_t=\g(\fx_t;\fy_t,\fw_t)$ according to \eqref{eq:grad-est}.
\\[0.1cm]
\STATE Compute $\H_t=\H(\fx_t;\fy_t,\fw_t)$ according to \eqref{eq:Hess-est}.
\\[0.1cm]
\STATE 

$
 \fs^*_{t} =\argmin_{\fs\in\RB^{dx}}\left\{{\g_t}^{\top}{\fs}\! +\! \frac{1}{2}\fs^{\top}\H_t\fs+\frac{M}{6}\|\fs\|^3\right\}   
$
\\[0.1cm]
\STATE \textbf{If}  {$\left\|\fs^*_t\right\| \leq \frac{1}{2} \sqrt{\epsilon / M}$}~\textbf{then~break}
\\[0.1cm]
\STATE $\fx_{t+1}=\fx_t +\fs^*_t$
\ENDFOR
\STATE \textbf{Output:} $\hat{\fx} = \fx_{t+1}$
\end{algorithmic}
\end{algorithm}


\subsection{Convergence Analysis of FSBA}
In this section, we provide the convergence analysis of FSBA. 
The following lemma shows that once $\g(\fx_t;\fy_t,\fw_t)$ and $\H(\fx_t;\fy_t,\fw_t)$ are close enough to $\nabla \fL^*_{\lambda}(\fx_t)$ and $\nabla^2\fL^*_{\lambda}(\fx_t)$,
FSBA enjoys a similar convergence rate as applying the exact CRN method~ \eqref{eq:cubic-CRN}.
\begin{lemma}[Theorem 1, \citet{Luo22finding}]
\label{lem:inexact-cubic}
Under Assumption~\ref{assum:basic}, if we run Algorithm~\ref{alg:FSBA} with $M\geq \bar{\rho}$ and 
$T = \lceil192\left(\fL_{\lambda}^*(\fx_0)-\min_{\fx}\fL_{\lambda}^*(\fx)\right)\rceil\sqrt{M}\epsilon^{-3/2}$, and suppose the following condition
\begin{align}
\label{eq:inexact-cubic-iteration}
\begin{split}
    &\left\|\nabla \mathcal{L}_\lambda^*(\boldsymbol{x}_t)-\g(\fx_t;\fy_t,\fw_t)\right\| \leq C_g \epsilon, \\
    &\left\|\nabla^2 \mathcal{L}_\lambda^*(\boldsymbol{x}_t)-\H(\fx_t;\fy_t,\fw_t)\right\| \leq C_H \sqrt{M \epsilon},
\end{split}
\end{align}
hold with $C_g:=1/192$ and $C_H:=1/48$, 
then $\hat{\fx}$ is an $(\epsilon,\sqrt{M\epsilon})$-SOSP of $\fL^*_{\lambda}(\cdot)$.
\end{lemma}
 
In the following lemma, we show that the above condition can be achieved by properly choosing the iteration numbers $K_t^1$ and $K_t^2$ in the AGD subroutine.

\begin{lemma}
\label{lem:exact_enough}
Under Assumption~\ref{assum:basic}, let {\small $\Delta=\varphi\left(\boldsymbol{x}_0\right) -\varphi^{*}$, 
$\tilde{\epsilon}=\min \left\{\frac{C_g \epsilon}{4\lambda \ell},\frac{C_H \sqrt{M \epsilon}}{ 2 C_2} \right\}$, 
$R=\max\{\left\|\fy^*\left(\boldsymbol{x}_0\right)\right\|,  \left\|\fy_\lambda^*\left(\boldsymbol{x}_0\right)\right\|\}$}, 
if we run Algorithm~\ref{alg:FSBA} with  $M\geq \bar{\rho}$, $\kappa_1 = \kappa$, $\ell_1 = \ell$, $\kappa_2 = 3\kappa$, $\ell_2 = (1+\lambda)\ell$, $\lambda = \max \left\{ \bar{\ell} \kappa^2 / \Delta, \bar{\ell}\kappa^3 / \epsilon, \bar{\ell}\kappa^5 / \sqrt{M\epsilon}\right\}$,  and
{\small
\begin{align*}
    K_t^1\!=\!K_t^2\!=\!\begin{cases}
    &\!\!\!\!\!\!\left\lceil 2 \sqrt{\kappa_2} \log \left(\frac{\sqrt{\kappa_2+1}}{\tilde{\epsilon}}R\right)\right\rceil~~~~~~~~~~~~~~~~~~~~~~~~t=0 \\[0.3cm]
     &\! \!\!\!\!\!\left\lceil 2 \sqrt{\kappa_2} \log \left(\frac{\sqrt{\kappa_2+1}}{\tilde{\epsilon}}\left(\tilde{\epsilon}+4\kappa\left\|\fs^*_{t-1}\right\|\right)\right)\right\rceil~t\geq 1
 \end{cases},
\end{align*}}
\!\!then the condition \eqref{eq:inexact-cubic-iteration} in Lemma~\ref{lem:inexact-cubic} is satisfied. 
\end{lemma}
Combining Lemma~\ref{lem:inexact-cubic} and \ref{lem:exact_enough}, we know that FSBA can find $(\epsilon,\sqrt{M\epsilon})$-SOSP of $\fL_{\lambda}^*(\cdot)$ with iteration complexities of $\tilde{\OM}(\epsilon^{-1.5})$.
In the following theorem, we formally present the first- and second-order oracle complexities of FSBA to find the SOSP of $\varphi(\cdot)$.
\begin{theorem}
\label{thm:FSBA}
Under Assumption~\ref{assum:basic}, run Algorithm~\ref{alg:FSBA} with the same setting as Lemma~\ref{lem:exact_enough}, let $M=\Omega( \bar{\rho})$ and  
$T = \Theta((\varphi(\fx_0)-\varphi^*)\sqrt{M}\epsilon^{-3/2})$, 
then
$\hat{\fx}$ is an $((\OM(\epsilon),\OM(\kappa^{2.5}\bar{\ell}^{0.5}\epsilon^{0.5}))$-SOSP of $\varphi(\cdot)$. 
In addition, 
the complexities of the first- and second-order oracle can be bounded by $\tilde{\OM}(\kappa^3\bar{\ell}^{0.5}\epsilon^{-1.5})$ and ${\OM}(\kappa^{2.5}\bar{\ell}^{0.5}\epsilon^{-1.5})$, respectively.
\end{theorem}

\subsection{An Inexact Version of FSBA}
\label{sec:IFSBA}
Both the computation of $\H(\fx_t;\fy_t,\fw_t)$ (line 6) and solving the cubic-regularized subproblem (line 7) in Algorithm~\ref{alg:FSBA} require the explicit construction of Hessian and the inverse of regularized Hessian, 
which may limit the application of FSBA in large-scale problems when the problem dimension is extremely large.

In this section, we propose an inexact variant of FSBA.
Instead of accessing $\H(\fx_t;\fy_t,\fw_t)$ directly according to \eqref{eq:Hess-est}, we compute $\C(\fx_t;\fy_t,\fw_t)$
\begin{align}
\label{eq:Cheby-est}
\begin{split}
   &\C(\fx_t;\fy_t,\fw_t) \\
   &\!\!:= \! \nabla_{x x}^2 f\!\left(\boldsymbol{x}_t, \fy_t\right)
  \! +\!\lambda\big(\nabla_{x x}^2 g\!\left(\boldsymbol{x}_t, \fy_t\right)\!-\!\nabla_{x x}^2 g\!\left(\boldsymbol{x}_t, \fw_t\right)\!\big)\\
   &~~~~~+\lambda \nabla_{x y}^2 g\left(\boldsymbol{x}_t, \fw_t\right) \C_{1,t}\nabla_{y x}^2 g\left(\boldsymbol{x}_t, \fw_t\right)\\
   &~~~~~-\nabla_{x y}^2 \mathcal{L}_\lambda\left(\boldsymbol{x}_t, \fy_t\right)\C_{2,t}\nabla_{y x}^2 \mathcal{L}_\lambda\!\left(\boldsymbol{x}_t, \fy_t\right),
\end{split}
\end{align}
which replaces $\nabla^2_{yy}g(\fx_t,\fw_t)^{-1}$ and  $\nabla_{yy}^{2}\fL_{\lambda}(\fx_t,\fy_t)^{-1}$ by their Chebyshev Polynomials approximations $\C_{1,t}$ and $\C_{2,t}$. 
In addition, we do not solve the cubic subproblem by regularized Newton step, but instead, using a gradient-type method to approximately solve 
\begin{align*}
    \min_{\fs\in\RB^{d_x}} m(\fs) =  \fs^{\top}\g_t + \frac{1}{2}\fs^{\top}\C_t\fs +\frac{M}{6}\|\fs\|^3,
\end{align*}
whose gradient $\nabla m(\fs)= \g_t + \C_t\fs + \frac{M}{2}\|\fs\|\fs$ can be computed with only gradients and Hessian-vector products of $f$ and $g$.
We present the inexact fully second-order bilevel approximation method (ISFBA) in Algorithm~\ref{alg:IFSBA}.
The detailed implementation of constructing $\C_{1,t}$, $\C_{2,t}$ in $\C_t$ (line 6) and the sub-problem solvers (line 7 and line 10) are presented in Appendix~\ref{app:ifsba}.

\begin{algorithm}
\caption{Inexact Fully Second-Order Bilevel Approximation method (IFSBA)}
\label{alg:IFSBA}
    \begin{algorithmic}[1]
     \STATE \textbf{Input:} $\boldsymbol{x}_0 \in \mathbb{R}^{d_x},  \boldsymbol{y}_{-1}=\mathbf{0}$, $\ell_1$, $\ell_2$, $\kappa_1$, $\kappa_2$, $\epsilon$, $M$, $\boldsymbol{w}_{-1}=\boldsymbol{0}, T, \left\{K_t^1\right\}_{t=0}^T, \left\{K_t^2\right\}_{t=0}^T$, 
\FOR{$t = 0, 1, \cdots, T-1$}
\STATE $\boldsymbol{w}_t=\operatorname{AGD}\left(g\left(\boldsymbol{x}_t, \cdot\right), \boldsymbol{w}_{t-1}, K_t^1, \frac{1}{\ell_1}, \frac{\sqrt{\kappa_1}-1}{\sqrt{\kappa_1}+1}\right)$
\\[0.1cm]
\STATE $\boldsymbol{y}_t=\operatorname{AGD}\left(\mathcal{L}_\lambda(\boldsymbol{x}_t, \cdot), \boldsymbol{y}_{t-1}, K_t^2, \frac{1}{\ell_2}, \frac{\sqrt{\kappa_2}-1}{\sqrt{\kappa_2}+1}\right)$\
\STATE Compute $\g_t=\g(\boldsymbol{x}_t;\boldsymbol{y}_t,\boldsymbol{w}_t)$ according to \eqref{eq:grad-est}
\\[0.1cm]
\STATE Compute $\C_t=\C(\boldsymbol{x}_t;\boldsymbol{y}_t,\boldsymbol{w}_t)$ according to \eqref{eq:Cheby-est}.
\\[0.1cm]
\STATE $(\fs_t, \Delta_t) = \operatorname{Cubic\text{-}Solver}(\g_t, {\C}_t, \sigma, \mathcal{K}(\epsilon, \delta'))$
\\[0.1cm]
    \STATE $\boldsymbol{x}_{t+1} = \boldsymbol{x}_t + \boldsymbol{s}_t$
\\[0.1cm]
\STATE \textbf{If} {$\Delta_t > -\dfrac{\epsilon^3}{128M}$} \textbf{then}
\\[0.1cm]
        \STATE $\quad \hat{\fs} = \operatorname{Final\text{-}Cubic\text{-}Solver}(\g_t, {\C}_t, \epsilon)$
        \\[0.1cm]
        \STATE $\quad \boldsymbol{x}_{t+1} = \boldsymbol{x}_t + \hat{\fs}$
        \\[0.1cm]
        \STATE $\quad$ \textbf{break}
    \STATE \textbf{end If}
\ENDFOR
\STATE \textbf{Output:} $\hat{\boldsymbol{x}} = \boldsymbol{x}_{t+1}$
\end{algorithmic}
\end{algorithm}

\section{Lazy Fully Second-Order Bilevel Approximation Algorithm with Better Computational Complexity}
\label{sec:LFSBA}
In the previous section, we propose FSBA with $\tilde{\OM}(\epsilon^{-1.5})$ oracle complexity for non-convex strongly convex bilevel optimization, which is faster than existing first-order methods.
However, the second-order oracle always leads to a heavier computational complexity than the first-order oracle.
We make the following assumption to differentiate the computational complexity of first- and second-order oracles by following~\citet{doikov2023second}.
\begin{assumption}
\label{assum:computational complexity}  
We count the computational complexity of first-order oracle of $f$ and $g$, i.e., $\nabla_x f(\fx,\fy)$, $\nabla_y f(\fx,\fy)$, $\nabla_x g(\fx,\fy)$, $\nabla_y g(\fx,\fy)$ and HVPs computed via automatic differentiation, as $N$.
We count the computational complexity of the second-order oracle of $f$ and $g$, i.e., $\nabla^2_{xx}f(\fx,\fy)$, $\nabla^2_{xy}f(\fx,\fy)$, $\nabla^2_{xx}g(\fx,\fy)$, $\nabla^2_{yy} g(\fx,\fy)$, $\nabla^2_{xy}g(\fx,\fy)$, as $dN$,
where $d := \max\{d_x,d_y\}$ denotes the problem dimension.
\end{assumption}
The computational complexity of FSBA can be bounded by 
\begin{align}
\label{eq:C(FSBA)}
\begin{split}
&\text{Cost(FSBA)}  \\
&=N \cdot \# \text{1st-order oracle} + Nd \cdot\#\text{2nd-order oracle}\\
&= \tilde{\OM}\big(N(\kappa^{0.5}+d)\kappa^{2.5}\bar{\ell}^{0.5}\epsilon^{-1.5}\big).
\end{split}
\end{align}
By Theorem~\ref{thm:imcn_main_complexity}, IFSBA attains an $(\epsilon,\sqrt{\epsilon})$-SOSP with the following cost:
\begin{align}
\label{eq:C(IFSBA)}
\begin{split}
&\text{Cost(IFSBA)}  \\
&= N \cdot \# \text{gradient oracle} + N \cdot \# \text{HVP oracle} \\
&= \tilde{\OM}\big(
N(\kappa^3\bar{\ell}^{0.5}\epsilon^{-1.5}
+
\kappa^{3.5}\bar{\ell}\epsilon^{-2})
\big).
\end{split}
\end{align}
\subsection{The Lazy FSBA Method and its Convergence Analysis}
In this section, our aim is to reduce the computational complexity of FSBA.
At each iteration of FSBA, it takes $\tilde{\OM}(\kappa^{0.5}N)$ computational complexity to obtain $\fw_t$ and $\fy_t$ by AGD, and $\OM(dN)$ computational complexity to update $\fx_t$ by the inexact CRN step.
When $d\gg \kappa^{0.5}$ such that the computational complexity of a second-order oracle is large, it is expensive to call the second-order oracle for every iteration.
Motivated by the lazy Hessian mechanism~\cite{doikov2023second,doikov2023first,chen2024second,liu2025enhanced,chen2025faster}, we propose the lazy fully second-order bilevel approximation method (LFSBA), which computes the approximate Hessian only at the snapshot point and reuses it for the next $m$ iterations.
We formally present LFSBA method in Algorithm~\ref{alg:LFSBA}.
\begin{algorithm}
\caption{Lazy Fully Second-order Bilevel Approximation method (LFSBA)}
\label{alg:LFSBA}
    \begin{algorithmic}[1]
     \STATE \textbf{Input:} $\boldsymbol{x}_0 \in \mathbb{R}^{d_x}$, $\boldsymbol{y}_{-1}=\mathbf{0}$,$\boldsymbol{w}_{-1}=\boldsymbol{0}$, $\ell_1$, $\ell_2$, $\kappa_1$, $\kappa_2$, $m$, $\epsilon$, $M$, $T$, $\left\{K_t^1\right\}_{t=0}^T, \left\{K_t^2\right\}_{t=0}^T$.
\vskip 0.1cm
\FOR{$t = 0, 1, \cdots, T-1$}
\STATE $\boldsymbol{w}_t=\operatorname{AGD}\left(g\left(\boldsymbol{x}_t, \cdot\right), \boldsymbol{w}_{t-1}, K_t^1, \frac{1}{\ell_1}, \frac{\sqrt{\kappa_1}-1}{\sqrt{\kappa_1}+1}\right)$
\STATE $\boldsymbol{y}_t=\operatorname{AGD}\left(\mathcal{L}_\lambda(\boldsymbol{x}_t, \cdot), \boldsymbol{y}_{t-1}, K_t^2, \frac{1}{\ell_2}, \frac{\sqrt{\kappa_2}-1}{\sqrt{\kappa_2}+1}\right)$
\STATE Compute $\g_t= \g(\fx_t;\fy_t,\fw_t)$ according to \eqref{eq:grad-est}\vskip 0.2cm
\STATE \textbf{if} $t\%m = 0$ \vskip 0.1cm
\STATE \quad Set $\tilde{\fx}=\fx_t$ \vskip 0.1cm
\STATE \quad Compute $\tilde{\H}= \H(\fx_t;\fy_t,\fw_t)$ according to \eqref{eq:Hess-est} \vskip 0.1cm
\STATE
$\fs_t^* = \argmin_{\fs\in\RB^{dx}}\big\{\g_t^{\top} \fs + \frac{1}{2}\fs^{\top}\tilde{\H}\fs + \frac{M}{6}\|\fs\|^3\big\}$
\vskip 0.1cm
\STATE
$\fx_{t+1}=\fx_t + \fs^*_t$
\vskip 0.2cm
\textbf{If} 
$\epsilon \geq \frac{1}{M}\big(\frac{288}{287}\big)^2
        \big(\frac{M+2 \bar{\rho}}{\sqrt{2}} \left\|\fs_t^*\right\|+\bar{\rho}\|\tilde{\fx}-\fx_{t}\|\big)^2$
\STATE ~\textbf{then~break}
\ENDFOR
\STATE \textbf{Ouput:} $\hat{\fx} = \fx_{t+1}$
\end{algorithmic}
\end{algorithm}

Now, we study the convergence analysis of LFSBA, which updates according to the following direction
\begin{align*}
    \fs^*_{t}= \argmin_{\fs\in\RB^{d_x}} \fs^{\top}\g_t + \frac{1}{2}\fs^{\top}\H_{\pi(t)}\fs + \frac{M}{6}\|\fs\|^3,
\end{align*}
where we denote $\g_t:=\! \g(\fx_t;\fy_t,\fw_t)$, $\H_t := \!\H(\fx_{t};\fy_t,\fw_t)$, and $\pi(t): = t-t~\text{mod}~ m$.
The following Lemma shows that once $\g_t$ and $\H_{\pi(t)}$ are good approximations of $\nabla \fL_{\lambda}^*(\fx_t)$ and $\nabla^2 \fL_{\lambda}^*(\fx_{\pi(t)})$, then LFSBA enjoys a similar descent property as the lazy cubic-regularized Newton method~\cite{doikov2023second}.
\begin{lemma}
    \label{lem:pre-LFSBA}
    Under Assumption~\ref{assum:basic}, let $M\geq \bar{\rho}$ and suppose
the following conditions
\begin{align}
\label{eq:condi-inexact-lazy-cubic}
\begin{split}
    &\left\|\nabla \mathcal{L}_\lambda^*(\boldsymbol{x}_t) - \g_t\right\| \leq \bar{C}_g \epsilon, \\
    &\left\|\nabla^2 \mathcal{L}_\lambda^*(\boldsymbol{x}_{\pi(t)}) - \H_{\pi(t)}\right\| \leq \bar{C}_H \sqrt{M \epsilon}.
\end{split}
\end{align}
hold with $\bar{C}_g:=1/576, \bar{C}_H:=1/288$ in Algorithm~\ref{alg:LFSBA}, denoting {\small$\gamma(\boldsymbol{x}) := \max \left\{\frac{1}{987M^2} \xi\left(\boldsymbol{x}\right)^3 , \frac{1}{120\sqrt{3M}} \left\|\nabla \mathcal{L}_\lambda^*\left(\boldsymbol{x}\right)\right\|^{3 / 2} \right\}$}, 
then it holds that
\begin{align}
\label{eq:cubic_progress_one_step}
\begin{split}
    \!\!\!\!&\mathcal{L}_\lambda^*(\boldsymbol{x}_{t})
    -\mathcal{L}_\lambda^*(\boldsymbol{x}_{t+1})
    \geq \\
    &~~\gamma(\boldsymbol{x}_{t+1}) \!+\!\frac{M}{72} \|\fx_{t+1}-\fx_t\|^3 \! -\!\frac{13 \bar{\rho}^3}{M^2}\|\boldsymbol{x}_{\pi(t)}\!-\!\boldsymbol{x}_t\|^3.
\end{split}
\end{align}
\end{lemma}
The following theorem indicates that by properly choosing the iteration steps in AGD subroutine  and the regularization parameter, LSFBA converges at a similar convergence rate as the lazy cubic-regularized-Newton method.

\begin{theorem}
\label{thm:Lazy_Hessian_calls}  
Under Assumption~\ref{assum:basic}, let {\small $\!\Delta\!=\varphi\left(\boldsymbol{x}_0\right)\! -\!\varphi^{*}$, \\
$\!\tilde{\epsilon}\!=\!\min \{\frac{\bar{C}_g \epsilon}{4\lambda \ell}, \frac{\bar{C}_H \sqrt{M \epsilon}}{ 2 C_2} \}$}, and {\small
$R \!=\!\max\{\big\|\fy^*(\boldsymbol{x}_0)\big\|,  \left\|\fy_\lambda^*\left(\boldsymbol{x}_0\right)\right\|\}$}, if we run Algorithm \ref{alg:LFSBA} with 
{\small $M \!=\! \Omega(m\bar{\rho})$, 
$T = \Theta\big(\Delta\sqrt{M}\epsilon^{-3/2}\big)$, $\lambda = \max \left\{ \bar{\ell} \kappa^2 / \Delta, \bar{\ell}\kappa^3 / \epsilon, \bar{\ell}\kappa^5 / \sqrt{M\epsilon}\right\}$,  $\kappa_1 = \kappa$, $\ell_1 = \ell$, $\kappa_2 = 3\kappa$, $\ell_2 = (1+\lambda)\ell$
}, and
{\small
\begin{align*}
K_t^1\! =\! K_t^2\! =\!\begin{cases}\!
&\!\!\!\!\!\!\!\left\lceil 2 \sqrt{\kappa_2} \log \left(\frac{\sqrt{\kappa_2+1}}{\tilde{\epsilon}}R\right)\!\right\rceil~~~~~~~~~~~~~~~~~~~~~~~~~t=0\\[0.3cm]
&\!\!\!\!\!\!\!\left\lceil 2 \sqrt{\kappa_2} \log \left(\frac{\sqrt{\kappa_2+1}}{\tilde{\epsilon}}\left(\tilde{\epsilon}+4\kappa\left\|\fs^*_{t-1}\right\|\right)\right)\!\right\rceil~t\geq 1
 \end{cases},
\end{align*}
}
\!\!then the output $\hat{\boldsymbol{x}}$ is an $(\OM(\epsilon), \OM(\kappa^{2.5}\bar{\ell}^{0.5}m^{0.5}\epsilon^{0.5}))$-SOSP of $\varphi(\boldsymbol{x})$. 
Furthermore, the first-order and second-order oracle complexities of Algorithm~\ref{alg:LFSBA} can be bounded by $\tilde{\OM}(\kappa^3\bar{\ell}^{0.5}m^{0.5}\epsilon^{-1.5})$ and ${\OM}(1+\kappa^{2.5}\bar{\ell}^{0.5}m^{-0.5}\epsilon^{-1.5})$, respectively.
\end{theorem}

\paragraph{Discussion on the computational complexity.}
Theorem~\ref{thm:Lazy_Hessian_calls} indicates that the iteration complexity $\tilde\OM(m^{0.5}\epsilon^{-1.5})$ of LFSBA is worse than $\tilde{\OM}(\epsilon^{-1.5})$ of FSBA, 
which is due to the reuse of Hessian.
However, considering the difference in computational complexity between first- and second-order oracles in Assumption~\ref{assum:computational complexity}, 
LFSBA achieves a better computational complexity by tuning $m$ for a trade-off of per-iteration computation cost and iteration complexity.
We state the computational complexity of LFSBA as follows
\begin{align*}
\begin{split}
    &\text{Cost(LFSBA)} \\
    &= N \cdot \# \text{1st-order oracle} + Nd \cdot \text{2nd-order oracle} \\
    &= \tilde{\OM} \big( N\kappa^3m^{0.5}\bar{\ell}^{0.5}\epsilon^{-1.5} 
    + Nd\kappa^{2.5}m^{-0.5}\bar{\ell}^{0.5}\epsilon^{-1.5} \big)\\
    &= \tilde{\OM} \big( N(\kappa^{0.5}+\kappa^{0.25}d^{0.5})\kappa^{2.5}\bar{\ell}^{0.5}\epsilon^{-1.5} \big),
\end{split}
\end{align*}
where the last inequality is by setting the frequency of update Hessian as $m = \Theta\left(1+\frac{d}{\sqrt{\kappa}}\right)$.

\begin{remark}
  The computational complexity of LFSBA improves FSBA \eqref{eq:C(FSBA)} by a factor of $d^{0.5}/\kappa^{0.25}$, significantly reducing the computational cost when the dimension is large.
\end{remark}
\begin{remark}
Once $\tilde{\H}=\H(\fx_{\pi(t)};\fy_{\pi(t)},\fw_{\pi(t)})$ is computed, the cubic regularized-Newton update of line 9 in Algorithm~\ref{alg:LFSBA} can be performed efficiently within $\tilde{\OM}(d^2)$ by performing the eigenvalue decomposition on $\tilde{\H}$~\cite{doikov2023second}.
\end{remark}
\subsection{Improved Results for Nonconvex Strongly-Concave Minimax Problems}

\begin{algorithm}
\caption{Lazy Minimax Cubic Newton method (LMCN)}
\label{alg:LCMN}
\begin{algorithmic}[1]
    \STATE \textbf{Input:} $\boldsymbol{x}_0 \in \mathbb{R}^{d_x}, \boldsymbol{y}_{-1}=\boldsymbol{0}, T,\left\{K_t\right\}_{t=0}^T,$ $\kappa_1$, $\ell_1$, $\epsilon$, $m$, $M$
    \FOR{$t=0,1\cdots, T-1$}
        \STATE $\boldsymbol{y}_t = \operatorname{AGD}\left(-f\left(\boldsymbol{x}_t, \cdot\right), \boldsymbol{y}_{t-1}, K_t, \frac{1}{\ell_1}, \frac{\sqrt{\kappa_1}-1}{\sqrt{\kappa_1}+1}\right)$
        \STATE Compute $\g_t = \g(\fx_t;\fy_t)$ according to \eqref{eq:g-minimax}.
        \vskip 0.2cm
\STATE \textbf{if} $t\%m = 0$ \vskip 0.1cm
\STATE \quad Set $\tilde{\fx}=\fx_t$\\[0.1cm]
\STATE \quad Compute $\tilde{\H}= \H(\fx_t;\fy_t)$ according to \eqref{eq:H-minimax} \vskip 0.1cm
\STATE
$\fs_t^* = \argmin_{\fs\in\RB^{dx}}\big\{\g_t^{\top} \fs + \frac{1}{2}\fs^{\top}\tilde{\H}\fs + \frac{M}{6}\|\fs\|^3\big\}$
\vskip 0.1cm
\STATE
$\fx_{t+1}=\fx_t + \fs_t^*$
\vskip 0.1cm
\STATE \textbf{If} 
$\epsilon\! \geq\! \frac{1}{M}\big(\frac{288}{287}\big)^2\!
        \big(\frac{M+2 \bar{\rho}}{\sqrt{2}} \|\fs_t^*\|\!+\!\bar{\rho}\|\tilde{\fx}-\fx_{t}\|\big)^2$
        \vskip 0.1cm
    \STATE \textbf{then~break}
    \ENDFOR
    \STATE \textbf{Output:} $\hat{\boldsymbol{x}} = \boldsymbol{x}_{t+1}$
\end{algorithmic}
\end{algorithm}
We adopts the idea of LFSBA to the solve following non-convex strongly-concave minimax problem
\begin{align}
    \label{eq:min-max}
    \min_{\fx\in\RB^{dx}}\max_{\fy\in\RB^{dy}} f(\fx,\fy).
\end{align}
Let $\varphi(\cdot):= \argmax_{\fy\in\RB^{d_x}} f(\cdot,\fy)$, the above minimax problem can be regarded as a special bilevel optimization problem~\eqref{eq:bilevel} with the lower function $g(\fx,\fy)=-f(\fx,\fy)$.
We suppose $f(\cdot,\cdot)$ and $\varphi(\cdot)$ satisfy the following assumption.
\begin{assumption}
\label{assum:minimax}
$f(\cdot,\cdot)$ and $\varphi(\cdot)$ satisfy the following conditions:
(a) ${f}(\boldsymbol{x}, \boldsymbol{y})$ is twice differentiable and $\mu$-strongly concave with respect to $\boldsymbol{y}$ for any fixed $\boldsymbol{x}$;
(b) $\nabla f(\cdot,\cdot)$ is $\ell$-Lipschitz continuous and and $\nabla^2 f(\cdot,\cdot)$ is $\rho$-Lipschitz continuous;
(c)  $\varphi^* :=\min _{\boldsymbol{x} \in \mathbb{R}^{d_x}}{\varphi}(\boldsymbol{x})>-\infty$.
\end{assumption}

Since $\nabla \varphi(\cdot)$ and $\nabla^2 \varphi(\cdot)$ are Lipschitz continuous~\cite{Luo22finding}, they can be well approximated by introducing an additional variable $\fy\approx \fy^*(\fx)$ such that 
\begin{small}
 \begin{align}
 \label{eq:g-minimax}
 &\g(\fx;\fy):= \nabla_x f(\fx,\fy)\\
\label{eq:H-minimax}
\begin{split}
&\H(\fx;\fy) :=\\
&\nabla_{x x}^2 {f}(\boldsymbol{x}, \fy)-\nabla_{x y}^2 {f}(\boldsymbol{x}, \fy)[\nabla_{y y}^2 {f}(\boldsymbol{x}, \fy)]^{-1}\nabla_{y x}^2 {f}(\boldsymbol{x}, \fy).
\end{split}
 \end{align}
\end{small} 
\!\!Then, applying a similar ``lazy'' strategy as introduced in LFSBA based on $\g(\fx;\fy)$ and $\H(\fx;\fy)$ leads to our lazy minimax cubic-regularized Newton method (LMCN), presented in Algorithm~\ref{alg:LCMN}.
The LMCN generalizes the MCN method~\cite{Luo22finding} and has better computational complexity than the MCN method in finding the SOSP of $\varphi(\fx)$.
\begin{theorem}
\label{thm:LMCN}
    Under Assumptions~\ref{assum:computational complexity} and \ref{assum:minimax}, LMCN (Algorithm~\ref{alg:LCMN}) can find an $(\epsilon,\kappa^{1.25}\sqrt{d\rho\epsilon})$-SOSP of $\varphi(\cdot)$, where $\kappa = \ell/\mu$ within computational complexity 
    \begin{align}
    \label{eq:compu-LMCN}
       \!\!\! \text{\rm Cost(LMCN)}\! =\! \tilde{\OM}(\!N\!(\kappa^{0.5}\!+\!\kappa^{0.25}d^{0.5})\kappa^{1.5}\!\rho^{0.5}\epsilon^{-1.5})\!).
    \end{align}
We let $\tilde{\epsilon}\!=\!\min \{ \epsilon / (576\ell), \sqrt{M \epsilon} / (288\rho)\}$, $\bar{\rho}=4\sqrt{2}\kappa^3 \rho$, and set $\kappa_1=\kappa$, $\ell_1=\ell$, $m\!=d/\sqrt{\kappa}\!+\!1$, 
$M=\Theta(m\bar{\rho})$, \\
$K_t = \begin{cases}
&\!\!\! \left\lceil 2 \sqrt{\kappa} \log \left(\frac{\sqrt{\kappa+1}}{\tilde{\epsilon}}\left\|\fy^*\left(\boldsymbol{x}_0\right)\right\|\right)\right\rceil~~~~~~~~~~~~~~~~~t=0\\
&\!\!\! \left\lceil 2 \sqrt{\kappa} \log \left(\frac{\sqrt{\kappa+1}}{\tilde{\epsilon}}\left(\tilde{\epsilon}+\kappa\left\|\fx_t-\fx_{t-1}\right\|\right)\right)\right\rceil~t\geq 1
     \end{cases}$, and $T = \Theta\left((\varphi(\fx_0)-\varphi^*)\sqrt{M}\epsilon^{-3/2}\right)$  in Algorithm~\ref{alg:LCMN} to achieve \eqref{eq:compu-LMCN}.
\end{theorem}
\begin{remark}
    The computational complexity of the LMCN is better than $\tilde{\OM}(N(\kappa^{0.5}+d)\kappa^{1.5}\rho^{0.5}\epsilon^{-1.5}))$ of the MCN~\cite{Luo22finding} due to the AM-GM inequality.
\end{remark}

\section{Numerical Experiments}
\label{sec:exp}
\subsection{Synthetic Minimax Problem}
\label{sec:syn}
We first consider the synthetic minimax problem with $f(\cdot,\cdot)$ defined as
$ f(\fx, \fy):=w\left(x_3\right)-10 y_1^2+x_1 y_1-5 y_2^2+x_2 y_2, $
where $\fx=[x_1, x_2, x_3]^{\top}$, $\fy=[y_1, y_2]^{\top}$, and $w(\cdot)$ is a multi-stage function defined in Appendix~\ref{app:expminimax}.

The experiments are conducted with different initial points:
$(\fx_1, \fy_1)=([10^{-3}, 10^{-3}, 10^{-1}]^{\top},[0,0]^{\top})$ and $(\fx_2, \fy_2)=([10^{-3},10^{-3},1]^{\top},[0,0]^{\top})$.
We compare our LMCN algorithm with the following baseline algorithms: PRAGDA \cite{yang2023accelerating}, MCN \cite{Luo22finding}, iMCN \cite{Luo22finding} and classical GDA \cite{lin2020gradient}. The results are shown in Figure \ref{fig:minmax_comparison}.

 \begin{figure*}[t]
    \centering
    \setlength{\tabcolsep}{6pt}
    \begin{tabular}{cc}
        \includegraphics[width=0.38\textwidth]{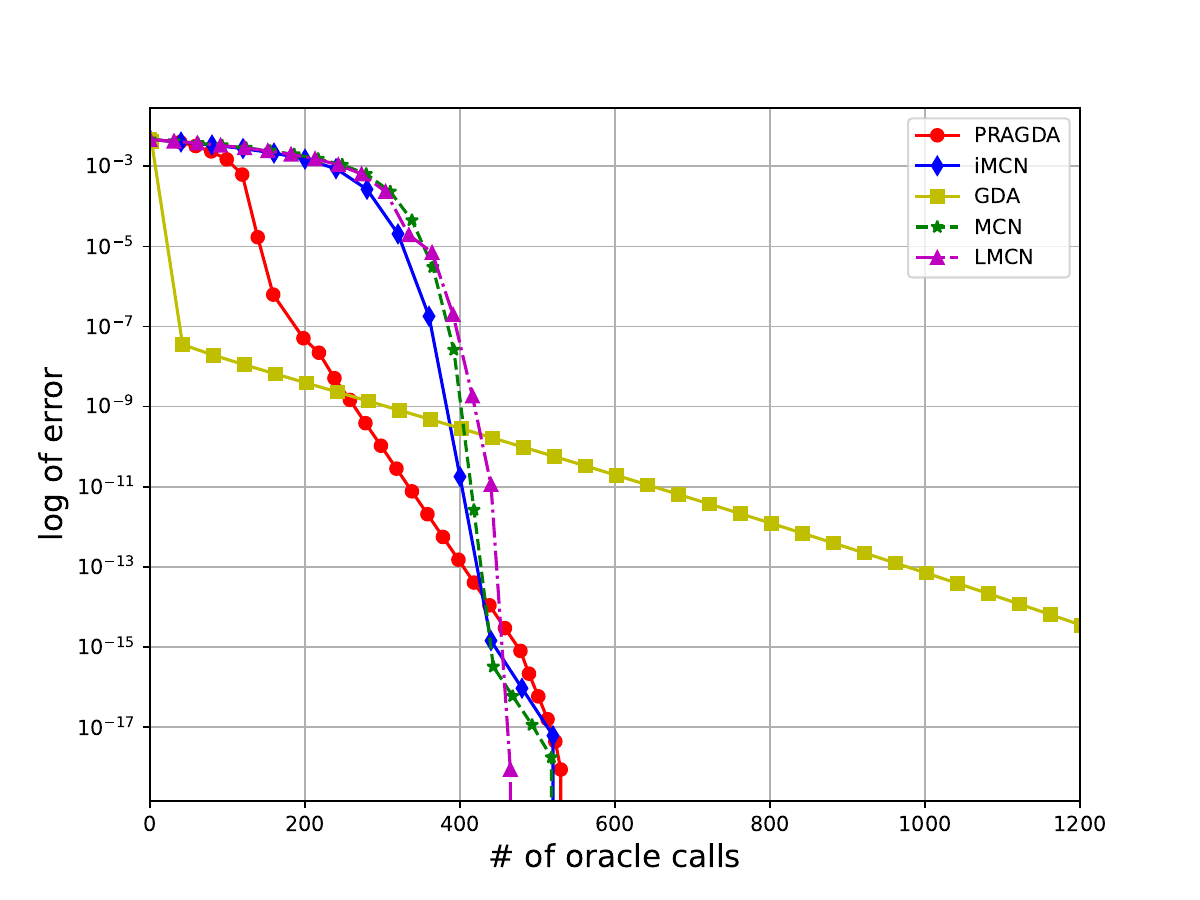} &
        \includegraphics[width=0.38\textwidth]{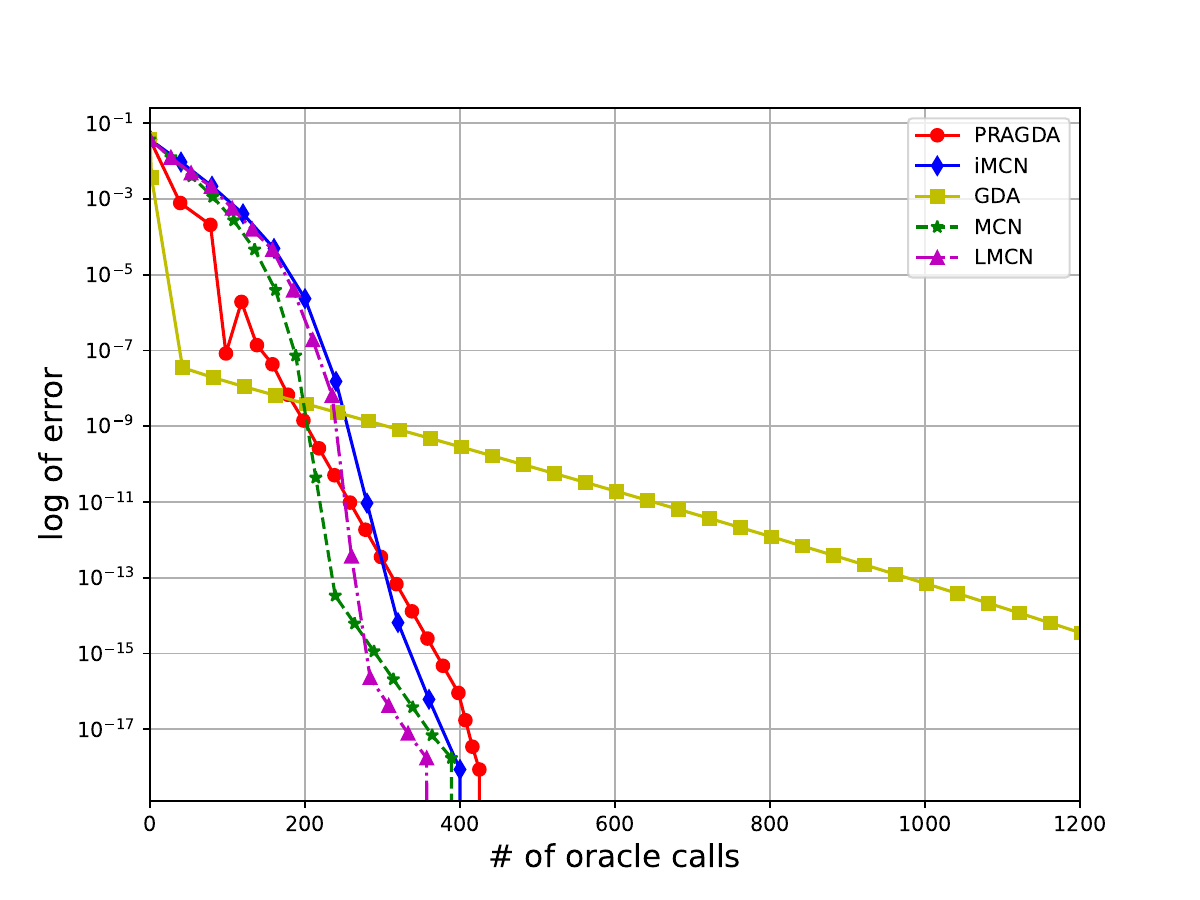} \\
        \scriptsize Initial point $(\fx_1, \fy_1)$ &
        \scriptsize Initial point $(\fx_2, \fy_2)$
    \end{tabular}
    \caption{Comparison of LMCN and baseline algorithms in terms of oracle calls under different initial points $(\fx_1,\fy_1)$ and $(\fx_2,\fy_2)$.}
    \label{fig:minmax_comparison}
\end{figure*}
 
\subsection{Data Hypercleaning}
\label{sec:hypercleaning}
\begin{figure*}[t] 
    \centering
    \setlength{\tabcolsep}{1pt} 
    
    \begin{tabular}{cccc} 
        \includegraphics[width=0.245\linewidth]{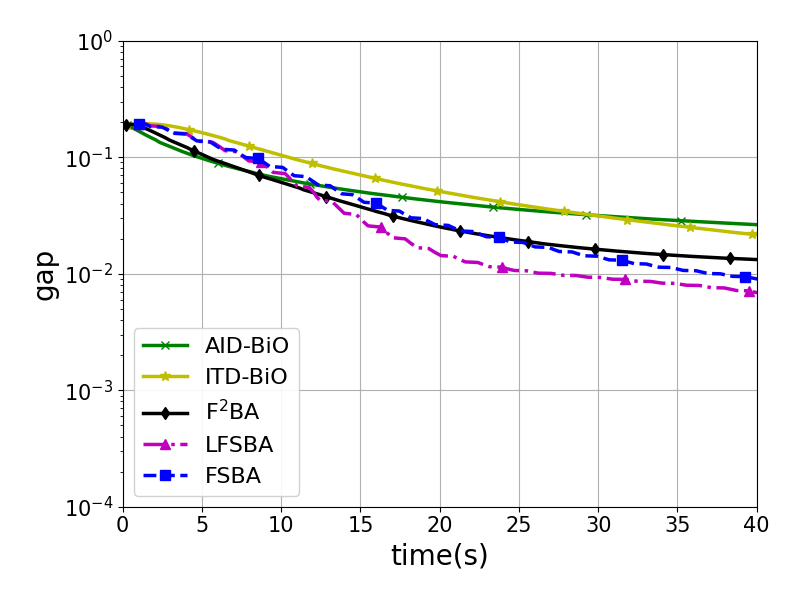} &
        \includegraphics[width=0.245\linewidth]{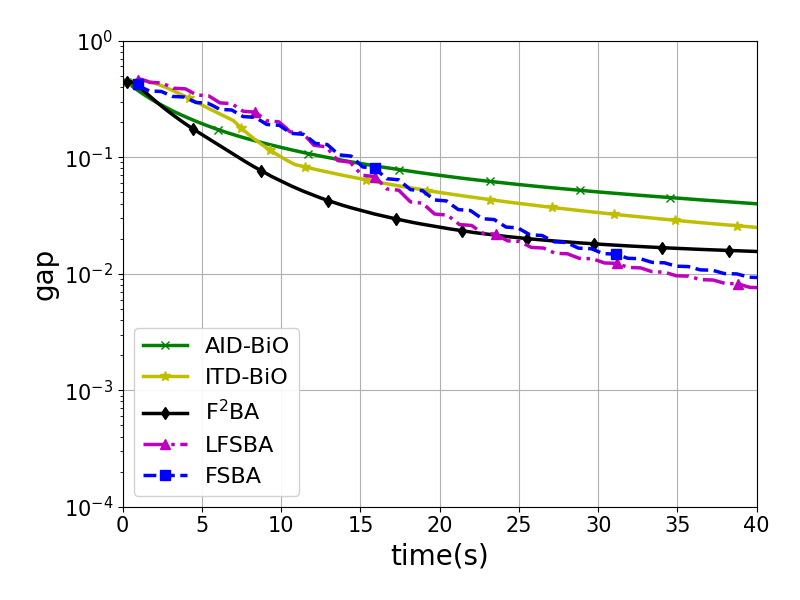} &
        \includegraphics[width=0.245\linewidth]{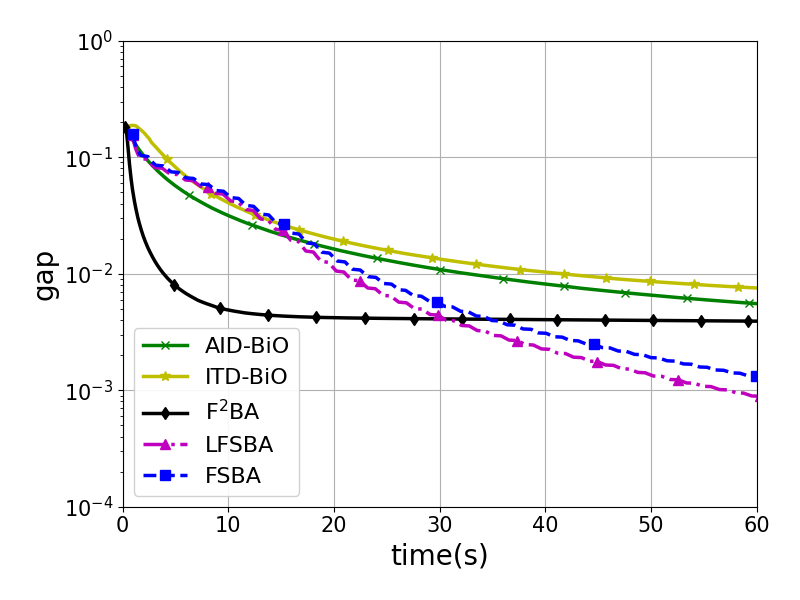} &
        \includegraphics[width=0.245\linewidth]{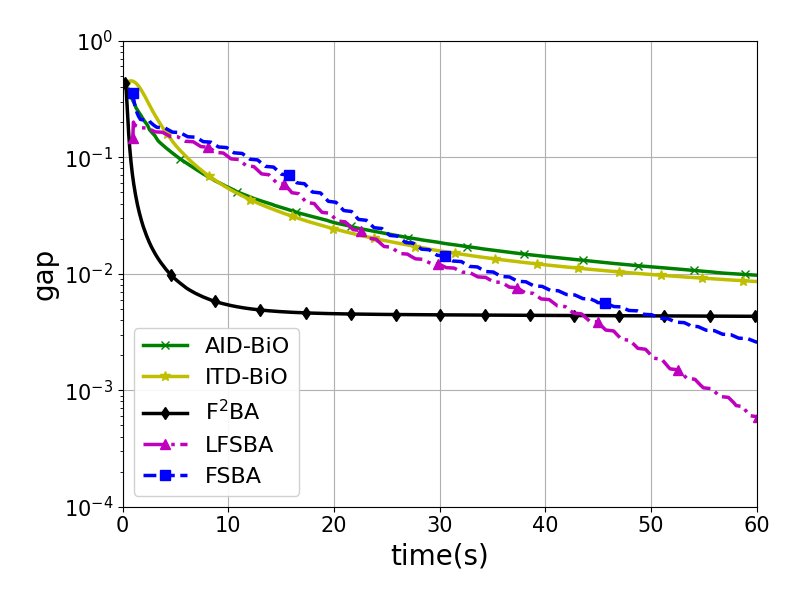} \\
        
        \scriptsize Breast $p=0.25$ & 
        \scriptsize Breast $p=0.5$ &
        \scriptsize Australian $p=0.25$ & 
        \scriptsize Australian $p=0.5$
    \end{tabular}
    
    \vspace{-0.5em} 
    \caption{Comparison of various bilevel algorithms with different noise rate $p$ on ``breast-cancer'' and ``australian'' datasets.}
    \label{fig:data hyper}
\end{figure*}

We then conduct experiments to validate the efficiency of the proposed methods on the \textit{data hyper-cleaning} task~\cite{franceschi2018bilevel,shaban2019truncated,zhou2022model}, which can be formulated as a bilevel optimization problem~\eqref{eq:bilevel} with  the following 
upper and lower-level objectives:
{\small
\begin{equation*}
\begin{split}
    f(\fx, \fy) &:= \frac{1}{|\mathcal{D}^{\mathrm{val}}|} \!\sum_{(\va_i, b_i) \in \mathcal{D}^{\mathrm{val}}} \ell(\langle \va_i, \fy\rangle, b_i), \\
    g(\fx, \fy) &:= \frac{1}{|\mathcal{D}^{\mathrm{tr}}|} \!\sum_{(\va_i, b_i) \in \mathcal{D}^{\mathrm{tr}}} \sigma(x_i) \ell(\langle \va_i, \fy\rangle, b_i)\! +\! c\|\fy\|^2.
\end{split}
\end{equation*}}
\!\!In the above, $\mathcal{D}^{\operatorname{tr}}$ denotes the noisy training set and $\mathcal{D}^{\mathrm{val}}$ denotes the validation set. $\left(\va_i, b_i\right)$ denotes the $i$-th sample in the dataset, where $\va_i$ represents the feature and $b_i$ represents its corresponding label. 
We denote $\sigma(\cdot)$ as a clipping function that maps a scalar to the interval $[0,1]$ and $\ell(\cdot, \cdot)$ is the loss of cross entropy.
We set $c=10^{-3}$.

We compare the FSBA method (Algorithm~\ref{alg:FSBA}) and its lazy variant (Algorithm~\ref{alg:LFSBA}) with baseline methods, including ITD~\cite{ji2021bilevel}, AID with conjugate gradient~\cite{maclaurin2015gradient}, and near optimal fully first-order methods F$^2$BA~\cite{Chen23Nearoptimal} on ``breast-cancer" and ``australian" datasets~\cite{chang2011libsvm}.
We report the results on $\fD_{tr}$ with different noise rates $p  = 25\%$ and $p = 50\%$ (the ratio of training samples with disrupted labels) in Figure~\ref{fig:data hyper}, which demonstrates that our LFSBA and FSBA methods converge faster than the baselines. We defer the hyperparameter tuning details for this experiment to the appendix~\ref{app:expdatacleaning}.

\subsection{Hyperparameter Tuning}
We validate the proposed methods on \textit{hyperparameter tuning} task, which aims to find the optimal hyperparameter that minimizes the loss on the validation dataset.
The \textit{hyperparameter tuning} task can be reformulated as a bilevel optimization problem with the upper and lower-level objectives: 
\begin{align*}
\begin{split}
&f(\fx,\fy) := \! \frac{1}{|\mathcal{D}_{\text{val}}|} \sum_{(\boldsymbol{a}_i,\boldsymbol{b}_i)\in \mathcal{D}_{\text{val}}} L(\boldsymbol{y}^*(\fx); \boldsymbol{a}_i, \boldsymbol{b}_i),\\
&g(\fx,\fy) :=\!\!\frac{1}{|\mathcal{D}_{\text{tr}}|}\!\!\!\! \sum_{(\boldsymbol{a}_i,\boldsymbol{b}_i)\in \mathcal{D}_{\text{tr}}} \!\!\!\!\!\!\!\!L(\boldsymbol{w}; \boldsymbol{a}_i, \boldsymbol{b}_i) \!+\! \frac{1}{2cp} \sum_{j=1}^c \sum_{k=1}^p \!\exp(x_k) y_{jk}^2, 
\end{split}
\end{align*}
where $\fx = [x_1,\cdots,x_k,\cdots,x_p]^{\top}\in\RB^{p}$, $\y\in\RB^{c\times p}$, $\mathcal{D}_{\text{tr}} = \{(\boldsymbol{a}_i, \boldsymbol{b}_i)\}$ is the training dataset, $\mathcal{D}_{\text{val}} = \{(\boldsymbol{x}_i, \boldsymbol{y}_i)\}$ is the validation dataset, $L(\cdot;\cdot,\cdot)$ is the cross-entropy loss.

We compare the performance of the inexact variant of FSBA (IFSBA, Algorithm~\ref{alg:IFSBA}) with the baseline algorithms over a logistic regression problem on 20 News group dataset \cite{grazzi2020iteration} ($c=20, p=130170$).
We divide the datasets into three parts: $5657$ for training, $5657$ for validation, and $7532$ for testing. 

We use the same hyperparameter tuning protocol as in the data cleaning experiments.
The results are presented in Figure~\ref{fig:hyper_optimization} and we observe that IFSBA converges faster than the baselines.



\begin{figure*}[t]
    \centering
    \setlength{\tabcolsep}{6pt}
    \begin{tabular}{cc}
        \includegraphics[width=0.38\textwidth]{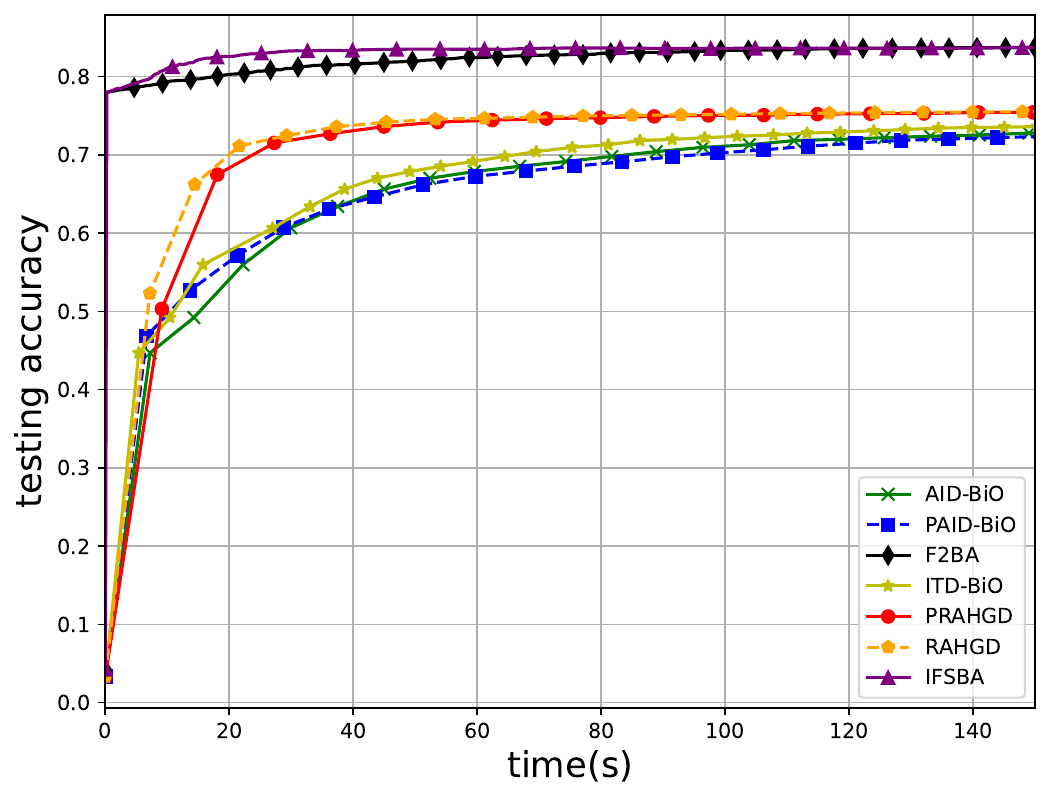} &
        \includegraphics[width=0.38\textwidth]{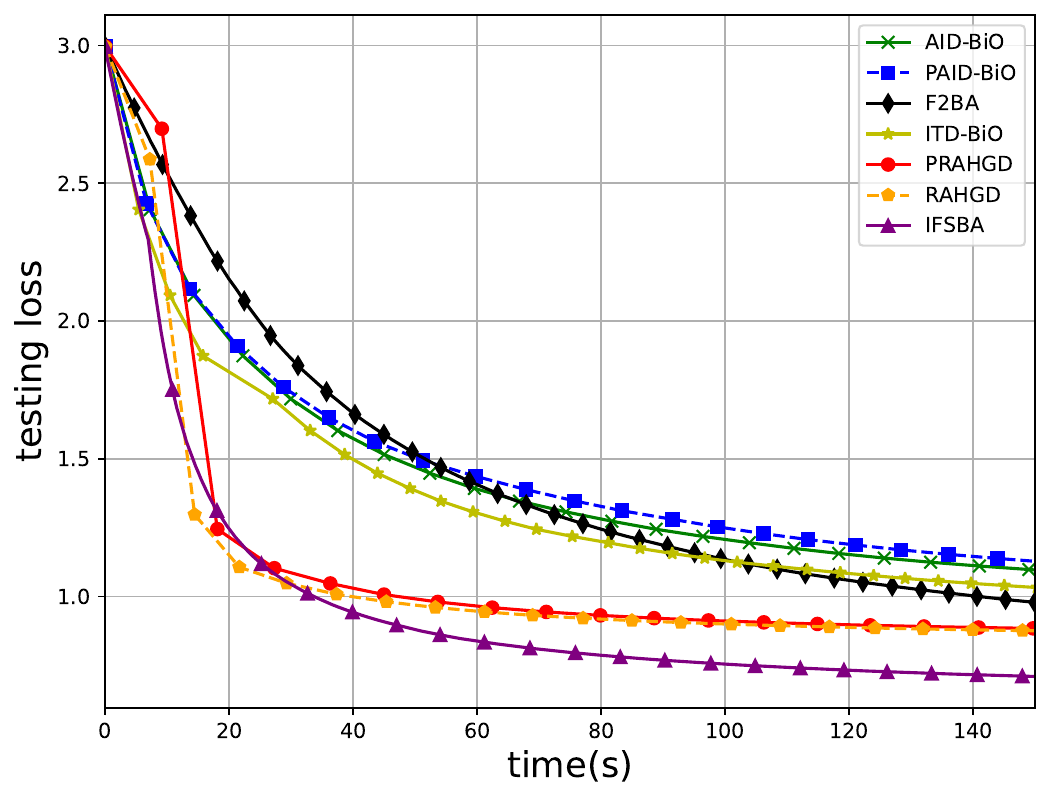} \\
        \scriptsize Testing accuracy vs. running time &
        \scriptsize Testing loss vs. running time
    \end{tabular}
    \caption{Comparison of various bilevel algorithms on logistic regression on the 20 Newsgroups dataset.}
    \label{fig:hyper_optimization}
\end{figure*}

\subsection{Few-Shot Meta-Learning}
\begin{figure*}[t]
    \centering
    \setlength{\tabcolsep}{6pt}
    \begin{tabular}{cc}
        \includegraphics[width=0.38\textwidth]{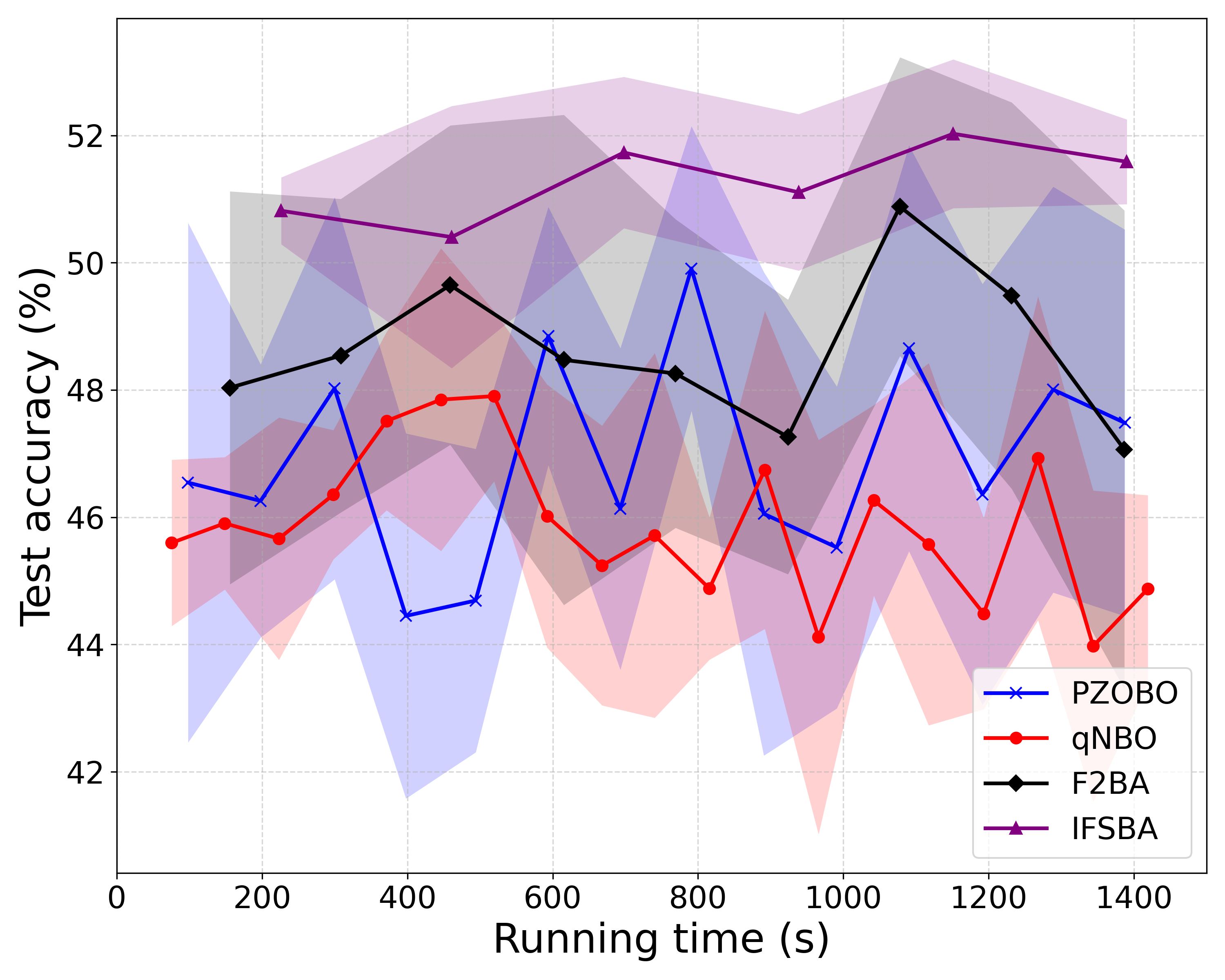} &
        \includegraphics[width=0.38\textwidth]{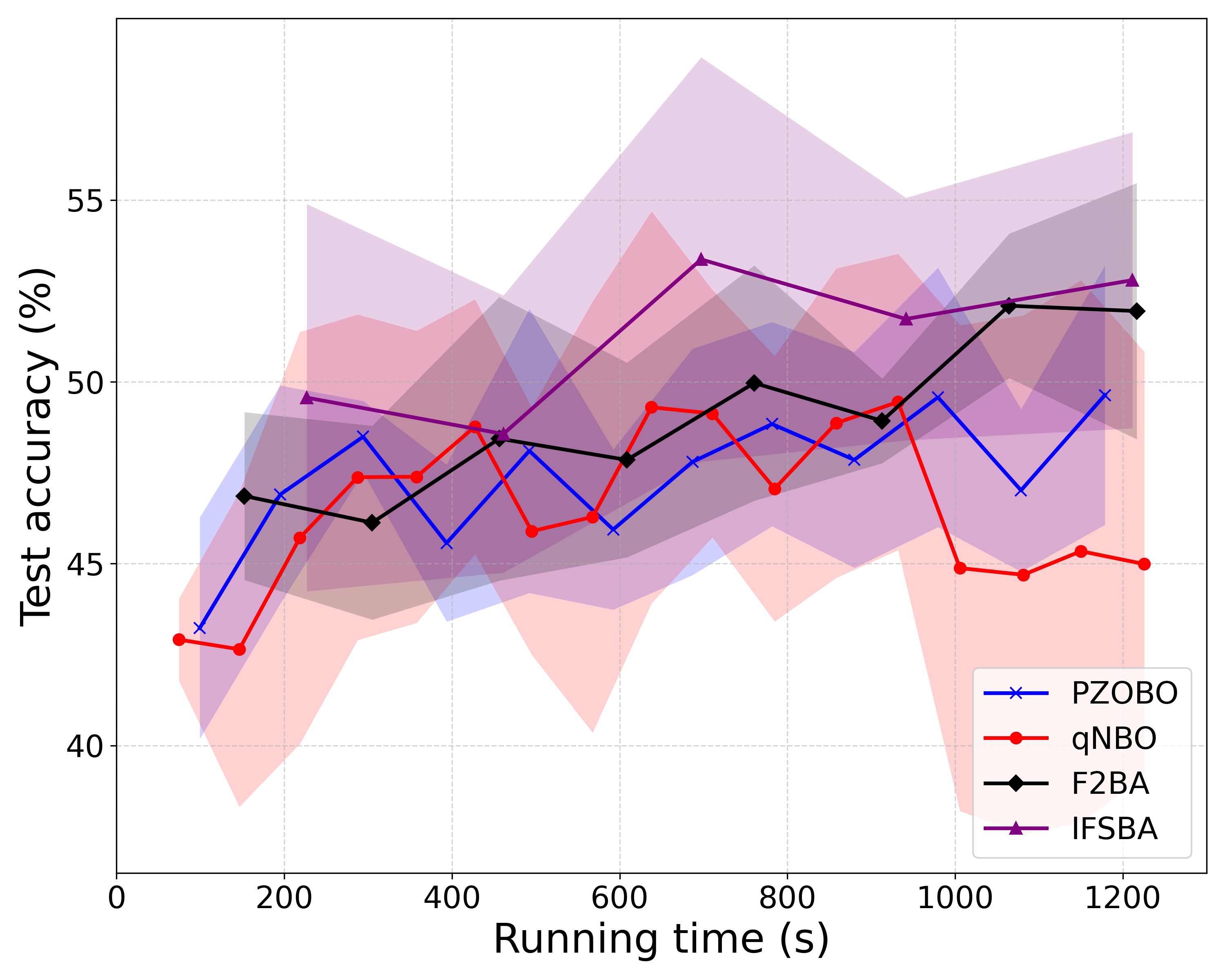} \\
        \scriptsize (a) FC100 &
        \scriptsize (b) miniImageNet
    \end{tabular}
    \caption{Test accuracy versus running time of different bilevel algorithms in 5-way 5-shot few-shot meta-learning experiments on FC100 and miniImageNet.}
    \label{fig:meta_learning_5way5shot}
\end{figure*}

We then conduct experiments to validate the efficiency of the proposed methods on \textit{few-shot meta-learning} task~\cite{finn2017model,raghu2019rapid,ji2021bilevel,fang2025qnbo}.
We consider $m$ few-shot tasks $\{\mathcal{T}_i\}_{i=1}^m$ sampled from a task
distribution $\mathcal{P}_{\mathcal{T}}$, where each task has a support set
$\mathcal{S}_i$ and a query set $\mathcal{D}_i$. We use a four-layer CNN with shared parameters $\fx$ as the feature extractor, and use $\fy_i$ as the last-layer linear classifier for task $\mathcal{T}_i$. The
meta-learning objective can be written as
\begin{align*}
\min_\fx \quad & \frac{1}{m}\sum_{i=1}^m \mathcal{L}_{\mathcal{D}_i}(\fx,\fy_i^*(\fx)) \\
\text{s.t.}\quad & \fy_i^*(\fx)\in \arg\min_{\fy_i}\mathcal{L}_{\mathcal{S}_i}(\fx,\fy_i).
\end{align*}
Here, $\mathcal{L}_{\mathcal{D}_i}(\fx,\fy_i)
= \frac{1}{|\mathcal{D}_i|}\sum_{\xi\in\mathcal{D}_i}\ell(\fx,\fy_i;\xi)$ is the query
loss, and $\mathcal{L}_{\mathcal{S}_i}(\fx,\fy_i)
=  \frac{1}{|\mathcal{S}_i|}\sum_{\xi\in\mathcal{S}_i}
(\ell(\fx,\fy_i;\xi)+\mathcal{R}(\fy_i))$ is the support loss. In our experiments,
$\ell$ is the cross-entropy loss and $\mathcal{R}$ is an $\ell_2$ regularizer.

Since prior work has shown that PZOBO~\cite{sow2022convergence} outperforms standard baselines such as MAML~\cite{finn2017model} and ANIL~\cite{raghu2019rapid}, we follow the same evaluation protocol and compare only against the stronger baselines PZOBO and qNBO~\cite{fang2025qnbo}. Under this setting, we evaluate F$^2$BA and IFSBA (Algorithm~\ref{alg:IFSBA}), in 5-way 5-shot experiments on miniImageNet~\cite{vinyals2016matching} and FC100~\cite{oreshkin2018tadam}. Results are averaged over five runs. All algorithms start from the same initialization with 20\% test accuracy, and the first data point is omitted for clarity.

For PZOBO and qNBO, we follow the hyperparameter settings used in their respective original implementations~\cite{sow2022convergence, fang2025qnbo}. For F$^2$BA and IFSBA, we tune the inner- and outer-loop learning rates over $\{10^{-3}, 10^{-2}, 10^{-1}, 1, 10, 10^2, 10^3\}$, the number of GD or AGD iterations over $\{5, 10, 30, 50\}$, and the penalty multiplier $\lambda$ over $\{1, 10, 10^2, 10^3\}$.
For IFSBA, we tune $M$ from $\left\{1,10^1, 10^2, 10^3\right\}$, the number of Cubic-Solver iterations and the order of Matrix Chebyshev Polynomials from $\{1,5,10,100\}$.
The results are presented in Figure~\ref{fig:meta_learning_5way5shot}, where IFSBA achieves higher test accuracy than the baselines within the same running time.

\section{Conclusion}
\label{sec:conclu}
In this paper, we have proposed several fully second-order methods for  nonconvex strongly-convex bilevel optimization.
The FSBA method takes $\tilde{\OM}(\epsilon^{-1.5})$ second-order oracle complexity to find the $(\epsilon,\OM(\sqrt{\epsilon}))$ SOSP of the hyper objective $\varphi(\cdot)$, and it is faster than the existing first- and second-order methods, 
showing the advantage of using second-order oracles in bilevel optimization.
The LFSBA method applies the lazy Hessian strategy and reduces the computational complexity of FSBA. 

\section*{Impact Statement}
This paper presents work whose goal is to advance the field of machine learning. There are many potential societal consequences of our work, none of which we feel must be specifically highlighted here.

\bibliography{references}

@article{Chen23Nearoptimal,
  title={Near-optimal fully first-order algorithms for finding stationary points in bilevel optimization},
  author={Chen, Lesi and Ma, Yaohua and Zhang, Jingzhao},
  journal={arXiv preprint arXiv:2306.14853},
  year={2023}
}

@inproceedings{chen2024finding,
  title={On finding small hyper-gradients in bilevel optimization: Hardness results and improved analysis},
  author={Chen, Lesi and Xu, Jing and Zhang, Jingzhao},
  booktitle={Conference on Learning Theory (COLT)},
  year={2024}
}

@article{doikov2023first,
  title={First and zeroth-order implementations of the regularized Newton method with lazy approximated Hessians},
  author={Doikov, Nikita and Grapiglia, Geovani Nunes},
  journal={arXiv preprint arXiv:2309.02412},
  year={2023}
}

@inproceedings{doikov2023second,
  title={Second-order optimization with lazy hessians},
  author={Doikov, Nikita and Jaggi, Martin and others},
  booktitle={International Conference on Machine Learning},
  pages={8138--8161},
  year={2023},
  organization={PMLR}
}

@article{ghadimi2018approximation,
  title={Approximation methods for bilevel programming},
  author={Ghadimi, Saeed and Wang, Mengdi},
  journal={arXiv preprint arXiv:1802.02246},
  year={2018}
}

@inproceedings{Kown23Fully,
  title={A fully first-order method for stochastic bilevel optimization},
  author={Kwon, Jeongyeol and Kwon, Dohyun and Wright, Stephen and Nowak, Robert D},
  booktitle={International Conference on Machine Learning},
  pages={18083--18113},
  year={2023},
  organization={PMLR}
}

@article{Luo22finding,
  title={Finding second-order stationary points in nonconvex-strongly-concave minimax optimization},
  author={Luo, Luo and Li, Yujun and Chen, Cheng},
  journal={Advances in Neural Information Processing Systems},
  volume={35},
  pages={36667--36679},
  year={2022}
}

@article{wang2020improved,
  title={Improved algorithms for convex-concave minimax optimization},
  author={Wang, Yuanhao and Li, Jian},
  journal={Advances in Neural Information Processing Systems},
  volume={33},
  pages={4800--4810},
  year={2020}
}

@inproceedings{ji2021bilevel,
  title={Bilevel optimization: Convergence analysis and enhanced design},
  author={Ji, Kaiyi and Yang, Junjie and Liang, Yingbin},
  booktitle={International conference on machine learning},
  pages={4882--4892},
  year={2021},
  organization={PMLR}
}

@inproceedings{liao2018reviving,
  title={Reviving and improving recurrent back-propagation},
  author={Liao, Renjie and Xiong, Yuwen and Fetaya, Ethan and Zhang, Lisa and Yoon, KiJung and Pitkow, Xaq and Urtasun, Raquel and Zemel, Richard},
  booktitle={International Conference on Machine Learning},
  pages={3082--3091},
  year={2018},
  organization={PMLR}
}

@article{liu2022bome,
  title={Bome! bilevel optimization made easy: A simple first-order approach},
  author={Liu, Bo and Ye, Mao and Wright, Stephen and Stone, Peter and Liu, Qiang},
  journal={Advances in neural information processing systems},
  volume={35},
  pages={17248--17262},
  year={2022}
}

@article{sow2022convergence,
  title={On the convergence theory for hessian-free bilevel algorithms},
  author={Sow, Daouda and Ji, Kaiyi and Liang, Yingbin},
  journal={Advances in Neural Information Processing Systems},
  volume={35},
  pages={4136--4149},
  year={2022}
}

@inproceedings{franceschi2018bilevel,
  title={Bilevel programming for hyperparameter optimization and meta-learning},
  author={Franceschi, Luca and Frasconi, Paolo and Salzo, Saverio and Grazzi, Riccardo and Pontil, Massimiliano},
  booktitle={International conference on machine learning},
  pages={1568--1577},
  year={2018},
  organization={PMLR}
}

@inproceedings{pedregosa2016hyperparameter,
  title={Hyperparameter optimization with approximate gradient},
  author={Pedregosa, Fabian},
  booktitle={International conference on machine learning},
  pages={737--746},
  year={2016},
  organization={PMLR}
}

@article{rajeswaran2019meta,
  title={Meta-learning with implicit gradients},
  author={Rajeswaran, Aravind and Finn, Chelsea and Kakade, Sham M and Levine, Sergey},
  journal={Advances in neural information processing systems},
  volume={32},
  year={2019}
}

@article{ji2022theoretical,
  title={Theoretical convergence of multi-step model-agnostic meta-learning},
  author={Ji, Kaiyi and Yang, Junjie and Liang, Yingbin},
  journal={Journal of machine learning research},
  volume={23},
  number={29},
  pages={1--41},
  year={2022}
}

@article{konda1999actor,
  title={Actor-critic algorithms},
  author={Konda, Vijay and Tsitsiklis, John},
  journal={Advances in neural information processing systems},
  volume={12},
  year={1999}
}

@article{hong2023two,
  title={A two-timescale stochastic algorithm framework for bilevel optimization: Complexity analysis and application to actor-critic},
  author={Hong, Mingyi and Wai, Hoi-To and Wang, Zhaoran and Yang, Zhuoran},
  journal={SIAM Journal on Optimization},
  volume={33},
  number={1},
  pages={147--180},
  year={2023},
  publisher={SIAM}
}

@inproceedings{liu2018darts,
  title={DARTS: Differentiable Architecture Search},
  author={Liu, Hanxiao and Simonyan, Karen and Yang, Yiming},
  booktitle={International Conference on Learning Representations},
  year={2018}
}

@article{goodfellow2020generative,
  title={Generative adversarial networks},
  author={Goodfellow, Ian and Pouget-Abadie, Jean and Mirza, Mehdi and Xu, Bing and Warde-Farley, David and Ozair, Sherjil and Courville, Aaron and Bengio, Yoshua},
  journal={Communications of the ACM},
  volume={63},
  number={11},
  pages={139--144},
  year={2020},
  publisher={ACM New York, NY, USA}
}

@inproceedings{zhang2021idarts,
  title={idarts: Differentiable architecture search with stochastic implicit gradients},
  author={Zhang, Miao and Su, Steven W and Pan, Shirui and Chang, Xiaojun and Abbasnejad, Ehsan M and Haffari, Reza},
  booktitle={International Conference on Machine Learning},
  pages={12557--12566},
  year={2021},
  organization={PMLR}
}

@inproceedings{zoph2016neural,
  title={Neural Architecture Search with Reinforcement Learning},
  author={Zoph, Barret and Le, Quoc},
  booktitle={International Conference on Learning Representations},
  year={2016}
}

@article{yang2023accelerating,
  title={Accelerating inexact hypergradient descent for bilevel optimization},
  author={Yang, Haikuo and Luo, Luo and Li, Chris Junchi and Jordan, Michael I},
  journal={arXiv preprint arXiv:2307.00126},
  year={2023}
}

@article{nesterov2006cubic,
  title={Cubic regularization of Newton method and its global performance},
  author={Nesterov, Yurii and Polyak, Boris T},
  journal={Mathematical programming},
  volume={108},
  number={1},
  pages={177--205},
  year={2006},
  publisher={Springer}
}

@article{carmon2020lower,
  title={Lower bounds for finding stationary points I},
  author={Carmon, Yair and Duchi, John C and Hinder, Oliver and Sidford, Aaron},
  journal={Mathematical Programming},
  volume={184},
  number={1},
  pages={71--120},
  year={2020},
  publisher={Springer}
}

@inproceedings{bruckner2011stackelberg,
  title={Stackelberg games for adversarial prediction problems},
  author={Br{\"u}ckner, Michael and Scheffer, Tobias},
  booktitle={Proceedings of the 17th ACM SIGKDD international conference on Knowledge discovery and data mining},
  pages={547--555},
  year={2011}
}

@inproceedings{zhang2022revisiting,
  title={Revisiting and advancing fast adversarial training through the lens of bi-level optimization},
  author={Zhang, Yihua and Zhang, Guanhua and Khanduri, Prashant and Hong, Mingyi and Chang, Shiyu and Liu, Sijia},
  booktitle={International Conference on Machine Learning},
  pages={26693--26712},
  year={2022},
  organization={PMLR}
}

@article{robey2023adversarial,
  title={Adversarial training should be cast as a non-zero-sum game},
  author={Robey, Alexander and Latorre, Fabian and Pappas, George J and Hassani, Hamed and Cevher, Volkan},
  journal={arXiv preprint arXiv:2306.11035},
  year={2023}
}

@inproceedings{lorraine2020optimizing,
  title={Optimizing millions of hyperparameters by implicit differentiation},
  author={Lorraine, Jonathan and Vicol, Paul and Duvenaud, David},
  booktitle={International conference on artificial intelligence and statistics},
  pages={1540--1552},
  year={2020},
  organization={PMLR}
}

@inproceedings{domke2012generic,
  title={Generic methods for optimization-based modeling},
  author={Domke, Justin},
  booktitle={Artificial Intelligence and Statistics},
  pages={318--326},
  year={2012},
  organization={PMLR}
}

@inproceedings{franceschi2017forward,
  title={Forward and reverse gradient-based hyperparameter optimization},
  author={Franceschi, Luca and Donini, Michele and Frasconi, Paolo and Pontil, Massimiliano},
  booktitle={International Conference on Machine Learning},
  pages={1165--1173},
  year={2017},
  organization={PMLR}
}

@article{bolte2021nonsmooth,
  title={Nonsmooth implicit differentiation for machine-learning and optimization},
  author={Bolte, J{\'e}r{\^o}me and Le, Tam and Pauwels, Edouard and Silveti-Falls, Tony},
  journal={Advances in neural information processing systems},
  volume={34},
  pages={13537--13549},
  year={2021}
}

@inproceedings{arbel2022amortized,
  title={Amortized implicit differentiation for stochastic bilevel optimization},
  author={Arbel, Michael and Mairal, Julien},
  booktitle={The Tenth International Conference on Learning Representations},
  year={2022}
}

@inproceedings{grazzi2020iteration,
  title={On the iteration complexity of hypergradient computation},
  author={Grazzi, Riccardo and Franceschi, Luca and Pontil, Massimiliano and Salzo, Saverio},
  booktitle={International Conference on Machine Learning},
  pages={3748--3758},
  year={2020},
  organization={PMLR}
}

@inproceedings{lin2020gradient,
  title={On gradient descent ascent for nonconvex-concave minimax problems},
  author={Lin, Tianyi and Jin, Chi and Jordan, Michael},
  booktitle={International Conference on Machine Learning},
  pages={6083--6093},
  year={2020},
  organization={PMLR}
}

@article{tripuraneni2018stochastic,
  title={Stochastic cubic regularization for fast nonconvex optimization},
  author={Tripuraneni, Nilesh and Stern, Mitchell and Jin, Chi and Regier, Jeffrey and Jordan, Michael I},
  journal={Advances in neural information processing systems},
  volume={31},
  year={2018}
}

@inproceedings{shaban2019truncated,
  title={Truncated back-propagation for bilevel optimization},
  author={Shaban, Amirreza and Cheng, Ching-An and Hatch, Nathan and Boots, Byron},
  booktitle={The 22nd International Conference on Artificial Intelligence and Statistics},
  pages={1723--1732},
  year={2019},
  organization={PMLR}
}

@inproceedings{zhou2022model,
  title={Model agnostic sample reweighting for out-of-distribution learning},
  author={Zhou, Xiao and Lin, Yong and Pi, Renjie and Zhang, Weizhong and Xu, Renzhe and Cui, Peng and Zhang, Tong},
  booktitle={International Conference on Machine Learning},
  pages={27203--27221},
  year={2022},
  organization={PMLR}
}

@inproceedings{maclaurin2015gradient,
  title={Gradient-based hyperparameter optimization through reversible learning},
  author={Maclaurin, Dougal and Duvenaud, David and Adams, Ryan},
  booktitle={International conference on machine learning},
  pages={2113--2122},
  year={2015},
  organization={PMLR}
}

@article{chang2011libsvm,
  title={LIBSVM: a library for support vector machines},
  author={Chang, Chih-Chung and Lin, Chih-Jen},
  journal={ACM transactions on intelligent systems and technology (TIST)},
  volume={2},
  number={3},
  pages={1--27},
  year={2011},
  publisher={Acm New York, NY, USA}
}

@inproceedings{
yang2025lancbio,
title={LancBiO: Dynamic Lanczos-aided Bilevel Optimization via Krylov Subspace},
author={Yan Yang and Bin Gao and Ya-xiang Yuan},
booktitle={The Thirteenth International Conference on Learning Representations},
year={2025}
}

@inproceedings{fang2025qnbo,
  title={qNBO: quasi-Newton Meets Bilevel Optimization},
  author={Fang, Sheng and Liu, Yong-Jin and Yao, Wei and Yu, Chengming and Zhang, Jin},
  booktitle={International Conference on Learning Representations},
  year={2025}
}

@article{wang2024efficient,
  title={Efficient first order method for saddle point problems with higher order smoothness},
  author={Wang, Nuozhou and Zhang, Junyu and Zhang, Shuzhong},
  journal={SIAM Journal on Optimization},
  volume={34},
  number={4},
  pages={3342--3370},
  year={2024},
  publisher={SIAM}
}

@article{huo2023new,
  title={A New Simple Stochastic Gradient Descent Type Algorithm With Lower Computational Complexity for Bilevel Optimization},
  author={Huo, Haimei and Liu, Risheng and Su, Zhixun},
  journal={arXiv preprint arXiv:2306.11211},
  year={2023}
}

@article{khanduri2021near,
  title={A near-optimal algorithm for stochastic bilevel optimization via double-momentum},
  author={Khanduri, Prashant and Zeng, Siliang and Hong, Mingyi and Wai, Hoi-To and Wang, Zhaoran and Yang, Zhuoran},
  journal={Advances in neural information processing systems},
  volume={34},
  pages={30271--30283},
  year={2021}
}

@article{chu2025provably,
  title={A Provably Convergent Plug-and-Play Framework for Stochastic Bilevel Optimization},
  author={Chu, Tianshu and Xu, Dachuan and Yao, Wei and Yu, Chengming and Zhang, Jin},
  journal={arXiv preprint arXiv:2505.01258},
  year={2025}
}

@article{wang2024fully,
  title={Fully First-Order Methods for Decentralized Bilevel Optimization},
  author={Wang, Xiaoyu and Chen, Xuxing and Ma, Shiqian and Zhang, Tong},
  journal={arXiv preprint arXiv:2410.19319},
  year={2024}
}

@article{yang2021provably,
  title={Provably faster algorithms for bilevel optimization},
  author={Yang, Junjie and Ji, Kaiyi and Liang, Yingbin},
  journal={Advances in Neural Information Processing Systems},
  volume={34},
  pages={13670--13682},
  year={2021}
}

@inproceedings{scaman2017optimal,
  title={Optimal algorithms for smooth and strongly convex distributed optimization in networks},
  author={Scaman, Kevin and Bach, Francis and Bubeck, S{\'e}bastien and Lee, Yin Tat and Massouli{\'e}, Laurent},
  booktitle={international conference on machine learning},
  pages={3027--3036},
  year={2017},
  organization={PMLR}
}

@article{dong2025efficient,
  title={Efficient Curvature-Aware Hypergradient Approximation for Bilevel Optimization},
  author={Dong, Youran and Yang, Junfeng and Yao, Wei and Zhang, Jin},
  journal={arXiv preprint arXiv:2505.02101},
  year={2025}
}

@inproceedings{liu2020generic,
  title={A generic first-order algorithmic framework for bi-level programming beyond lower-level singleton},
  author={Liu, Risheng and Mu, Pan and Yuan, Xiaoming and Zeng, Shangzhi and Zhang, Jin},
  booktitle={International conference on machine learning},
  pages={6305--6315},
  year={2020},
  organization={PMLR}
}

@inproceedings{liu2021value,
  title={A value-function-based interior-point method for non-convex bi-level optimization},
  author={Liu, Risheng and Liu, Xuan and Yuan, Xiaoming and Zeng, Shangzhi and Zhang, Jin},
  booktitle={International conference on machine learning},
  pages={6882--6892},
  year={2021},
  organization={PMLR}
}

@article{liu2021towards,
  title={Towards gradient-based bilevel optimization with non-convex followers and beyond},
  author={Liu, Risheng and Liu, Yaohua and Zeng, Shangzhi and Zhang, Jin},
  journal={Advances in Neural Information Processing Systems},
  volume={34},
  pages={8662--8675},
  year={2021}
}

@inproceedings{shen2023penalty,
  title={On penalty-based bilevel gradient descent method},
  author={Shen, Han and Chen, Tianyi},
  booktitle={International Conference on Machine Learning},
  pages={30992--31015},
  year={2023},
  organization={PMLR}
}

@article{sow2022primal,
  title={A primal-dual approach to bilevel optimization with multiple inner minima},
  author={Sow, Daouda and Ji, Kaiyi and Guan, Ziwei and Liang, Yingbin},
  journal={arXiv preprint arXiv:2203.01123},
  year={2022}
}

@article{xiao2023generalized,
  title={A Generalized Alternating Method for Bilevel Learning under the Polyak-$\{$$\backslash$L$\}$ ojasiewicz Condition},
  author={Xiao, Quan and Lu, Songtao and Chen, Tianyi},
  journal={arXiv preprint arXiv:2306.02422},
  year={2023}
}

@article{lian2017can,
  title={Can decentralized algorithms outperform centralized algorithms? a case study for decentralized parallel stochastic gradient descent},
  author={Lian, Xiangru and Zhang, Ce and Zhang, Huan and Hsieh, Cho-Jui and Zhang, Wei and Liu, Ji},
  journal={Advances in neural information processing systems},
  volume={30},
  year={2017}
}

@inproceedings{mishchenko2022proxskip,
  title={Proxskip: Yes! local gradient steps provably lead to communication acceleration! finally!},
  author={Mishchenko, Konstantin and Malinovsky, Grigory and Stich, Sebastian and Richt{\'a}rik, Peter},
  booktitle={International Conference on Machine Learning},
  pages={15750--15769},
  year={2022},
  organization={PMLR}
}

@article{wang2022fine,
  title={Fine-tuning language models over slow networks using activation quantization with guarantees},
  author={Wang, Jue and Yuan, Binhang and Rimanic, Luka and He, Yongjun and Dao, Tri and Chen, Beidi and R{\'e}, Christopher and Zhang, Ce},
  journal={Advances in Neural Information Processing Systems},
  volume={35},
  pages={19215--19230},
  year={2022}
}

@inproceedings{ye2018communication,
  title={Communication-computation efficient gradient coding},
  author={Ye, Min and Abbe, Emmanuel},
  booktitle={International Conference on Machine Learning},
  pages={5610--5619},
  year={2018},
  organization={PMLR}
}

@article{yuan2022decentralized,
  title={Decentralized training of foundation models in heterogeneous environments},
  author={Yuan, Binhang and He, Yongjun and Davis, Jared and Zhang, Tianyi and Dao, Tri and Chen, Beidi and Liang, Percy S and Re, Christopher and Zhang, Ce},
  journal={Advances in Neural Information Processing Systems},
  volume={35},
  pages={25464--25477},
  year={2022}
}

@article{huang2025efficiently,
  title={Efficiently escaping saddle points in bilevel optimization},
  author={Huang, Minhui and Chen, Xuxing and Ji, Kaiyi and Ma, Shiqian and Lai, Lifeng},
  journal={Journal of Machine Learning Research},
  volume={26},
  number={1},
  pages={1--61},
  year={2025}
}

@article{kwon2023penalty,
  title={On penalty methods for nonconvex bilevel optimization and first-order stochastic approximation},
  author={Kwon, Jeongyeol and Kwon, Dohyun and Wright, Stephen and Nowak, Robert},
  journal={arXiv preprint arXiv:2309.01753},
  year={2023}
}

@article{xiao2023alternating,
  title={An alternating optimization method for bilevel problems under the Polyak-{\L}ojasiewicz condition},
  author={Xiao, Quan and Lu, Songtao and Chen, Tianyi},
  journal={Advances in Neural Information Processing Systems},
  volume={36},
  pages={63847--63873},
  year={2023}
}

@article{lecun2002gradient,
  title={Gradient-based learning applied to document recognition},
  author={LeCun, Yann and Bottou, L{\'e}on and Bengio, Yoshua and Haffner, Patrick},
  journal={Proceedings of the IEEE},
  volume={86},
  number={11},
  pages={2278--2324},
  year={2002},
  publisher={Ieee}
}

@article{chen2025faster,
  title={Faster Gradient Methods for Highly-smooth Stochastic Bilevel Optimization},
  author={Chen, Lesi and Li, Junru and Zhang, Jingzhao},
  journal={arXiv preprint arXiv:2509.02937},
  year={2025}
}

@book{axelsson1996iterative,
  title={Iterative solution methods},
  author={Axelsson, Owe},
  year={1996},
  publisher={Cambridge university press}
}

@inproceedings{chen2022single,
  title={A single-timescale method for stochastic bilevel optimization},
  author={Chen, Tianyi and Sun, Yuejiao and Xiao, Quan and Yin, Wotao},
  booktitle={International Conference on Artificial Intelligence and Statistics},
  pages={2466--2488},
  year={2022},
  organization={PMLR}
}

@article{huang2024optimal,
  title={Optimal Hessian/Jacobian-free nonconvex-PL bilevel optimization},
  author={Huang, Feihu},
  journal={arXiv preprint arXiv:2407.17823},
  year={2024}
}

@inproceedings{li2022fully,
  title={A fully single loop algorithm for bilevel optimization without hessian inverse},
  author={Li, Junyi and Gu, Bin and Huang, Heng},
  booktitle={Proceedings of the AAAI Conference on Artificial Intelligence},
  volume={36},
  pages={7426--7434},
  year={2022}
}

@inproceedings{kovalev2022first,
  title={The first optimal acceleration of high-order methods in smooth convex optimization},
  author={Kovalev, Dmitry and Gasnikov, Alexander},
  booktitle={NeurIPS},
  year={2022}
}

@inproceedings{carmon2022optimal,
  title={Optimal and adaptive {M}onteiro-{S}vaiter acceleration},
  author={Carmon, Yair and Hausler, Danielle and Jambulapati, Arun and Jin, Yujia and Sidford, Aaron},
  booktitle={NeurIPS},
  year={2022}
}

@article{kornowski2020high,
  title={High-order oracle complexity of smooth and strongly convex optimization},
  author={Kornowski, Guy and Shamir, Ohad},
  journal={arXiv preprint arXiv:2010.06642},
  year={2020}
}

@article{yao2024constrained,
  title={Constrained bi-level optimization: Proximal lagrangian value function approach and hessian-free algorithm},
  author={Yao, Wei and Yu, Chengming and Zeng, Shangzhi and Zhang, Jin},
  journal={arXiv preprint arXiv:2401.16164},
  year={2024}
}

@article{lu2023slm,
  title={Slm: A smoothed first-order lagrangian method for structured constrained nonconvex optimization},
  author={Lu, Songtao},
  journal={Advances in Neural Information Processing Systems},
  volume={36},
  pages={80414--80454},
  year={2023}
}

@article{arbel2022non,
  title={Non-convex bilevel games with critical point selection maps},
  author={Arbel, Michael and Mairal, Julien},
  journal={Advances in Neural Information Processing Systems},
  volume={35},
  pages={8013--8026},
  year={2022}
}

@article{chen2024second,
  title={Second-order min-max optimization with lazy hessians},
  author={Chen, Lesi and Liu, Chengchang and Zhang, Jingzhao},
  journal={arXiv preprint arXiv:2410.09568},
  year={2024}
}

@inproceedings{liu2025enhanced,
  title={An Enhanced Levenberg--Marquardt Method via Gram Reduction},
  author={Liu, Chengchang and Luo, Luo and Lui, John CS},
  booktitle={Proceedings of the AAAI Conference on Artificial Intelligence},
  volume={39},
  number={18},
  pages={18772--18779},
  year={2025}
}

@inproceedings{finn2017model,
  title={Model-agnostic meta-learning for fast adaptation of deep networks},
  author={Finn, Chelsea and Abbeel, Pieter and Levine, Sergey},
  booktitle={International conference on machine learning},
  pages={1126--1135},
  year={2017},
  organization={PMLR}
}

@article{raghu2019rapid,
  title={Rapid learning or feature reuse? towards understanding the effectiveness of maml},
  author={Raghu, Aniruddh and Raghu, Maithra and Bengio, Samy and Vinyals, Oriol},
  journal={arXiv preprint arXiv:1909.09157},
  year={2019}
}

@article{vinyals2016matching,
  title={Matching networks for one shot learning},
  author={Vinyals, Oriol and Blundell, Charles and Lillicrap, Timothy and Wierstra, Daan and others},
  journal={Advances in neural information processing systems},
  volume={29},
  year={2016}
}

@article{oreshkin2018tadam,
  title={Tadam: Task dependent adaptive metric for improved few-shot learning},
  author={Oreshkin, Boris and Rodr{\'\i}guez L{\'o}pez, Pau and Lacoste, Alexandre},
  journal={Advances in neural information processing systems},
  volume={31},
  year={2018}
}
\bibliographystyle{icml2026}

\newpage
\appendix
\onecolumn

\section{Future Directions}
\label{sec:limit}
We list some future directions of this paper in this section.
\begin{itemize}
\item We use AGD to solve the lower-level problems to obtain $\fy_t\approx \fy_{\lambda}^*(\fx_t)$ and $\fw_t\approx \fy(\fx_t)$ within $\tilde{\OM}(\sqrt{\kappa})$ iterations. 
It is possible to accelerate the lower solvers by using second-order methods~\cite{carmon2022optimal,kovalev2022first,kornowski2020high} to improve the complexity dependency on $\kappa$.
\item 
We consider the fully second-order methods for nonconvex strongly-convex bilevel optimization. It will be interesting to develop second-order methods for bilevel optimization without lower strong convexity \cite{chen2024finding,liu2020generic,liu2021value,liu2021towards,shen2023penalty,sow2022primal,xiao2023alternating,xiao2023generalized,kwon2023penalty,yao2024constrained,lu2023slm,arbel2022non} and demonstrate the superiority of second-order methods over the first-order methods under this setting. 
\item We consider the deterministic setting such that one can access the exact gradient and Hessian oracle of $f$ and $g$.
It is also important to design stochastic~\cite{Kown23Fully,huo2023new,khanduri2021near,chu2025provably,wang2024fully,yang2021provably,dong2025efficient} and distributed~\cite{wang2024fully,lian2017can,scaman2017optimal,mishchenko2022proxskip,wang2022fine,ye2018communication,yuan2022decentralized} variants of FSBA and LFSBA to further improve the practical performance.
\end{itemize}

\section{Useful Lemmas}

\begin{lemma}[Lemma 2, \citet{wang2020improved}]
\label{lem:AGD}
Running Algorithm \ref{alg:AGD} on $\ell_h$-smooth and $\mu_h$-strongly-convex
objective function $h(\cdot)$ with parameters $\eta=1 / \ell_h$ and $\theta=\frac{\sqrt{\kappa_h}-1}{\sqrt{\kappa_h}+1}$ produces the output $\boldsymbol{y}_K$ satisfying $\left\|\boldsymbol{y}_K-\boldsymbol{y}^*\right\|^2 \leq\left(\kappa_h+1\right)\left(1-\frac{1}{\sqrt{\kappa_h}}\right)^K\left\|\boldsymbol{y}_0-\boldsymbol{y}^*\right\|^2$,
where $\boldsymbol{y}^*=\arg \min _y h(\boldsymbol{y})$ and $\kappa_h=\ell_h / \mu_h$.
\end{lemma}

\begin{lemma}[Lemma 3.2, \citet{Kown23Fully}]
\label{lem:strong_convex_L} 
Under Assumption~\ref{assum:basic}, for $\lambda \geq 2 \ell / \mu$, $\mathcal{L}_\lambda(\boldsymbol{x}, \cdot)$ is $(\lambda \mu / 2)$-strongly convex.
\end{lemma}

It is clear that $\boldsymbol{y}^*(\boldsymbol{x})$ is $\ell/\mu$-Lipschitz. And we can also show a similar result for $\boldsymbol{y}_\lambda^*(\boldsymbol{x})$.

\begin{lemma}[Lemma B.6, \citet{Chen23Nearoptimal}]
\label{lem:gradient_smooth_y_star}  
Under Assumption~\ref{assum:basic}, for $\lambda \geq 2 \ell / \mu$, it holds that
 $\boldsymbol{y}_\lambda^*(\boldsymbol{x})$ is $\left(4 \ell / \mu\right)$-Lipschitz.
\end{lemma}

\begin{lemma}[{\citet{nesterov2006cubic}}]
\label{lem:property_Hessain_smooth} 
Suppose Assumption \ref{assum:basic} holds, according to Proposition~\ref{prop:strong_convex_L}, we have the following inequalities for the Hessian Lipschitz continuity:
\begin{align}
    &\left\|\nabla \mathcal{L}_\lambda^*(\boldsymbol{x}^{\prime})-\nabla \mathcal{L}_\lambda^*(\boldsymbol{x})-\nabla^2 \mathcal{L}_\lambda^*(\boldsymbol{x})(\boldsymbol{x}^{\prime}-\boldsymbol{x})\right\| \leq \frac{\bar{\rho}}{2}\|\boldsymbol{x}^{\prime}-\boldsymbol{x}\|^2,\\
\label{eq:Hessian_smmoth_2}
   &\left| \mathcal{L}_\lambda^*(\boldsymbol{x}^{\prime})-\mathcal{L}_\lambda^*(\boldsymbol{x})-\langle \nabla \mathcal{L}_\lambda^*(\boldsymbol{x}), \boldsymbol{x}^{\prime}-\boldsymbol{x}\rangle-\frac{1}{2}\left \langle \nabla^2 \mathcal{L}_\lambda^*(\boldsymbol{x}) (\boldsymbol{x}^{\prime}-\boldsymbol{x}), \boldsymbol{x}^{\prime}-\boldsymbol{x}\right\rangle\right| \leq \frac{\bar{\rho}}{6}\|\boldsymbol{x}^{\prime}-\boldsymbol{x}\|^3.
\end{align}
\end{lemma}

\begin{lemma}[{\citet{nesterov2006cubic}}]
\label{lem:stationary conditions}  
For any  $M^{\prime} \geq 0 $, we denote $\boldsymbol{g}$ is the gradient of the objective function and  $\boldsymbol{H}$ is the Hessian of the objective function, then the solution $\boldsymbol{s}^*$ of the following cubic regularized quadratic problem
\begin{equation*}
\boldsymbol{s}^*=\underset{\boldsymbol{x} \in \mathbb{R}^{d_x}}{\arg \min }\left(\g^{\top} \boldsymbol{s}+\frac{1}{2} \boldsymbol{s}^{\top} \H \boldsymbol{s}+\frac{M^{\prime}}{6}\|\boldsymbol{s}\|^3\right)    
\end{equation*}
satisfies
\begin{align}
\label{eq:stationary_condition_1}
\g + \H \boldsymbol{ s}^* + \frac{M^{\prime}}{2}\left\|\boldsymbol{s}^*\right\| \boldsymbol{s}^* &= \boldsymbol{0}, \\
\label{eq:stationary_condition_2}
\H + \frac{M^{\prime}}{2}\left\|\boldsymbol{s}^*\right\| \I &\succeq \boldsymbol{0}, \\
\label{eq:stationary_condition_3}
\g^{\top} \boldsymbol{s}^* + \frac{1}{2}\left(\boldsymbol{s}^*\right)^{\top} \H  \boldsymbol{s}^* + \frac{M^{\prime}}{6}\left\|\boldsymbol{s}^*\right\|^3 &\leq -\frac{M^{\prime}}{12}\left\|\boldsymbol{s}^*\right\|^3.
\end{align}
\end{lemma}

\begin{lemma}[\cite{doikov2023second}, Lemma B.1]
\label{lem:sum_inequality}  
For any sequence of positive numbers $\left\{r_t\right\}_{t \geq 1}$, it holds for any $m \geq 1$:
\begin{equation}
\sum_{t=1}^{m-1}\left(\sum_{i=1}^t r_i\right)^3 \leq \frac{m^3}{3} \sum_{t=1}^{m-1} r_t^3.
\end{equation}
\end{lemma}

\section{The Proof of Section~\ref{sec:FSBA}}

\subsection{The Proof of Lemma~\ref{lem:estimators_error}}
\begin{proof}
We first need to derive an upper bound for the following equations:
\begin{align*}
&\left\|\nabla_{x y}^2 g\left(\boldsymbol{x}, \boldsymbol{y}^*(\boldsymbol{x})\right)\left[\nabla_{y y}^2 g\left(\boldsymbol{x}, \boldsymbol{y}^*(\boldsymbol{x})\right)\right]^{-1}-\nabla_{x y}^2 g\left(\boldsymbol{x}, \boldsymbol{w}\right)\left[\nabla_{y y}^2 g\left(\boldsymbol{x}, \boldsymbol{w}\right)\right]^{-1} \right\| \\
\text{and}~~~&\left\|\nabla_{x y}^2 \mathcal{L}_\lambda\left(\boldsymbol{x}, \boldsymbol{y}_\lambda^*(\boldsymbol{x})\right)\left[\nabla_{y y}^2 \mathcal{L}_\lambda\left(\boldsymbol{x}, \boldsymbol{y}_\lambda^*(\boldsymbol{x})\right)\right]^{-1}-\nabla_{x y}^2 \mathcal{L}_\lambda\left(\boldsymbol{x}, \boldsymbol{y}\right)\left[\nabla_{y y}^2 \mathcal{L}_\lambda\left(\boldsymbol{x}, \boldsymbol{y}\right)\right]^{-1} \right\|.
\end{align*}

Using the matrix identity $A^{-1}-B^{-1}=A^{-1}(B-A) B^{-1}$, we have
\begin{equation*}
\begin{aligned}
& \left\| \left[\nabla_{y y}^2 g\left(\boldsymbol{x}, \boldsymbol{y}^*(\boldsymbol{x})\right)\right]^{-1} -
\left[\nabla_{y y}^2 g\left(\boldsymbol{x}, \boldsymbol{w}\right)\right]^{-1}\right\|\\
\leq & \left\| \left[\nabla_{y y}^2 g\left(\boldsymbol{x}, \boldsymbol{y}^*(\boldsymbol{x})\right)\right]^{-1} \right\|
\left\| \nabla_{y y}^2 g\left(\boldsymbol{x}, \boldsymbol{w}\right)- \nabla_{y y}^2 g\left(\boldsymbol{x}, \boldsymbol{y}^*(\boldsymbol{x})\right)\right\|
\left\| \left[\nabla_{y y}^2 g\left(\boldsymbol{x}, \boldsymbol{w}\right)\right]^{-1} \right\| \\
\leq & \frac{\rho}{\mu^2} \left\|\boldsymbol{w}- \boldsymbol{y}^*(\boldsymbol{x}) \right\|,
\end{aligned}    
\end{equation*}
and we further have 
\begin{equation*}
\begin{aligned}
& \left\|\nabla_{x y}^2 g\left(\boldsymbol{x}, \boldsymbol{y}^*(\boldsymbol{x})\right)\left[\nabla_{y y}^2 g\left(\boldsymbol{x}, \boldsymbol{y}^*(\boldsymbol{x})\right)\right]^{-1}-\nabla_{x y}^2 g\left(\boldsymbol{x}, \boldsymbol{w}\right)\left[\nabla_{y y}^2 g\left(\boldsymbol{x}, \boldsymbol{w}\right)\right]^{-1} \right\|\\
\leq & \left\|\nabla_{x y}^2 g\left(\boldsymbol{x}, \boldsymbol{y}^*(\boldsymbol{x})\right)- \nabla_{x y}^2 g\left(\boldsymbol{x}, \boldsymbol{w}\right) \right\| 
\left\|\left[\nabla_{y y}^2 g\left(\boldsymbol{x}, \boldsymbol{y}^*(\boldsymbol{x})\right)\right]^{-1}\right\| \\
& +\left\| \nabla_{x y}^2 g\left(\boldsymbol{x}, \boldsymbol{w}\right) \right\|\left\| \left[\nabla_{y y}^2 g\left(\boldsymbol{x}, \boldsymbol{y}^*(\boldsymbol{x})\right)\right]^{-1} -
\left[\nabla_{y y}^2 g\left(\boldsymbol{x}, \boldsymbol{w}\right)\right]^{-1} \right\| \\
\leq & \left(\frac{\rho}{\mu}+\frac{\ell \rho}{\mu^2}\right)\left\|\boldsymbol{w}- \boldsymbol{y}^*(\boldsymbol{x}) \right\|.
\end{aligned}    
\end{equation*}

Similarly, using the matrix identity $A^{-1}-B^{-1}=A^{-1}(B-A) B^{-1}$, we have
\begin{equation*}
\begin{aligned}
& \left\|\left[\nabla_{y y}^2 \mathcal{L}_\lambda\left(\boldsymbol{x}, \boldsymbol{y}_\lambda^*(\boldsymbol{x})\right)\right]^{-1}-
\left[\nabla_{y y}^2 \mathcal{L}_\lambda\left(\boldsymbol{x}, \boldsymbol{y}\right)\right]^{-1} \right\|\\
\leq & \left\| \left[\nabla_{y y}^2 \mathcal{L}_\lambda\left(\boldsymbol{x}, \boldsymbol{y}_\lambda^*(\boldsymbol{x})\right)\right]^{-1} \right\|
\left\| \nabla_{y y}^2 \mathcal{L}_\lambda\left(\boldsymbol{x}, \boldsymbol{y}\right)- \nabla_{y y}^2 \mathcal{L}_\lambda\left(\boldsymbol{x}, \boldsymbol{y}_\lambda^*(\boldsymbol{x})\right)\right\|
\left\| \left[\nabla_{y y}^2 \mathcal{L}_\lambda\left(\boldsymbol{x}, \boldsymbol{y}\right)\right]^{-1} \right\| \\
\leq & \frac{4\left( \rho + \lambda \rho \right) }{\lambda^2 \mu^2} \left\|\boldsymbol{y}- \boldsymbol{y}_\lambda^*(\boldsymbol{x}) \right\|,
\end{aligned}   
\end{equation*}

and we further have 
\begin{equation*}
\begin{aligned}
& \left\|\nabla_{x y}^2 \mathcal{L}_\lambda\left(\boldsymbol{x}, \boldsymbol{y}_\lambda^*(\boldsymbol{x})\right)\left[\nabla_{y y}^2 \mathcal{L}_\lambda\left(\boldsymbol{x}, \boldsymbol{y}_\lambda^*(\boldsymbol{x})\right)\right]^{-1}-\nabla_{x y}^2 \mathcal{L}_\lambda\left(\boldsymbol{x}, \boldsymbol{y}\right)\left[\nabla_{y y}^2 \mathcal{L}_\lambda\left(\boldsymbol{x}, \boldsymbol{y}\right)\right]^{-1} \right\|\\
\leq & \left\|\nabla_{x y}^2 \mathcal{L}_\lambda\left(\boldsymbol{x}, \boldsymbol{y}_\lambda^*(\boldsymbol{x})\right)-\nabla_{x y}^2 \mathcal{L}_\lambda\left(\boldsymbol{x}, \boldsymbol{y}\right) \right\|
\left\|\left[\nabla_{y y}^2 \mathcal{L}_\lambda\left(\boldsymbol{x}, \boldsymbol{y}_\lambda^*(\boldsymbol{x})\right)\right]^{-1}\right\| \\
& +\left\| \nabla_{x y}^2 \mathcal{L}_\lambda\left(\boldsymbol{x}, \boldsymbol{y}\right) \right\|\left\| \left[\nabla_{y y}^2 \mathcal{L}_\lambda\left(\boldsymbol{x}, \boldsymbol{y}_\lambda^*(\boldsymbol{x})\right)\right]^{-1} -
\left[\nabla_{y y}^2 \mathcal{L}_\lambda\left(\boldsymbol{x}, \boldsymbol{y}\right)\right]^{-1}  \right\| \\
\leq & \left(\frac{2 \left( \rho + \lambda \rho\right)}{\lambda \mu}+\frac{4\left( \rho + \lambda \rho \right) \left(\ell +\lambda \ell\right) }{\lambda^2 \mu^2}\right)\left\|\boldsymbol{y}- \boldsymbol{y}_\lambda^*(\boldsymbol{x}) \right\|.
\end{aligned}    
\end{equation*}

According to 
\begin{equation*}
\nabla \mathcal{L}_\lambda^*(\boldsymbol{x}) =\nabla_x f\left(\boldsymbol{x}, \boldsymbol{y}_\lambda^*(\boldsymbol{x})\right)+\lambda\left(\nabla_x g\left(\boldsymbol{x}, \boldsymbol{y}_\lambda^*(\boldsymbol{x})\right)-\nabla_x g\left(\boldsymbol{x}, \boldsymbol{y}^*(\boldsymbol{x})\right)\right),   
\end{equation*}
and     
\begin{equation*}
\g(\fx;\fy,\fw)=\nabla_x f\left(\boldsymbol{x}, \boldsymbol{y}\right)+\lambda\left(\nabla_x g\left(\boldsymbol{x}, \boldsymbol{y}\right)-\nabla_x g\left(\boldsymbol{x}, \boldsymbol{w}\right)\right).  
\end{equation*}

Then we have
\begin{equation*}
\begin{aligned}
& \left\|{\nabla} \mathcal{L}_\lambda^*(\fx)-\g(\fx;\fy,\fw)\right\|  \\
\leq & \left\|\nabla_x f\left(\boldsymbol{x},\boldsymbol{y}\right)-\nabla_x f\left(\boldsymbol{x}, \boldsymbol{y}_\lambda^*(\boldsymbol{x})\right) \right\| + \lambda \left\|\nabla_x g\left(\boldsymbol{x}, \boldsymbol{y}\right)-\nabla_x g\left(\boldsymbol{x}, \boldsymbol{y}_\lambda^*(\boldsymbol{x})\right) \right\|\\
+ & \lambda \left\|\nabla_x g\left(\boldsymbol{x},\boldsymbol{w}\right)-\nabla_x g\left(\boldsymbol{x}, \boldsymbol{y}^*(\boldsymbol{x})\right)\right\| \\
\leq &  2 \lambda \ell \left\|\boldsymbol{y}-\fy_\lambda^*\left(\boldsymbol{x}\right)\right\|+\lambda \ell \left\|\boldsymbol{w}-\fy^*\left(\boldsymbol{x}\right)\right\|.
\end{aligned}    
\end{equation*}

Note that 
\begin{align*}
\begin{split}
   & \nabla^2 \mathcal{L}_\lambda^*(\boldsymbol{x})= \nabla_{x x}^2 f\left(\boldsymbol{x}, \fy_\lambda^*(\boldsymbol{x})\right)
    -\nabla_{x y}^2 \mathcal{L}_\lambda\left(\boldsymbol{x}, \fy_\lambda^*(\boldsymbol{x})\right)\left[\nabla_{y y}^2 \mathcal{L}_\lambda\left(\boldsymbol{x}, \fy_\lambda^*(\boldsymbol{x})\right)\right]^{-1}\nabla_{y x}^2 \mathcal{L}_\lambda\left(\boldsymbol{x}, \fy_\lambda^*(\boldsymbol{x})\right)\\
&       +\lambda\left(\nabla_{x x}^2 g\left(\boldsymbol{x}, \fy_\lambda^*(\boldsymbol{x})\right)\!-\!\nabla_{x x}^2 g\left(\boldsymbol{x}, \fy^*(\boldsymbol{x})\right) \!+\!\nabla_{x y}^2 g\left(\boldsymbol{x}, \fy^*(\boldsymbol{x})\right)\left[\nabla_{y y}^2 g\left(\boldsymbol{x}, \fy^*(\boldsymbol{x})\right)\right]^{-1}\nabla_{y x}^2 g\left(\boldsymbol{x}, \fy^*(\boldsymbol{x})\right)\right),
\end{split}
\end{align*}
and
\begin{align*}
\begin{split}
   &\H(\fx;\fy,\fw) : =  \nabla_{x x}^2 f\left(\boldsymbol{x}, \fy\right)-\nabla_{x y}^2 \mathcal{L}_\lambda\left(\boldsymbol{x}, \fy\right)\left[\nabla_{y y}^2 \mathcal{L}_\lambda\left(\boldsymbol{x}, \fy\right)\right]^{-1}\nabla_{y x}^2 \mathcal{L}_\lambda\left(\boldsymbol{x}, \fy\right)\\
   &~~~~
   +\lambda\left(\nabla_{x x}^2 g\left(\boldsymbol{x}, \fy\right)-\nabla_{x x}^2 g\left(\boldsymbol{x}, \fw\right) +\nabla_{x y}^2 g\left(\boldsymbol{x}, \fw\right)\left[\nabla_{y y}^2 g\left(\boldsymbol{x}, \fw\right)\right]^{-1}\nabla_{y x}^2 g\left(\boldsymbol{x}, \fw\right)\right), 
\end{split}
\end{align*}

We can obtain the following inequalities:
\begin{align*}
\left\| \nabla_{x x}^2 f\left(\boldsymbol{x}, \boldsymbol{y}_\lambda^*(\boldsymbol{x})\right) 
- \nabla_{x x}^2 f\left(\boldsymbol{x}, \boldsymbol{y}\right) \right\| 
&\leq \rho \left\|\boldsymbol{y} - \boldsymbol{y}_\lambda^*\left(\boldsymbol{x}\right)\right\|, \\
\lambda \left\|\nabla_{x x}^2 g\left(\boldsymbol{x}, \boldsymbol{y}_\lambda^*(\boldsymbol{x})\right)
- \nabla_{x x}^2 g\left(\boldsymbol{x}, \boldsymbol{y}\right) \right\| 
&\leq \lambda \rho \left\|\boldsymbol{y} - \boldsymbol{y}_\lambda^*\left(\boldsymbol{x}\right)\right\|, \\
\lambda \left\|\nabla_{x x}^2 g\left(\boldsymbol{x}, \fy^*(\boldsymbol{x})\right)
- \nabla_{x x}^2 g\left(\boldsymbol{x}, \boldsymbol{w}\right)\right\| 
&\leq \lambda \rho \left\|\boldsymbol{w} - \fy^*\left(\boldsymbol{x}\right) \right\|,
\end{align*}
and
{
\begin{equation*}
\begin{aligned}
& \left\|\nabla_{x y}^2 g\left(\boldsymbol{x}, \fy^*(\boldsymbol{x})\right)\left[\nabla_{y y}^2 g\left(\boldsymbol{x}, \fy^*(\boldsymbol{x})\right)\right]^{-1}\nabla_{y x}^2 g\left(\boldsymbol{x}, \fy^*(\boldsymbol{x})\right)-  \nabla_{x y}^2 g\left(\boldsymbol{x}, \fw\right)\left[\nabla_{y y}^2 g\left(\boldsymbol{x}, \fw\right)\right]^{-1}\nabla_{y x}^2 g\left(\boldsymbol{x}, \fw\right)\right\| \\
& \leq \left\| \nabla_{x y}^2 g\left(\boldsymbol{x}, \fy^*(\boldsymbol{x})\right)\left[\nabla_{y y}^2 g\left(\boldsymbol{x}, \fy^*(\boldsymbol{x})\right)\right]^{-1}\right\| \left\| \nabla_{y x}^2 g\left(\boldsymbol{x}, \fy^*(\boldsymbol{x})\right)- \nabla_{y x}^2g\left(\boldsymbol{x}, \boldsymbol{w}\right) \right\| \\
& ~~~~+ \left\| \nabla_{y x}^2g\left(\boldsymbol{x}, \boldsymbol{w}\right) \right\| \left\|\nabla_{x y}^2 g\left(\boldsymbol{x}, \fy^*(\boldsymbol{x})\right)\left[\nabla_{y y}^2 g\left(\boldsymbol{x}, \fy^*(\boldsymbol{x})\right)\right]^{-1}-\nabla_{x y}^2 g\left(\boldsymbol{x}, \boldsymbol{w}\right)\left[\nabla_{y y}^2 g\left(\boldsymbol{x}, \boldsymbol{w}\right)\right]^{-1} \right\| \\
& \leq\left(\frac{2 \ell \rho}{\mu}+ \frac{\ell^2 \rho}{\mu^2} \right) \left\|\boldsymbol{w}-\fy^*\left(\boldsymbol{x}\right) \right\|,
\end{aligned}    
\end{equation*}
}
and
{
\begin{equation*}
\begin{aligned}
& \left\|\nabla_{x y}^2 \mathcal{L}_\lambda\left(\boldsymbol{x}, \fy_\lambda^*(\boldsymbol{x})\right)\left[\nabla_{y y}^2 \mathcal{L}_\lambda\left(\boldsymbol{x}, \fy_\lambda^*(\boldsymbol{x})\right)\right]^{-1}\nabla_{y x}^2 \mathcal{L}_\lambda\left(\boldsymbol{x}, \fy_\lambda^*(\boldsymbol{x})\right)-\nabla_{x y}^2 \mathcal{L}_\lambda\left(\boldsymbol{x}, \fy\right)\left[\nabla_{y y}^2 \mathcal{L}_\lambda\left(\boldsymbol{x}, \fy\right)\right]^{-1}\nabla_{y x}^2 \mathcal{L}_\lambda\left(\boldsymbol{x}, \fy\right)\right\| \\
& \leq \left\| \nabla_{x y}^2 \mathcal{L}_\lambda\left(\boldsymbol{x}, \fy_\lambda^*(\boldsymbol{x})\right)\left[\nabla_{y y}^2 \mathcal{L}_\lambda\left(\boldsymbol{x}, \fy_\lambda^*(\boldsymbol{x})\right)\right]^{-1}\right\| \left\| \nabla_{y x}^2 \mathcal{L}_\lambda\left(\boldsymbol{x}, \fy_\lambda^*(\boldsymbol{x})\right)- \nabla_{y x}^2 \mathcal{L}_\lambda\left(\boldsymbol{x},\boldsymbol{y}\right) \right\| \\
&~~~~~~+ \left\| \nabla_{y x}^2 \mathcal{L}_\lambda\left(\boldsymbol{x},\boldsymbol{y}\right)\right\| \left\|\nabla_{x y}^2 \mathcal{L}_\lambda\left(\boldsymbol{x}, \fy_\lambda^*(\boldsymbol{x})\right)\left[\nabla_{y y}^2 \mathcal{L}_\lambda\left(\boldsymbol{x}, \fy_\lambda^*(\boldsymbol{x})\right)\right]^{-1}-\nabla_{x y}^2 \mathcal{L}_\lambda\left(\boldsymbol{x}, \boldsymbol{y}\right)\left[\nabla_{y y}^2 \mathcal{L}_\lambda\left(\boldsymbol{x}, \boldsymbol{y}\right)\right]^{-1} \right\|\\
& \leq\left[\frac{2 (\ell+\lambda \ell) \left( \rho+\lambda \rho\right)}{\lambda \mu} +
\left( \ell + \lambda \ell\right)
\left(\frac{2 \left( \rho + \lambda \rho\right)}{\lambda \mu}+\frac{4\left( \rho + \lambda \rho \right) \left(\ell +\lambda \ell
\right) }{\lambda^2 \mu^2}\right) \right]\left\|\boldsymbol{y}-\fy_\lambda^*(\boldsymbol{x}) \right\|.
\end{aligned}    
\end{equation*}
}

Combining the above inequations, we have 
\begin{equation*}
\left\|{\nabla}^2 \mathcal{L}_\lambda^*\left(\boldsymbol{x}\right)-\H(\fx;\fy,\fw)\right\| \leq C_1 \left\|\boldsymbol{w}-\fy^*\left(\boldsymbol{x}\right) \right\| + C_2
 \left\|\boldsymbol{y} -\fy_\lambda^*\left(\boldsymbol{x}\right)\right\|,    
\end{equation*}
 where $C_1= \lambda \rho+\frac{2 \ell \rho}{\mu}+ \frac{\ell^2 \rho}{\mu^2}$, $C_2=\rho+\lambda \rho +\left( \ell + \lambda \ell \right)
\left(\frac{4 \left( \rho + \lambda \rho\right)}{\lambda \mu}+\frac{4\left( \rho + \lambda \rho \right) \left(\ell +\lambda \ell \right) }{\lambda^2 \mu^2}\right)$.
\end{proof}

\subsection{The Proof of Lemma~\ref{lem:exact_enough}}
\begin{proof}
We first use induction to show that
\begin{equation}
\label{eq:error_y_w}
\left\|\boldsymbol{y}_t-\fy_\lambda^*\left(\boldsymbol{x}_t\right)\right\| \leq \tilde{\epsilon}, \left\|\boldsymbol{w}_t-\fy^*\left(\boldsymbol{x}_t\right)\right\| \leq \tilde{\epsilon}
\end{equation} 
holds for any $t \geq 0$. For $t=0$,
 Lemma~\ref{lem:AGD} directly implies
\begin{align*}
& \left\|\boldsymbol{y}_0-\fy_\lambda^*\left(\boldsymbol{x}_0\right)\right\| \\ 
\leq & \sqrt{\kappa_2+1}\left(1-\frac{1}{\sqrt{\kappa_2}}\right)^{K_0^2 /2}\left\|\boldsymbol{y}_{-1}-\fy_\lambda^*\left(\boldsymbol{x}_0\right)\right\|\\ 
= & \sqrt{\kappa_2+1}\left(1-\frac{1}{\sqrt{\kappa_2}}\right)^{K_0^2 /2}\left\|\fy_\lambda^*\left(\boldsymbol{x}_0\right)\right\|\\
\leq & \tilde{\epsilon},   
\end{align*}

\begin{align*}
 & \left\|\boldsymbol{w}_0-\fy^*\left(\boldsymbol{x}_0\right)\right\| \\ 
\leq & \sqrt{\kappa_1+1}\left(1-\frac{1}{\sqrt{\kappa_1}}\right)^{K_0^1 /2}\left\|\boldsymbol{w}_{-1}-\fy^*\left(\boldsymbol{x}_0\right)\right\|\\ 
= & \sqrt{\kappa_1+1}\left(1-\frac{1}{\sqrt{\kappa_1}}\right)^{K_0^1 /2}\left\|\fy^*\left(\boldsymbol{x}_0\right)\right\| \\
\leq & \tilde{\epsilon}.  
\end{align*}
The above two blocks of inequalities are justified as follows: the first inequality is based on Lemma \ref{lem:AGD}; the second equation uses the initialization of $\boldsymbol{y}_{-1}$ and $\boldsymbol{w}_{-1}$ ; the last step use the deﬁnition of $K_0^1$, $K_0^2$ and $\tilde{\epsilon}$.

Suppose it holds that $\left\|\boldsymbol{w}_{t-1}-\fy^*\left(\boldsymbol{x}_{t-1}\right)\right\| \leq \tilde{\epsilon}$ and $\left\|\boldsymbol{y}_{t-1}-\fy_\lambda^*\left(\boldsymbol{x}_{t-1}\right)\right\| \leq \tilde{\epsilon}$
for any $t=t^{\prime}-1$,  then we have
\begin{align*}
& \left\|\boldsymbol{w}_{t^{\prime}}-\fy^*\left(\boldsymbol{x}_{t^{\prime}}\right)\right\| \\
\leq & \sqrt{\kappa_1+1}\left(1-\frac{1}{\sqrt{\kappa_1}}\right)^{K_{t^{\prime}}^1 / 2}\left\|\boldsymbol{w}_{t^{\prime}-1}-\fy^*\left(\boldsymbol{x}_{t^{\prime}}\right)\right\| \\
\leq & \sqrt{\kappa_1+1}\left(1-\frac{1}{\sqrt{\kappa_1}}\right)^{K_{t^{\prime}}^1 / 2}\left(\left\|\boldsymbol{w}_{t^{\prime}-1}-\fy^*\left(\boldsymbol{x}_{t^{\prime}-1}\right)\right\|+\left\|\fy^*\left(\boldsymbol{x}_{t^{\prime}-1}\right)-\fy^*\left(\boldsymbol{x}_{t^{\prime}}\right)\right\|\right) \\
\leq & \sqrt{\kappa_1+1}\left(1-\frac{1}{\sqrt{\kappa_1}}\right)^{K_{t^{\prime}}^1 / 2}\left(\tilde{\epsilon}+\kappa\left\|\boldsymbol{x}_{t^{\prime}-1}-\boldsymbol{x}_{t^{\prime}}\right\|\right) \\
= & \sqrt{\kappa_1+1}\left(1-\frac{1}{\sqrt{\kappa_1}}\right)^{K_{t^{\prime}}^1 / 2}\left(\tilde{\epsilon}+\kappa\left\|\boldsymbol{s}_{t^{\prime}-1}^*\right\|\right) \leq \tilde{\epsilon},    
\end{align*}

\begin{align*}
& \left\|\boldsymbol{y}_{t^{\prime}}-\fy_\lambda^*\left(\boldsymbol{x}_{t^{\prime}}\right)\right\| \\
\leq & \sqrt{\kappa_2+1}\left(1-\frac{1}{\sqrt{\kappa_2}}\right)^{K_{t^{\prime}}^2 / 2}\left\|\boldsymbol{y}_{t^{\prime}-1}-\fy_\lambda^*\left(\boldsymbol{x}_{t^{\prime}}\right)\right\| \\
\leq & \sqrt{\kappa_2+1}\left(1-\frac{1}{\sqrt{\kappa_2}}\right)^{K_{t^{\prime}}^2 / 2}\left(\left\|\boldsymbol{y}_{t^{\prime}-1}-\fy^*(\boldsymbol{x}_{t^{\prime}-1})\right\|+\left\|\fy_\lambda^*\left(\boldsymbol{x}_{t^{\prime}-1}\right)-\fy_\lambda^*\left(\boldsymbol{x}_{t^{\prime}}\right)\right\|\right) \\
\leq & \sqrt{\kappa_2+1}\left(1-\frac{1}{\sqrt{\kappa_2}}\right)^{K_{t^{\prime}}^2 / 2}\left(\tilde{\epsilon}+4\kappa\left\|\boldsymbol{x}_{t^{\prime}-1}-\boldsymbol{x}_{t^{\prime}}\right\|\right) \\
= & \sqrt{\kappa_2+1}\left(1-\frac{1}{\sqrt{\kappa_2}}\right)^{K_{t^{\prime}}^2 / 2}\left(\tilde{\epsilon}+4\kappa\left\|\boldsymbol{s}_{t^{\prime}-1}^*\right\|\right) \leq \tilde{\epsilon}.   
\end{align*}
The above two blocks of inequalities are justified as follows: the first inequality is based on Lemma \ref{lem:AGD}; the second one uses triangle inequality; the third one is based on the hypothesis of induction and Proposition~\ref{prop:condinum_y_star} and Lemma~\ref{lem:gradient_smooth_y_star}; the last step uses the definition of $K_t^1$, $K_t^2$ and $\tilde{\epsilon}$.

Combining inequality \eqref{eq:error_y_w} with Lemma \ref{lem:estimators_error}, we obtain
\begin{align*}
\left\|\nabla \mathcal{L}_\lambda^*\left(\boldsymbol{x}_t\right)-\g(\fx_t;\fy_t,\fw_t)\right\|
&\leq 2 \lambda \ell \left\|\boldsymbol{y}_t - \fy_\lambda^*\left(\boldsymbol{x}_t\right)\right\|
+ \lambda \ell\left\|\boldsymbol{w}_t - \fy^*\left(\boldsymbol{x}_t\right)\right\| \leq C_g \epsilon, \\
\left\|\nabla^2 \mathcal{L}_\lambda^*(\boldsymbol{x}_t)-\H(\fx_t;\fy_t,\fw_t)\right\|
&\leq C_1 \left\|\boldsymbol{w}_t - \fy^*\left(\boldsymbol{x}_t\right) \right\| + C_2
\left\|\boldsymbol{y}_t - \fy_\lambda^*\left(\boldsymbol{x}_t\right)\right\| \leq C_H \sqrt{M \epsilon}.
\end{align*}
\end{proof}

\subsection{The Proof of Theorem~\ref{thm:FSBA}}
\begin{proof}
Let $M=\Omega( \bar{\rho})$,
$T = \Theta\left((\varphi(\fx_0)-\varphi^*)\sqrt{M}\epsilon^{-3/2}\right)$ and the setting of $\lambda$,
then we can prove that the output
$\hat{\fx}$ of Algorithm~\ref{alg:FSBA} is an $\left((\OM(\epsilon),\OM(\kappa^{2.5}\bar{\ell}^{0.5}\epsilon^{0.5})\right)$-SOSP of $\varphi(\cdot)$.

Since the algorithm~\ref{alg:FSBA} could find an $(\epsilon, \sqrt{M \epsilon})$-SOSP of $\mathcal{L}_\lambda^*(\boldsymbol{x})$ in Lemma~\ref{lem:inexact-cubic},
then we have
\begin{equation*}
\left\|\nabla \mathcal{L}_\lambda^*\left(\boldsymbol{x}\right)\right\| \leq \epsilon, \quad
\nabla^2 \mathcal{L}_\lambda^*\left(\boldsymbol{x}\right) \succeq -\sqrt{M \epsilon} I. 
\end{equation*}

According to Proposition~\ref{prop:proxy_property}, we have
\begin{align*}
\left\|\nabla \mathcal{L}_\lambda^*(\boldsymbol{x}) - \nabla \varphi(\boldsymbol{x})\right\| &= \mathcal{O}\left(\frac{\bar{\ell} \kappa^3}{\lambda}\right), \quad \forall \boldsymbol{x} \in \mathbb{R}^{d_x}, \\
\left\|\nabla^2 \mathcal{L}_\lambda^*(\boldsymbol{x}) - \nabla^2 \varphi(\boldsymbol{x})\right\| &= \mathcal{O}\left(\frac{\bar{\ell} \kappa^5}{\lambda}\right), \quad \forall \boldsymbol{x} \in \mathbb{R}^{d_x},\\
\left|\mathcal{L}_\lambda^*(\boldsymbol{x}) - \varphi(\boldsymbol{x})\right|& = \mathcal{O}\left(\frac{\bar{\ell} \kappa^2}{\lambda}\right), \quad  \forall \boldsymbol{x} \in \mathbb{R}^{d_x}.
\end{align*}

With $\lambda \geq \bar{\ell} \kappa^3 / \epsilon $,
we have
\begin{align*}
\left\|\nabla \varphi(\boldsymbol{x})\right\| 
&= \left\|\nabla \varphi(\boldsymbol{x})-\nabla \mathcal{L}_\lambda^*(\boldsymbol{x}) + \nabla \mathcal{L}_\lambda^*(\boldsymbol{x})\right\| \\
&\leq \left\|\nabla \mathcal{L}_\lambda^*(\boldsymbol{x}) - \nabla \varphi(\boldsymbol{x})\right\| + \left\|\nabla \mathcal{L}_\lambda^*\left(\boldsymbol{x}\right)\right\| \\
&\leq \mathcal{O}\left(\frac{\bar{\ell} \kappa^3}{\lambda}\right) + \epsilon \\
& \leq \mathcal{O}(\epsilon).
\end{align*}

With $\lambda \geq \bar{\ell} \kappa^5 / \sqrt{M\epsilon}$, we have
\begin{equation*}
    \nabla^2 \varphi(\boldsymbol{x}) \succeq \nabla^2 \mathcal{L}_\lambda^*(\boldsymbol{x})-\mathcal{O}\left(\frac{\bar{\ell} \kappa^5}{\lambda}\right)I \succeq -\sqrt{M \epsilon} I -\mathcal{O}\left(\frac{\bar{\ell} \kappa^5}{\lambda}\right)I \succeq -\mathcal{O}(\sqrt{M \epsilon}) I.
\end{equation*}

With $\lambda \geq \bar{\ell} \kappa^2 / \Delta$, we have
\begin{align*}
\mathcal{L}_\lambda^*\left(\boldsymbol{x}_0\right) - \min\fL_{\lambda}^*(\fx)
&= 
\mathcal{L}_\lambda^*\left(x_0\right)-\min\fL_{\lambda}^*(\fx) +\varphi\left(\boldsymbol{x}_0\right)-\varphi^* -
\varphi\left(\boldsymbol{x}_0\right) +
\varphi^*\\
&= \Delta + 2 \mathcal{O}\left(\frac{\bar{\ell} \kappa^2}{\lambda}\right)\\
&=\mathcal{O}(\Delta).
\end{align*}

We now proceed to establish the first-order oracle complexity.
{
\begin{align*}
& \sum_{t=0}^{T-1}(K^1_t+K^2_t) \\
\leq & 4 \sqrt{3\kappa}\left[\log \left(\frac{\sqrt{3\kappa+1}}{\tilde{\epsilon}}R\right)+\sum_{t=1}^T \log \left(\sqrt{3\kappa+1}+\frac{4\kappa \sqrt{3\kappa+1}}{\tilde{\epsilon}}\left\|\boldsymbol{s}_{t-1}^*\right\|\right)\right]+2T \\
= & \frac{4 \sqrt{3\kappa}}{3}\left[3 \log \left(\frac{\sqrt{3\kappa+1}}{\tilde{\epsilon}}R\right)+\sum_{t=1}^T \log \left(\sqrt{3\kappa+1}+\frac{4\kappa \sqrt{3\kappa+1}}{\tilde{\epsilon}}\left\|\boldsymbol{s}_{t-1}^*\right\|\right)^3\right]+2T \\
\leq & \frac{4 \sqrt{3\kappa}}{3}\left[3 \log \left(\frac{\sqrt{3\kappa+1}}{\tilde{\epsilon}}R\right)+\sum_{t=1}^T \log \left(8(3\kappa+1)^{1.5}+\frac{8 (4\kappa)^3(3\kappa+1)^{1.5}}{\tilde{\epsilon}^3}\left\|\boldsymbol{s}_{t-1}^*\right\|^3\right)\right]+2T \\
= & \frac{4 \sqrt{3\kappa}}{3}\left[3 \log \left(\frac{\sqrt{3\kappa+1}}{\tilde{\epsilon}}R\right)+\log \left(\prod_{t=1}^T\left(8(3\kappa+1)^{1.5}+\frac{8 (4\kappa)^3(3\kappa+1)^{1.5}}{\tilde{\epsilon}^3}\left\|\boldsymbol{s}_{t-1}^*\right\|^3\right)\right)\right]+2T \\
\leq & \frac{4 \sqrt{3\kappa}}{3}\left[3 \log \left(\frac{\sqrt{3\kappa+1}}{\tilde{\epsilon}}R\right)+\log \left(\frac{1}{T} \sum_{t=1}^T\left(8(3\kappa+1)^{1.5}+\frac{8 (4\kappa)^3(3\kappa+1)^{1.5}}{\tilde{\epsilon}^3}\left\|\boldsymbol{s}_{t-1}^*\right\|^3\right)\right)^T\right]+2T \\
= & \frac{4 \sqrt{3\kappa} T}{3}\left[\frac{3}{T} \log \left(\frac{\sqrt{3\kappa+1}}{\tilde{\epsilon}}R\right)+\log \left(8(3\kappa+1)^{1.5}+\frac{8 (4\kappa)^3(3\kappa+1)^{1.5}}{T \tilde{\epsilon}^3} \sum_{t=1}^T\left\|\boldsymbol{s}_{t-1}^*\right\|^3\right)\right]+2T,   
\end{align*}
}
\!\!where the first inequality is based on the fact $(a+b)^3 \leq 8\left(a^3+b^3\right)$ for $a, b \geq 0$; the second inequality
is based on AM–GM inequality.

Connecting the upper bound of $\sum_{t=1}^T\left\|\boldsymbol{s}_{t-1}^*\right\|^3$  in the proof of Lemma~\ref{lem:inexact-cubic}:
\begin{equation*}
\mathcal{L}_\lambda^*\left(\boldsymbol{x}_0\right) -\min\fL_{\lambda}^*(\fx) \geq \frac{M}{24}\sum_{t=0}^{T} \left\|\boldsymbol{s}_t^*\right\|^3,  
\end{equation*}
we have
{
\begin{align*}
&\sum_{t=0}^{T-1}(K^1_t+K^2_t) \\
& \leq 2T + \frac{4 \sqrt{3\kappa} T}{3} \left( \frac{3}{T} \log\left(\frac{\sqrt{3\kappa+1}}{\tilde{\epsilon}}R\right) \right)  +\frac{4 \sqrt{3\kappa} T}{3}\log \left(8(3\kappa+1)^{1.5}+\frac{192 (4\kappa)^3(3\kappa+1)^{1.5}}{T M \bar{\epsilon}^{ 3}}\Delta\right)\\
& =\mathcal{O}\left(\sqrt{ \bar{\ell}} \kappa^3 \epsilon^{-1.5} \log (\bar{\ell}^{1.5}\kappa^{-3}  \epsilon^{-4.5}) \right)  =\tilde{\mathcal{O}}\left(\sqrt{ \bar{\ell}} \kappa^3 \epsilon^{-1.5}\right).  
\end{align*}
}
The claim follows from the fact that we call gradient oracle for $\mathcal{O}\left(\sum_{t=0}^{T-1}(K^1_t+K^2_t) \right)$ times and perform
Hessian (inverse) and exact cubic sub-problem solver calls for $\mathcal{O}(T)$ times.

\end{proof}

\section{The Proof of Section~\ref{sec:LFSBA}}
\subsection{The Proof of Lemma~\ref{lem:pre-LFSBA}}
\begin{proof}
We first explain the stopping condition of the Algorithm~\ref{alg:LFSBA} with respect to $\epsilon$.
When $\boldsymbol{x}_{t+1}$ from Algorithm \ref{alg:LFSBA} is not an $(\epsilon, \sqrt{M \epsilon})$-SOSP of $\mathcal{L}_\lambda^*\left(\boldsymbol{x}_{t+1}\right)$,
we have $\left\|\nabla \mathcal{L}_\lambda^*\left(\boldsymbol{x}_{t+1}\right) \right\| \geq \epsilon$ or $
\xi\left(\boldsymbol{x}_{t+1}\right) \geq \sqrt{M \epsilon} $.

We consider the gradient case, the equation \eqref{eq:stationary_condition_1} in Lemma~\ref{lem:stationary conditions} and Lemma~\ref{lem:property_Hessain_smooth} means

\begin{equation} \label{eq:gradient_case}
\begin{aligned}
& \left\| \nabla \mathcal{L}_\lambda^*\left(\boldsymbol{x}_{t+1}\right)\right\| \\
& =\left\|\nabla \mathcal{L}_\lambda^*\left(\boldsymbol{x}_{t+1}\right)-\g(\fx_t;\fy_t,\fw_t)-\H(\fx_{\pi(t)};\fy_{\pi(t)},\fw_{\pi(t)})\boldsymbol{s}_t^*-\frac{M}{2}\left\|\boldsymbol{s}_t^*\right\| \boldsymbol{s}_t^*\right\| \\
& \leq\left\|\nabla \mathcal{L}_\lambda^*\left(\boldsymbol{x}_{t+1}\right)-\nabla \mathcal{L}_\lambda^*\left(\boldsymbol{x}_t\right)-\nabla^2 \mathcal{L}_\lambda^*\left(\boldsymbol{x}_t\right) \boldsymbol{s}_t^*\right\|+\left\|\nabla \mathcal{L}_\lambda^*\left(\boldsymbol{x}_t\right)-\g(\fx_t;\fy_t,\fw_t)\right\|+\frac{M}{2}\left\|\boldsymbol{s}_t^*\right\|^2 \\
&~~~~~+\left\|\nabla^2 \mathcal{L}_\lambda^*\left(\boldsymbol{x}_{\pi(t)}\right) \boldsymbol{s}_t^*-\H(\fx_{\pi(t)};\fy_{\pi(t)},\fw_{\pi(t)}) \boldsymbol{s}_t^*\right\| 
+\left\|\nabla^2 \mathcal{L}_\lambda^*\left(\boldsymbol{x}_t\right) \boldsymbol{s}_t^*-\nabla^2 \mathcal{L}_\lambda^*\left(\boldsymbol{x}_{\pi(t)}\right) \boldsymbol{s}_t^*\right\| \\
& \leq \frac{\bar{\rho}}{2}\left\|\boldsymbol{s}_t^*\right\|^2+\bar{C}_g \epsilon+\bar{C}_H \sqrt{M \epsilon}\left\|\boldsymbol{s}_t^*\right\|+\frac{M}{2}\left\|\boldsymbol{s}_t^*\right\|^2+\bar{\rho}\left\|\boldsymbol{s}_t^*\right\|\left\|\boldsymbol{x}_{\pi(t)}-\boldsymbol{x}_t\right\| \\
& =\frac{\bar{\rho} +M}{2}\left\|\boldsymbol{s}_t^*\right\|^2+\bar{C}_g \epsilon+\bar{C}_H \sqrt{M \epsilon}\left\|\boldsymbol{s}_t^*\right\|+\bar{\rho}\left\|\boldsymbol{s}_t^*\right\|\left\|\boldsymbol{x}_{\pi(t)}-\boldsymbol{x}_t\right\| \\
& \leq \frac{\bar{\rho} +M}{2}\left\|\boldsymbol{s}_t^*\right\|^2+\bar{C}_g \epsilon+\frac{\bar{C}_H\left(\epsilon+M\left\|\boldsymbol{s}_t^*\right\|^2\right)}{2}+ \bar{\rho} \left\|\boldsymbol{s}_t^*\right\|\left\|\boldsymbol{x}_{\pi(t)}-\boldsymbol{x}_t\right\| \\
& =\frac{\left(1+\bar{C}_H\right) M+\bar{\rho}}{2}\left\|\boldsymbol{s}_t^*\right\|^2+\left(\bar{C}_g+\frac{\bar{C}_H}{2}\right) \epsilon+ \bar{\rho}\left\|\boldsymbol{s}_t^*\right\|\left\|\boldsymbol{x}_{\pi(t)}-\boldsymbol{x}_t\right\|.   
\end{aligned}
\end{equation}

Then we consider the Hessian case, the equation \eqref{eq:stationary_condition_2} in Lemma~\ref{lem:stationary conditions} means:
\begin{equation} \label{eq:hessian_case}
\begin{aligned}
& \nabla^2 \mathcal{L}_\lambda^*\left(\boldsymbol{x}_{t+1}\right) \\
\succeq & \H(\fx_{\pi(t)};\fy_{\pi(t)},\fw_{\pi(t)})-\left\|\H(\fx_{\pi(t)};\fy_{\pi(t)},\fw_{\pi(t)})-\nabla^2 \mathcal{L}_\lambda^*\left(\boldsymbol{x}_{t+1}\right)\right\| \boldsymbol{I} \\
\succeq & -\frac{M}{2}\left\|\boldsymbol{s}_t^*\right\| \boldsymbol{I}-\left\|\H(\fx_{\pi(t)};\fy_{\pi(t)},\fw_{\pi(t)})-\nabla^2 \mathcal{L}_\lambda^*\left(\boldsymbol{x}_{t+1}\right)\right\| \boldsymbol{I} \\
\succeq & -\frac{M}{2}\left\|\boldsymbol{s}_t^*\right\| \boldsymbol{I}-\left\|\H(\fx_{\pi(t)};\fy_{\pi(t)},\fw_{\pi(t)})-\nabla^2 \mathcal{L}_\lambda^*\left(\boldsymbol{x}_{\pi(t)}\right)\right\| \boldsymbol{I}-\left\|\nabla^2 \mathcal{L}_\lambda^*\left(\boldsymbol{x}_{\pi(t)}\right)
-\nabla^2 \mathcal{L}_\lambda^*\left(\boldsymbol{x}_t\right)\right\| \boldsymbol{I} \\
&-\left\|\nabla^2 \mathcal{L}_\lambda^*\left(\boldsymbol{x}_t\right)-\nabla^2 \mathcal{L}_\lambda^*\left(\boldsymbol{x}_{t+1}\right)\right\| \boldsymbol{I}\\
\succeq & -\frac{M}{2}\left\|\boldsymbol{s}_t^*\right\| \boldsymbol{I}-\bar{C}_H \sqrt{M \epsilon} \boldsymbol{I} -\bar{\rho} \|\boldsymbol{x}_{\pi(t)}-\boldsymbol{x_t}\| \boldsymbol{I}
-\bar{\rho}\left\|\boldsymbol{s}_t^*\right\| \boldsymbol{I} \\
\succeq & -\frac{M+ 2\bar{\rho}}{2}\left\|\boldsymbol{s}_t^*\right\| \boldsymbol{I}-\bar{C}_H \sqrt{M \epsilon} \boldsymbol{I} -\bar{\rho} \|\boldsymbol{x}_{\pi(t)}-\boldsymbol{x}_t\| \boldsymbol{I}.    
\end{aligned}
\end{equation}

If $\boldsymbol{x}_{t+1}$ is not an $(\epsilon, \sqrt{M \epsilon})$-SOSP, then
\begin{itemize}
    \item if $\left\|\nabla \mathcal{L}_\lambda^*\left(\boldsymbol{x}_{t+1}\right) \right\| \geq \epsilon$, we have
    \begin{equation}
        \epsilon \leq \frac{1}{ \left( 1-\bar{C}_g -\frac{\bar{C}_H}{2} \right) } \left(\frac{\left(1+\bar{C}_H\right) M + \bar{\rho}}{2}\left\|\boldsymbol{s}_t^*\right\|^2 + \bar{\rho} \left\|\boldsymbol{s}_t^*\right\| \left\|\boldsymbol{x}_{\pi(t)}-\boldsymbol{x}_t\right\| \right).
    \end{equation}
    \item if $
\xi\left(\boldsymbol{x}_{t+1}\right) \geq \sqrt{M \epsilon} $, we have 
\begin{equation}
 \epsilon \leq \frac{1}{M}\left(\frac{1}{1-\bar{C}_H}\right)^2
\left(\frac{M+2\bar{\rho}}{2} \left\|\boldsymbol{s}_t^*\right\|+\bar{\rho} \|\boldsymbol{x}_{\pi(t)}-\boldsymbol{x}_t\|\right)^2.
\end{equation}
\end{itemize}

With $\bar{C}_g = 1/576$ and $\bar{C}_H= 1/288$, we can choose a upper bound as the stopping condition:
\begin{equation}
\label{eq:accuracy_super_bound}
\epsilon \leq  \frac{1}{M}(\frac{288}{287})^2
\left(\frac{M+2 \bar{\rho}}{\sqrt{2}} \left\|\boldsymbol{s}_t\right\|+\bar{\rho}\|\boldsymbol{x}_{\pi(t)}-\boldsymbol{x}_t\|\right)^2.
\end{equation}

That means if
\begin{equation*}
\epsilon \geq
\frac{1}{M}(\frac{288}{287})^2
\left(\frac{M+2 \bar{\rho}}{\sqrt{2}} \left\|\boldsymbol{s}_t\right\|+\bar{\rho}\|\boldsymbol{x}_{\pi(t)}-\boldsymbol{x}_{t}\|\right)^2. 
\end{equation*}
 
then $\boldsymbol{x}_{t+1}$ from Algorithm \ref{alg:LFSBA} is an $(\epsilon, \sqrt{M \epsilon})$-SOSP of $\mathcal{L}_\lambda^*\left(\boldsymbol{x}_{t+1}\right)$.

Next, we need to examine the difference $\mathcal{L}_\lambda^*\left(\boldsymbol{x}_{t}\right)- \mathcal{L}_\lambda^*\left(\boldsymbol{x}_{t+1}\right)$.

For the sake of analysis, we need to take a larger upper bound on $\epsilon$:
{
\begin{equation}
\label{eq:accuracy_further_result}
\begin{aligned}
 \epsilon \leq & \frac{1}{M}(\frac{288}{287})^2
\left(\frac{M+2 \bar{\rho}}{\sqrt{2}}
\left\|\boldsymbol{s}_t^*\right\|+\bar{\rho}\|\boldsymbol{x}_{\pi(t)}-\boldsymbol{x}_{t}\|\right)^2 \\
= & (\frac{288}{287})^2
\left(\frac{\left( M+2\bar{\rho}\right)^2}{2M}
\left\|\boldsymbol{s}_t^*\right\|^2+ \frac{\sqrt{2}\left(M+2\bar{\rho}\right)\bar{\rho}}{M}\left\|\boldsymbol{s}_t^*\right\|\|\boldsymbol{x}_{\pi(t)}-\boldsymbol{x}_{t}\|+
\frac{\bar{\rho}^2}{M} \|\boldsymbol{x}_{\pi(t)}-\boldsymbol{x}_{t}\|^2\right) \\
= & (\frac{288}{287})^2
\left(\frac{M^2+4M\bar{\rho}+4\bar{\rho}^2}{2M}
\left\|\boldsymbol{s}_t^*\right\|^2+\frac{\sqrt{2}\left(M+2\bar{\rho}\right)\bar{\rho}}{M}\left\|\boldsymbol{s}_t^*\right\|\|\boldsymbol{x}_{\pi(t)}-\boldsymbol{x}_{t}\|+\frac{\bar{\rho}^2}{M} \|\boldsymbol{x}_{\pi(t)}-\boldsymbol{x}_{t}\|^2\right)\\
\leq & (\frac{288}{287})^2
\left(\left(\frac{1}{2} M+4\bar{\rho} \right)
\left\|\boldsymbol{s}_t^*\right\|^2 +3\sqrt{2}\bar{\rho} \left\|\boldsymbol{s}_t^*\right\|\|\boldsymbol{x}_{\pi(t)}-\boldsymbol{x}_{t}\|
+\frac{\bar{\rho}^2}{M} \|\boldsymbol{x}_{\pi(t)}-\boldsymbol{x}_{t}\|^2\right).
\end{aligned}
\end{equation}
}
Then, according to inequality \eqref{eq:accuracy_super_bound} and $M \geq \bar{\rho}$, we will have

\begin{equation}
\label{eq:descent_L_first}
\begin{aligned}
& \mathcal{L}_\lambda^*\left(\boldsymbol{x}_{t+1}\right)-
\mathcal{L}_\lambda^*\left(\boldsymbol{x}_t\right) \\
\leq & \nabla \mathcal{L}_\lambda^*\left(\boldsymbol{x}_t\right)^{\top} \boldsymbol{s}_t^*+\frac{1}{2}\left(\boldsymbol{s}_t^*\right)^{\top} \nabla^2 \mathcal{L}_\lambda^*(\boldsymbol{x}_t) \boldsymbol{s}_t^*+\frac{\bar{\rho}}{6}\left\|\boldsymbol{s}_t^*\right\|^3 \\
= & \g(\fx_t;\fy_t,\fw_t)^{\top} \boldsymbol{s}_t^*+\left(\nabla \mathcal{L}_\lambda^*\left(\boldsymbol{x}_t\right)-\g(\fx_t;\fy_t,\fw_t)\right)^{\top} \boldsymbol{s}_t^*
+\frac{1}{2}\left(\boldsymbol{s}_t^*\right)^{\top} \nabla^2 \mathcal{L}_\lambda^*(\boldsymbol{x}_t) \boldsymbol{s}_t^*+\frac{\bar{\rho}}{6}\left\|\boldsymbol{s}_t^*\right\|^3\\
\leq &- \left(\boldsymbol{s}_t^*\right)^{\top} \H(\fx_{\pi(t)};\fy_{\pi(t)},\fw_{\pi(t)})\boldsymbol{s}_t^* -\frac{M }{2} \|\boldsymbol{s}_t^*\|^3  +\bar{C}_g\epsilon\left\|\boldsymbol{s}_t^*\right\|
 +\frac{1}{2}\left(\boldsymbol{s}_t^*\right)^{\top} \nabla^2 \mathcal{L}_\lambda^*(\boldsymbol{x}_t) \boldsymbol{s}_t^*+\frac{\bar{\rho}}{6}\left\|\boldsymbol{s}_t^*\right\|^3\\
 \leq &  -\frac{1}{2} \left(\boldsymbol{s}_t^*\right)^{\top} \H(\fx_{\pi(t)};\fy_{\pi(t)},\fw_{\pi(t)})\boldsymbol{s}_t^*
 -\frac{M }{4} \|\boldsymbol{s}_t^*\|^3 
+\bar{C}_g\epsilon\left\|\boldsymbol{s}_t^*\right\|
 +\frac{1}{2}\left(\boldsymbol{s}_t^*\right)^{\top} \nabla^2 \mathcal{L}_\lambda^*(\boldsymbol{x}_t) \boldsymbol{s}_t^*+\frac{\bar{\rho}}{6}\left\|\boldsymbol{s}_t^*\right\|^3 \\
\leq & -\frac{M }{4} \|\boldsymbol{s}_t^*\|^3 +\bar{C}_g\epsilon\left\|\boldsymbol{s}_t^*\right\|
+\frac{\bar{\rho}}{6}\left\|\boldsymbol{s}_t^*\right\|^3\\
&+\frac{1}{2}\left(\boldsymbol{s}_t^*\right)^{\top} \left(\nabla^2 \mathcal{L}_\lambda^*(\boldsymbol{x}_t) -\nabla^2 \mathcal{L}_\lambda^*(\boldsymbol{x}_{\pi(t)}) 
+ \nabla^2 \mathcal{L}_\lambda^*(\boldsymbol{x}_{\pi(t)}) -\H(\fx_{\pi(t)};\fy_{\pi(t)},\fw_{\pi(t)})\right) \boldsymbol{s}_t^* \\
 \leq & -\frac{M}{4}\left\|\boldsymbol{s}_t^*\right\|^3
 +\bar{C}_g\epsilon\left\|\boldsymbol{s}_t^*\right\| + \frac{\bar{\rho}}{2}\left\|\boldsymbol{s}_t^*\right\|^2 \left\|\boldsymbol{x}_{\pi(t)}-\boldsymbol{x}_t\right\|
 +\frac{\bar{C}_H\sqrt{M \epsilon}}{2}\left\|\boldsymbol{s}_t^*\right\|^2+\frac{\bar{\rho}}{6}\left\|\boldsymbol{s}_t^*\right\|^3 \\
 \leq & -\frac{M}{4}\left\|\boldsymbol{s}_t^*\right\|^3
   + \frac{\bar{\rho}}{2}\left\|\boldsymbol{s}_t^*\right\|^2 \left\|\boldsymbol{x}_{\pi(t)}-\boldsymbol{x}_t\right\|+ \left(\bar{C}_g+\frac{\bar{C}_H}{4}\right)\epsilon\left\|\boldsymbol{s}_t^*\right\|+
\frac{6\bar{C}_H M + 4\bar{\rho}}{24}\left\|\boldsymbol{s}_t^*\right\|^3,
\end{aligned}
\end{equation}
where the first inequality comes from the equation \eqref{eq:Hessian_smmoth_2} of Lemma \ref{lem:property_Hessain_smooth}; the second inequality comes from the  equation \eqref{eq:stationary_condition_1} of Lemma \ref{lem:stationary conditions}.

We need to address the cross terms in the preceding expression to derive a larger upper bound for 
$\mathcal{L}_\lambda^*\left(\boldsymbol{x}_{t+1}\right)-
\mathcal{L}_\lambda^*\left(\boldsymbol{x}_t\right)$.

By Young’s inequality, we can obtain
{
\begin{equation}
  \label{eq:descent_square_1}
  \frac{\bar{\rho}}{2}  \left\|\boldsymbol{s}_t^*\right\|^2\|\boldsymbol{x}_{\pi(t)}-\boldsymbol{x_t}\|=\left(\frac{M^\frac{2}{3}}{2 \cdot 32^\frac{1}{3}}  \left\|\boldsymbol{s}_t^*\right\|^2\right) \cdot\left(\frac{32^\frac{1}{3} \bar{\rho}}{M^\frac{2}{3}}\|\boldsymbol{x}_{\pi(t)}-\boldsymbol{x}_t\|\right) \leq \frac{M}{24}  \left\|\boldsymbol{s}_t^*\right\|^3+\frac{32 \bar{\rho}^3}{3 M^2}\|\boldsymbol{x}_{\pi(t)}-\boldsymbol{x}_t\|^3.  
\end{equation}
}

Then according to equation \eqref{eq:accuracy_further_result} and 
$M \geq \bar{\rho}$, we have
{
\begin{equation}
\label{eq:es_first}
    \epsilon\left\|\boldsymbol{s}_t^*\right\| 
 \leq (\frac{288}{287})^2
\left(\left(\frac{1}{2} M+4\bar{\rho} \right)
\left\|\boldsymbol{s}_t^*\right\|^3 +3\sqrt{2}\bar{\rho} \left\|\boldsymbol{s}_t^*\right\|^2\|\boldsymbol{x}_{\pi(t)}-\boldsymbol{x}_{t}\|
+\frac{\bar{\rho}^2}{M}\left\|\boldsymbol{s}_t^*\right\| \|\boldsymbol{x}_{\pi(t)}-\boldsymbol{x}_{t}\|^2\right).
\end{equation}
}

Also by Young’s inequality, we have 
\begin{equation}
  \label{eq:descent_square_2}
  \bar{\rho} \left\|\boldsymbol{s}_t^*\right\|^2\|\boldsymbol{x}_{\pi(t)}-\boldsymbol{x_t}\|=\left(\frac{M^{\frac{2}{3}}}{24^{\frac{2}{3}}}  \left\|\boldsymbol{s}_t^*\right\|^2\right) \left(\frac{24^{\frac{2}{3}} \bar{\rho}}{M^{\frac{2}{3}}}\|\boldsymbol{x}_{\pi(t)}-\boldsymbol{x}_t\|\right) \leq \frac{M}{36}  \left\|\boldsymbol{s}_t^*\right\|^3+\frac{576 \bar{\rho}^3}{3 M^2}\|\boldsymbol{x}_{\pi(t)}-\boldsymbol{x}_t\|^3,  
\end{equation}
{
\begin{equation}
\label{eq:lazy_square}
\frac{\bar{\rho}^2}{M}\left\|\boldsymbol{s}_t^*\right\| \|\boldsymbol{x}_{\pi(t)}-\boldsymbol{x}_t\|^2
= \left( \frac{ M^{\frac{1}{3}} }{36^{\frac{1}{3}}}\left\|\boldsymbol{s}_t^*\right\| \right) \left(\frac{36^{\frac{1}{3}}\bar{\rho}^2}{M^\frac{4}{3}}\|\boldsymbol{x}_{\pi(t)}-\boldsymbol{x}_t\|^2 \right) \leq \frac{M}{108}\left\|\boldsymbol{s}_t^*\right\|^3+ \frac{4\bar{\rho}^3}{M^2} \|\boldsymbol{x}_{\pi(t)}-\boldsymbol{x}_t\|^3.   
\end{equation}
}

By connecting inequalities \eqref{eq:descent_square_2} and \eqref{eq:lazy_square} to inequality \eqref{eq:es_first}, we get
\begin{equation} \label{eq:es_result}
\begin{aligned}
\epsilon\left\|\boldsymbol{s}_t^*\right\| 
&
\leq (\frac{288}{287})^2 \left( \left( (\frac{1}{2}  + \frac{\sqrt{2}}{12}+\frac{1}{108})M+4\bar{\rho}\right)\left\|\boldsymbol{s}_t^*\right\|^3 +\frac{\left(576\sqrt{2}+4\right)\bar{\rho}^3}{M^2}\|\boldsymbol{x}_{\pi(t)}-\boldsymbol{x}_{t}\|^3 \right)\\
& \leq (\frac{16}{25} M+ \frac{21}{5}\bar{\rho}) \left\|\boldsymbol{s}_t^*\right\|^3 +(\frac{288}{287})^2 \frac{\left(576\sqrt{2}+4\right)\bar{\rho}^3}{M^2}\|\boldsymbol{x}_{\pi(t)}-\boldsymbol{x}_{t}\|^3.
\end{aligned}
\end{equation}

By connecting inequalities \eqref{eq:descent_square_1} and \eqref{eq:es_result} to inequality \eqref{eq:descent_L_first}, we obtain

\begin{equation} \label{eq:descent_L_final}
\begin{aligned}
& \mathcal{L}_\lambda^*\left(\boldsymbol{x}_{t+1}\right)-
\mathcal{L}_\lambda^*\left(\boldsymbol{x}_t\right) \\
\overset{\eqref{eq:descent_square_1}}{\leq} & -\frac{M}{4}\left\|\boldsymbol{s}_t^*\right\|^3
+\frac{M}{24}\left\|\boldsymbol{s}_t^*\right\|^3
+\frac{32\bar{\rho}^3}{3M^2}\left\|\boldsymbol{x}_{\pi(t)}-\boldsymbol{x}_t\right\|^3
+\frac{6\bar{C}_H M + 4\bar{\rho}}{24}\left\|\boldsymbol{s}_t^*\right\|^3\\
+ &  \left(\bar{C}_g+\frac{\bar{C}_H}{4}\right)  \left( (\frac{16}{25} M+ \frac{21}{5}\bar{\rho})\left\|\boldsymbol{s}_t^*\right\|^3 +(\frac{288}{287})^2\frac{\left(576\sqrt{2}+4\right)\bar{\rho}^3}{M^2}\|\boldsymbol{x}_{\pi(t)}-\boldsymbol{x}_t\|^3 \right)\\
\leq & \frac{-\frac{37}{5}M + \frac{32}{5}\bar{\rho}}{36}\left\|\boldsymbol{s}_t^*\right\|^3 +\frac{322\bar{\rho}^3}{25M^2}\|\boldsymbol{x}_{\pi(t)}-\boldsymbol{x}_t\|^3\\
\leq & -\frac{M}{36}\left\|\boldsymbol{s}_t^*\right\|^3 +\frac{322\bar{\rho}^3}{25M^2}\|\boldsymbol{x}_{\pi(t)}-\boldsymbol{x}_t\|^3.   
\end{aligned}
\end{equation}

According to inequality \eqref{eq:gradient_case} and \eqref{eq:accuracy_further_result} and $\bar{\rho} \leq M$, we can get

\begin{align*}
\left\|\nabla \mathcal{L}_\lambda^*\left(\boldsymbol{x}_{t+1}\right) \right\| 
& \leq  \frac{\left(1+\bar{C}_H\right) M + \bar{\rho}}{2}\left\|\boldsymbol{s}_t^*\right\|^2+\left(\bar{C}_g+\frac{\bar{C}_H}{2}\right) 
\epsilon + \bar{\rho} \left\|\boldsymbol{s}_t^*\right\| \left\|\boldsymbol{x}_{\pi(t)}-\boldsymbol{x}_{t}\right\| \\
& \leq  \frac{\left(1+\bar{C}_H\right) M + \bar{\rho}}{2}\left\|\boldsymbol{s}_t^*\right\|^2 + \bar{\rho} \left\|\boldsymbol{s}_t^*\right\| \left\|\boldsymbol{x}_{\pi(t)}-\boldsymbol{x}_{t}\right\| \\
&+ \frac{288}{287^2}
\left(\left(\frac{1}{2} M+4\bar{\rho} \right)
\left\|\boldsymbol{s}_t^*\right\|^2 +3\sqrt{2}\bar{\rho}\left\|\boldsymbol{s}_t^*\right\|\|\boldsymbol{x}_{\pi(t)}-\boldsymbol{x}_{t}\|
+\frac{\bar{\rho}^2}{M} \|\boldsymbol{x}_{\pi(t)}-\boldsymbol{x}_{t}\|^2\right)\\
& \leq \frac{25 \left(M+\bar{\rho}\right)}{48} \left\|\boldsymbol{s}_t^*\right\|^2
+\frac{49}{48}\bar{\rho} \left\|\boldsymbol{s}_t^*\right\| \left\|\boldsymbol{x}_{\pi(t)}-\boldsymbol{x}_{t}\right\| 
+\frac{1}{286} \frac{\bar{\rho}^2}{M} \|\boldsymbol{x}_{\pi(t)}-\boldsymbol{x}_{t}\|^2\\
& \leq \frac{25M}{24} \left\|\boldsymbol{s}_t^*\right\|^2
+\frac{49}{48}\bar{\rho} \left\|\boldsymbol{s}_t^*\right\| \left\|\boldsymbol{x}_{\pi(t)}-\boldsymbol{x}_{t}\right\| 
+\frac{1}{286} \frac{\bar{\rho}^2}{M} \|\boldsymbol{x}_{\pi(t)}-\boldsymbol{x}_{t}\|^2. 
\end{align*}

Indeed, using the convexity of the function $t \mapsto t^{3/2}$ for $t \geq 0$, that means $(a + b + c)^{\frac{3}{2}} \leq 3^{\frac{1}{2}} (a^{\frac{3}{2}} + b^{\frac{3}{2}} + c^{\frac{3}{2}})$, we obtain
\begin{align*}
\left\|\nabla \mathcal{L}_\lambda^*\left(\boldsymbol{x}_{t+1}\right)  \right\|^{3 / 2}  &\leq  \left(\frac{25M}{24} \left\|\boldsymbol{s}_t^*\right\|^2
+\frac{49}{48}\bar{\rho} \left\|\boldsymbol{s}_t^*\right\| \left\|\boldsymbol{x}_{\pi(t)}-\boldsymbol{x}_{t}\right\| 
+\frac{1}{286} \frac{\bar{\rho}^2}{M} \|\boldsymbol{x}_{\pi(t)}-\boldsymbol{x}_{t}\|^2 \right)^{\frac{3}{2}} \\
&\leq \sqrt{3} \left(\frac{25M}{24} \right)^{\frac{3}{2}}\left\|\boldsymbol{s}_t^*\right\|^3 +\sqrt{3} \left(\frac{49}{48}\right)^{\frac{3}{2} }\left(\bar{\rho} \left\|\boldsymbol{s}_t^*\right\| \left\|\boldsymbol{x}_{\pi(t)}-\boldsymbol{x}_{t}\right\| \right)^{\frac{3}{2}} \\
&+\sqrt{3} \frac{1}{286^{\frac{3}{2}}} \frac{\bar{\rho}^3}{M^{\frac{3}{2}}}\|\boldsymbol{x}_{\pi(t)}-\boldsymbol{x}_{t}\|^3 \\
&\leq \frac{ 39\sqrt{3} }{24}  M^{\frac{3}{2}}
\left\|\boldsymbol{s}_t^*\right\|^3 + \frac{13 \sqrt{3}}{24} \frac{\bar{\rho}^3}{M^{\frac{3}{2}}}\|\boldsymbol{x}_{\pi(t)}-\boldsymbol{x}_{t}\|^3,   
\end{align*}
where the bound 
$(\bar{\rho} \left\|\boldsymbol{s}_t^*\right\| \|\boldsymbol{x}_{\pi(t)}-\boldsymbol{x}_{t}\|)^{3 / 2}
\leq \frac{ M^{3 / 2}}{2} \left\|\boldsymbol{s}_t^*\right\|^3 +\frac{\bar{\rho}^3}{2 M^{3 / 2}}\|\boldsymbol{x}_{\pi(t)}-\boldsymbol{x}_{t}\|^3$
is used to establish the third inequality.

Also, according to inequality \eqref{eq:hessian_case} 
and \eqref{eq:accuracy_super_bound} and $\bar{\rho} \leq M$, we can get

\begin{align*}
\xi\left(\mathbf{x}_{t+1}\right)
 & \leq  \frac{M+2\bar{\rho}}{2} \left\|\boldsymbol{s}_t^*\right\|  +\bar{\rho}\|\boldsymbol{x}_{\pi(t)}-\boldsymbol{x}_{t}\|+\bar{C}_H\sqrt{M \epsilon}\\
 &  \leq \left(\frac{\bar{C}_H }{ \sqrt{2} \left(1-\bar{C}_H\right)} + \frac{1}{2} \right)\left(M+2\bar{\rho}\right)\left\|\boldsymbol{s}_t^*\right\| + \frac{288}{287}\bar{\rho}\|\boldsymbol{x}_{\pi(t)}-\boldsymbol{x}_{t}\|\\
 & \leq \left(\frac{1}{287 \sqrt{2} } + \frac{1}{2} \right)\left(3M\right)\left\|\boldsymbol{s}_t^*\right\| + \frac{288}{287}\bar{\rho}\|\boldsymbol{x}_{\pi(t)}-\boldsymbol{x}_{t}\|.
\end{align*}

Then, using convexity of the function $t \mapsto t^3 $ for
$t \geq 0$, that means $(a + b)^3 \leq 4(a^3 + b^3)$, we get
\begin{align*}
 \xi\left(\boldsymbol{x}_{t+1}\right)^3 
 & \leq  \left(\left(\frac{1 }{ 287 \sqrt{2} } + \frac{1}{2} \right)\left(3M \right)\left\|\boldsymbol{s}_t^*\right\|  + \frac{288}{287} \bar{\rho} \|\boldsymbol{x}_{\pi(t)}-\boldsymbol{x}_{t}\| \right)^3\\
  &  \leq 108 \left(\frac{1}{287\sqrt{2}} + \frac{1}{2} \right)^3 M^3 \left\|\boldsymbol{s}_t^*\right\|^3 +4 (\frac{288}{287} )^3 \bar{\rho}^3 \|\boldsymbol{x}_{\pi(t)}-\boldsymbol{x}_{t}\|^3.   
\end{align*}

Hence, rearranging the above equation, we obtain
\begin{align*}
\frac{1}{120\sqrt{3M}} \left\|\nabla \mathcal{L}_\lambda^*\left(\boldsymbol{x}_{t+1}\right)\right\|^{3 / 2}  
&\leq \frac{M}{72}\left\|\boldsymbol{s}_t^*\right\|^3 +
\frac{\bar{\rho}^3}{216 M^2} \|\boldsymbol{x}_{\pi(t)}-\boldsymbol{x}_{t}\|^3, \\
\frac{1}{987M^2} \xi\left(\boldsymbol{x}_{t+1}\right)^3 
&\leq \frac{M}{72} \left\|\boldsymbol{s}_t^*\right\|^3 +
\frac{\bar{\rho}^3}{144M^2} \|\boldsymbol{x}_{\pi(t)}-\boldsymbol{x}_{t}\|^3.
\end{align*}

Finally, connecting with the inequality \eqref{eq:descent_L_final}, we can obtain 
\begin{equation*}
\mathcal{L}_\lambda^*\left(\boldsymbol{x}_{t}\right)
-\mathcal{L}_\lambda^*\left(\boldsymbol{x}_{t+1}\right)
\geq \gamma(\boldsymbol{x}_{t+1}) +\frac{M}{72} \left\|\boldsymbol{s}_t^*\right\|^3-\frac{13 \bar{\rho}^3}{M^2}\|\boldsymbol{x}_{\pi(t)}-\boldsymbol{x}_t\|^3.   
\end{equation*}
\end{proof}

\subsection{The Proof of Theorem~\ref{thm:Lazy_Hessian_calls}}
\begin{proof}
Without loss of generality, we assume $T$ is a multiple of $m$, such that $m: T = mh$, therefore we can divide the method into $h$ stages, with the 
$i$-th stage ($1 \leq i \leq h$). And by the definition of $T$, we have 
\begin{equation*}
\left\|\nabla \mathcal{L}_\lambda^*\left(\boldsymbol{x}_{t}\right) \right\| \geq \epsilon \quad \text{ or }   \quad \xi\left(\boldsymbol{x}_{t}\right) \geq \sqrt{M \epsilon} \quad \text { for } \quad t=0, \ldots, T-1.     
\end{equation*}

Consequently, by Lemma~\ref{lem:pre-LFSBA} we have
\begin{equation*}
\mathcal{L}_\lambda^*(\boldsymbol{x}_t)-\mathcal{L}_\lambda^*(\boldsymbol{x}_{t+1}) \geq \gamma(\boldsymbol{x}_{t+1})+ \frac{M}{720} \left\|\boldsymbol{s}_t^*\right\|^3+
 \frac{9M}{720} \left\|\boldsymbol{s}_t^*\right\|^3-\frac{13 \bar{\rho}^3}{M^2}\|\boldsymbol{x}_{\pi(t)}-\boldsymbol{x}_t\|^3,  \text { for }  t=0, \ldots, T-1. 
\end{equation*}

We first consider $1$-th phase of the method, we have
\begin{equation*}
\mathcal{L}_\lambda^*(\boldsymbol{x}_t)-\mathcal{L}_\lambda^*(\boldsymbol{x}_{t+1}) \geq \gamma(\boldsymbol{x}_{t+1})+ \frac{M}{720} \left\|\boldsymbol{s}_t^*\right\|^3+
 \frac{9M}{720} \left\|\boldsymbol{s}_t^*\right\|^3-\frac{13 \bar{\rho}^3}{M^2}\|\boldsymbol{x}_{0}-\boldsymbol{x}_t\|^3 ,\text { for } t=0, \ldots, m-1.    
\end{equation*}

Telescoping this bound for different $t$, and using triangle inequality for the last negative term,
\begin{equation*}
\left\|\boldsymbol{x}_0-\boldsymbol{x}_t\right\| \leq \sum_{i=0}^{t-1} \left\|\boldsymbol{x}_{i+1}-\boldsymbol{x}_i\right\|.    
\end{equation*}

Then, we have
\begin{equation}
\label{eq:1_y_2_T}
\mathcal{L}_\lambda^*(\boldsymbol{x}_0)-\mathcal{L}_\lambda^*(\boldsymbol{x}_m) 
 \geq \sum_{t=0}^{m-1} \gamma(\boldsymbol{x}_{t+1}) +
 \frac{M}{720} \sum_{t=0}^{m-1} \left\|\boldsymbol{s}_t^*\right\|^3
 + \frac{9M}{720} \sum_{t=1}^{m} r_{t}^3-\frac{13\bar{\rho}^3}{M^2} \sum_{t=1}^{m}\left(\sum_{i=1}^t r_i\right)^3.  
\end{equation}

By using Lemma \ref{lem:sum_inequality} with $r_{t+1}:=\left\|\boldsymbol{x}_{t+1}-\boldsymbol{x}_t\right\|$ for each $0 \leq t \leq m-1$, and let $M \geq 8(m+1)\bar{\rho}$, we have
\begin{equation*}
\mathcal{L}_\lambda^*\left(\boldsymbol{x}_0\right)-\mathcal{L}_\lambda^*\left(\boldsymbol{x}_m\right)  \geq \sum_{t=0}^{m-1} \gamma(\boldsymbol{x}_{t+1}).    
\end{equation*}

Using the same analytical method as the first phase and the definition of $\gamma(\boldsymbol{x})$, for the $i$-th $(1 \leq i \leq t)$ phase of the method with $M = \Omega(m\bar{\rho})$, we have
\begin{equation*}
\mathcal{L}_\lambda^*\left(\boldsymbol{x}_{m(i-1)}\right)-\mathcal{L}_\lambda^*\left(\boldsymbol{x}_{mi}\right) \geq \frac{1}{987M^2} \sum_{t=0}^{m-1}\xi\left(\boldsymbol{x}_{t+1}\right)^3  \geq \frac{m}{987\sqrt{ M}} \epsilon^{3 / 2},   
\end{equation*}
or
\begin{equation*}
\mathcal{L}_\lambda^*\left(\boldsymbol{x}_{m(i-1)}\right)-\mathcal{L}_\lambda^*\left(\boldsymbol{x}_{mi}\right) \geq  \frac{1}{120\sqrt{3M}} \sum_{t=0}^{m-1}\left\|\nabla \mathcal{L}_\lambda^*\left(\boldsymbol{x}_{t+1}\right)\right\|^{3 / 2}  \geq \frac{m}{120\sqrt{ 3M}} \epsilon^{3 / 2}.    
\end{equation*}

Telescoping this bound for all phases, we obtain
\begin{equation*}
\mathcal{L}_\lambda^*\left(\boldsymbol{x}_0\right) -\min\fL_{\lambda}^*(\fx) \geq \mathcal{L}_\lambda^*\left(\boldsymbol{x}_0\right) -\mathcal{L}_\lambda^*\left(\boldsymbol{x}_T\right)  \geq \frac{T}{120  \sqrt{3M} }\epsilon^{3/2}. 
\end{equation*}

That means the output $\hat{\boldsymbol{x}}$ of Algorithm \ref{alg:LFSBA} is an $(\epsilon, \OM(\sqrt{M\epsilon})$-SOSP of $\mathcal{L}_\lambda^*(\boldsymbol{x})$. And using the same analysis in Theorem~\ref{thm:FSBA}, we can prove the output $\hat{\boldsymbol{x}}$ of Algorithm \ref{alg:LFSBA} is also an $(\OM(\epsilon), \OM(\kappa^{2.5}\bar{\ell}^{0.5}m^{0.5}\epsilon^{0.5}))$-SOSP of $\varphi(\boldsymbol{x})$ with $T = \Theta\left((\varphi(\fx_0)-\varphi^*)\sqrt{M}\epsilon^{-3/2}\right)$.
It is worth noting that due to the Hessian being updated only every $m$ iterations, the second-order oracle complexities can be bounded by ${\OM}(1+\kappa^{2.5}\bar{\ell}^{0.5}m^{-0.5}\epsilon^{-1.5})$.

 And from inequality \eqref{eq:1_y_2_T}, we can also have
\begin{equation}
\label{eq:upperbound_lazy_st}
\mathcal{L}_\lambda^*\left(\boldsymbol{x}_0\right) -\min\fL_{\lambda}^*(\fx) \geq
\mathcal{L}_\lambda^*\left(\boldsymbol{x}_0\right)-\mathcal{L}_\lambda^*\left(\boldsymbol{x}_T\right)  \geq \frac{M}{720}\sum_{t=0}^{T-1} \left\|\boldsymbol{s}_t^*\right\|^3.
\end{equation}

Connecting the upper bound of $\sum_{t=0}^{T-1}(K^1_t+K^2_t)$  and the upper bound of $\sum_{t=0}^{T-1} \left\|\boldsymbol{s}_t^*\right\|^3$ in equation~\eqref{eq:upperbound_lazy_st}, we have
{
\begin{align*}
&\sum_{t=0}^{T-1}(K^1_t+K^2_t)\\
& \leq 2T + \frac{4 \sqrt{3\kappa} T}{3} \left( \frac{3}{T} \log\left(\frac{\sqrt{3\kappa+1}}{\tilde{\epsilon}}R\right) \right) +\frac{4 \sqrt{3\kappa} T}{3}\log \left(8(3\kappa+1)^{1.5}+\frac{5760 (4\kappa)^3(3\kappa+1)^{1.5}}{T M \tilde{\epsilon}^{ 3}}\Delta\right)\\
& =\mathcal{O}\left(\sqrt{m \bar{\ell}} \kappa^3 \epsilon^{-1.5} \log (\bar{\ell}^{1.5}\kappa^{-3} m^{-1.5} \epsilon^{-4.5}) \right)  =\tilde{\mathcal{O}}\left(\sqrt{m \bar{\ell}} \kappa^3 \epsilon^{-1.5}\right).
\end{align*}
}

The claim follows from the fact that we call gradient oracle for $\mathcal{O}\left(\sum_{t=0}^{T-1}(K^1_t+K^2_t)\right)$ times and perform
Hessian (inverse) and exact cubic sub-problem solver calls for $\mathcal{O}(T)$ times.
\end{proof}

\subsection{The Proof of Theorem~\ref{thm:LMCN}}
\begin{proof}
In the following proof, note that in minimax problems, the Hessian Lipschitz constant of $\varphi(\fx)$ is $\bar{\rho}=4\sqrt{2} \kappa^{3}\rho$.
LMCN and LFSBA follow the same approach in the proof of second-order complexity. Therefore, we only provide a brief explanation and present the necessary formulas.

We state the following facts without proof. Similar to Theorem~\ref{thm:Lazy_Hessian_calls}, we have the following results:

Under Assumption~\ref{assum:minimax}, let $M\geq \bar{\rho}$ and $\bar{\rho}=4\sqrt{2} \kappa^{3} \rho$ and suppose the following condition
\begin{align}
     \!\!\!  \!\! \!\left\|\nabla \varphi^*\left(\boldsymbol{x}_t\right)\!-\!\g(\fx_t;\fy_t)\right\| \!\leq\! \bar{C}_g \epsilon~~\text{and}~~\left\|\nabla^2 \varphi^*(\boldsymbol{x}_{\pi(t)})\!-\!\H(\fx_{\pi(t)};\fy_{\pi(t)})\right\| \!\leq\! \bar{C}_H \sqrt{M \epsilon}
\end{align}
hold with $\bar{C}_g:=1/576, \bar{C}_H:=1/288$ in Algorithm~\ref{alg:LCMN},
then it holds that
\begin{align}
\varphi^*\left(\boldsymbol{x}_{t}\right)
-\varphi^*\left(\boldsymbol{x}_{t+1}\right)
 \geq \gamma(\boldsymbol{x}_{t+1}) +\frac{M}{72} \left\|\fx_{t+1}-\fx_t\right\|^3-\frac{13 \bar{\rho}^3}{M^2}\|\boldsymbol{x}_{\pi(t)}-\boldsymbol{x}_t\|^3,
\end{align}
where we denote $\gamma(\boldsymbol{x}) := \max \left\{\frac{1}{987M^2} \xi\left(\boldsymbol{x}\right)^3 , \frac{1}{120\sqrt{3M}} \left\|\nabla \varphi^*\left(\boldsymbol{x}\right)\right\|^{3 / 2} \right\}$. The above result is the version of Lemma~\ref{lem:pre-LFSBA} for minimax problems and it can be proved using the same arguments.

Under the Assumption \ref{assum:minimax}, let $ \Delta:=\varphi\left(\boldsymbol{x}_0\right)-\varphi^{*}$, 
$\tilde{\epsilon}=\min \left\{\bar{C}_g \epsilon / \ell, \bar{C}_H \sqrt{M \epsilon} / \rho\right\}$, $\bar{C}_g=1/576$ and $\bar{C}_H=1/288$, if we run  Algorithm \ref{alg:LCMN} with
$M = \Omega(m\bar{\rho})$, 
$T = \Theta\left(\Delta\sqrt{M}\epsilon^{-3/2}\right)$,
$\kappa_1 = \kappa$, $\ell_1 = \ell$,
\begin{align*}
 K_t =\begin{cases}
    & \left\lceil 2 \sqrt{\kappa} \log \left(\frac{\sqrt{\kappa+1}}{\tilde{\epsilon}}\left\|\fy^*\left(\boldsymbol{x}_0\right)\right\|\right)\right\rceil~~~~~~~~~~~~~~~~~~~t=0\\
     & \left\lceil 2 \sqrt{\kappa} \log \left(\frac{\sqrt{\kappa+1}}{\tilde{\epsilon}}\left(\tilde{\epsilon}+\kappa\left\|\fx_t-\fx_{t-1}\right\|\right)\right)\right\rceil~~~t\geq 1
 \end{cases},
\end{align*}
and $\bar{\rho}=4\sqrt{2}\kappa^{3}\rho$, then the output $\hat{\boldsymbol{x}}$ of Algorithm \ref{alg:LCMN} is $\left(\epsilon, \kappa^{1.5} \sqrt{m \rho \epsilon}\right)$-SOSP of $\varphi(\boldsymbol{x})$. The first-order and second-order oracle complexities can be bounded by $\tilde{\OM}\left(\kappa^2 \sqrt{m\rho} \epsilon^{-1.5}\right)$ and ${\OM}\left(\sqrt{\rho/m} \kappa^{1.5} \epsilon^{-1.5}\right)$, respectively.
The above result is a simplified version of Theorem \ref{thm:Lazy_Hessian_calls} for minimax problems. 
It can be proved using the same arguments as in Theorem \ref{thm:Lazy_Hessian_calls}, with the only difference being the substitution of the parameter $\bar{\rho}$ and minor modifications in the complexity analysis of the AGD subroutine. The overall proof strategy and technical details remain identical.

The following inequality appears in the proof of the preceding result and serves as an essential intermediate step, similar to Equation~\eqref{eq:upperbound_lazy_st} in Theorem~\ref{thm:Lazy_Hessian_calls}. It will be used in the subsequent analysis:
\begin{equation}
\label{eq:upperbound_lazy_LMCN}
\varphi\left(\boldsymbol{x}_0\right) -\varphi^{*} \geq
\varphi\left(\boldsymbol{x}_0\right)-\varphi\left(\boldsymbol{x}_T\right)  \geq \frac{M}{720}\sum_{t=0}^{T-1} \left\|\boldsymbol{s}_t^*\right\|^3.
\end{equation}

As the AGD part involves different settings, the previous proof does not directly apply; we therefore provide a new analysis of its complexity. We need to explain that the following conditions:
\begin{align}
    \left\|\nabla \varphi\left(\boldsymbol{x}_t\right)-\g(\fx_t;\fy_t)\right\| \leq\! \bar{C}_g \epsilon~~~\text{and}~~~\left\|\nabla^2 \varphi(\boldsymbol{x}_t)-\H(\fx_t;\fy_t)\right\| \leq \bar{C}_H \sqrt{M \epsilon}
\end{align}
can be achieved by properly choosing the number of iterations $K_t$ in the AGD subroutine.

We first use induction to show that 
\begin{equation}
\left\|\boldsymbol{y}_t-\fy^*\left(\boldsymbol{x}_t\right)\right\| \leq \tilde{\epsilon} \label{eq:8}
\end{equation}
holds for any $t \geq 0$. For $t=0$, Lemma \ref{lem:AGD} directly implies $\left\|\boldsymbol{y}_0-\fy^*\left(\boldsymbol{x}_0\right)\right\| \leq \tilde{\epsilon}$. Suppose it holds that $\left\|\boldsymbol{y}_{t-1}-\fy^*\left(\boldsymbol{x}_{t-1}\right)\right\| \leq \tilde{\epsilon}$ for any $t=t^{\prime}-1$, then we have

\begin{align*}
& \left\|\boldsymbol{y}_{t^{\prime}}-\fy^*\left(\boldsymbol{x}_{t^{\prime}}\right)\right\| \\
\leq & \sqrt{\kappa+1}\left(1-\frac{1}{\sqrt{\kappa}}\right)^{K_{t^{\prime}} / 2}\left\|\boldsymbol{y}_{t^{\prime}-1}-\fy^*\left(\boldsymbol{x}_{t^{\prime}}\right)\right\| \\
\leq & \sqrt{\kappa+1}\left(1-\frac{1}{\sqrt{\kappa}}\right)^{K_{t^{\prime}} / 2}\left(\left\|\boldsymbol{y}_{t^{\prime}-1}-\fy^*\left(\boldsymbol{x}_{t^{\prime}-1}\right)\right\|+\left\|\fy^*\left(\boldsymbol{x}_{t^{\prime}-1}\right)-\fy^*\left(\boldsymbol{x}_{t^{\prime}}\right)\right\|\right) \\
\leq & \sqrt{\kappa+1}\left(1-\frac{1}{\sqrt{\kappa}}\right)^{K_{t^{\prime}} / 2}\left(\tilde{\epsilon}+\kappa\left\|\boldsymbol{x}_{t^{\prime}-1}-\boldsymbol{x}_{t^{\prime}}\right\|\right) \\
= & \sqrt{\kappa+1}\left(1-\frac{1}{\sqrt{\kappa}}\right)^{K_{t^{\prime}} /2}\left(\tilde{\epsilon}+\kappa\left\|\boldsymbol{s}_{t^{\prime}-1}^*\right\|\right) \leq \tilde{\epsilon},   
\end{align*}
where the first inequality is based on Lemma \ref{lem:AGD}; the second one use triangle inequality; the third one is based on induction hypothesis and Proposition~\ref{prop:condinum_y_star}; the last step use the definitions of $K_t$ and $\tilde{\epsilon}$.

Combining inequality $\eqref{eq:8}$ with Lemma \ref{lem:AGD}, Assumption \ref{assum:minimax}, we obtain

\begin{align*}
& \left\|\nabla \varphi\left(\boldsymbol{x}_t\right)- \g(\fx_t;\fy_t)\right\| \\ = & \left\|\nabla_x f\left(\boldsymbol{x}_t, \boldsymbol{y}_t\right)-\nabla_x f\left(\boldsymbol{x}_t, \fy^*\left(\boldsymbol{x}_t\right)\right)\right\| \\ \leq & \ell\left\|\boldsymbol{y}_t-\fy^*\left(\boldsymbol{x}_t\right)\right\| \leq \bar{C}_g \epsilon    
\end{align*}
 and
\begin{align*}
& \left\|\nabla^2 \varphi\left(\boldsymbol{x}_t\right)-\H(\fx_t;\fy_t)\right\| \\
= & \left\|\nabla^2 f\left(\boldsymbol{x}_t, \fy^*\left(\boldsymbol{x}_t\right)\right)-\nabla^2 f\left(\boldsymbol{x}_t, \boldsymbol{y}_t\right)\right\| \\
\leq & \rho\left\|\boldsymbol{y}_t-\fy^*\left(\boldsymbol{x}_t\right)\right\| \\
\leq & \bar{C}_H \sqrt{M \epsilon}.    
\end{align*}
 
The total gradient calls from AGD in Algorithm \ref{alg:LCMN} satisfy
{
\begin{align*}
& \sum_{t=0}^{T-1} K_t \\
\leq & 2 \sqrt{\kappa}\left[\log \left(\frac{\sqrt{\kappa+1}}{\tilde{\epsilon}}\left\|y^*\left(\boldsymbol{x}_0\right)\right\|\right)+\sum_{t=1}^T \log \left(\sqrt{\kappa+1}+\frac{\kappa \sqrt{\kappa+1}}{\tilde{\epsilon}}\left\|\boldsymbol{s}_{t-1}^*\right\|\right)\right]+T \\
= & \frac{2 \sqrt{\kappa}}{3}\left[3 \log \left(\frac{\sqrt{\kappa+1}}{\tilde{\epsilon}}\left\|y^*\left(\boldsymbol{x}_0\right)\right\|\right)+\sum_{t=1}^T \log \left(\sqrt{\kappa+1}+\frac{\kappa \sqrt{\kappa+1}}{\tilde{\epsilon}}\left\|\boldsymbol{s}_{t-1}^*\right\|\right)^3\right]+T \\
\leq & \frac{2 \sqrt{\kappa}}{3}\left[3 \log \left(\frac{\sqrt{\kappa+1}}{\tilde{\epsilon}}\left\|y^*\left(\boldsymbol{x}_0\right)\right\|\right)+\sum_{t=1}^T \log \left(8(\kappa+1)^{1.5}+\frac{8 \kappa^3(\kappa+1)^{1.5}}{\tilde{\epsilon}^3}\left\|\boldsymbol{s}_{t-1}^*\right\|^3\right)\right]+T \\
= & \frac{2 \sqrt{\kappa}}{3}\left[3 \log \left(\frac{\sqrt{\kappa+1}}{\tilde{\epsilon}}\left\| y^*\left(\boldsymbol{x}_0\right)\right\|\right)+\log \left(\prod_{t=1}^T\left(8(\kappa+1)^{1.5}+\frac{8 \kappa^3(\kappa+1)^{1.5}}{\tilde{\epsilon}^3}\left\|\boldsymbol{s}_{t-1}^*\right\|_2^3\right)\right)\right]+T \\
\leq & \frac{2 \sqrt{\kappa}}{3}\left[3 \log \left(\frac{\sqrt{\kappa+1}}{\tilde{\epsilon}}\left\|y^*\left(\boldsymbol{x}_0\right)\right\|\right)+\log \left(\frac{1}{T} \sum_{t=1}^T\left(8(\kappa+1)^{1.5}+\frac{8 \kappa^3(\kappa+1)^{1.5}}{\tilde{\epsilon}^3}\left\|\boldsymbol{s}_{t-1}^*\right\|^3\right)\right)^T\right]+T \\
= & \frac{2 \sqrt{\kappa} T}{3}\left[\frac{3}{T} \log \left(\frac{\sqrt{\kappa+1}}{\tilde{\epsilon}}\left\|y^*\left(\boldsymbol{x}_0\right)\right\|\right)+\log \left(8(\kappa+1)^{1.5}+\frac{8 \kappa^3(\kappa+1)^{1.5}}{T \tilde{\epsilon}^3} \sum_{t=1}^T\left\|\boldsymbol{s}_{t-1}^*\right\|^3\right)\right]+T,
\end{align*}
}
where the first inequality is based on the fact $(a+b)^3 \leq 8\left(a^3+b^3\right)$ for $a, b \geq 0$; the second inequality is based on AM–GM inequality.

Here we introduce $\epsilon^{\prime}=2^{-2.5} \epsilon$ to eliminate the constant term $4 \sqrt{2}$ in $M$. Connecting the upper bound of $\sum_{t=0}^{T-1} K_t$ and inequality \eqref{eq:upperbound_lazy_LMCN}, we have

{\small
\begin{align*}
&\sum_{t=0}^{T-1} K_t \\
& \leq T + \frac{2 \sqrt{\kappa} T}{3} \left( \frac{3}{T} \log\left(\frac{\sqrt{\kappa+1}}{\tilde{\epsilon}}\left\|\fy^*\left(\boldsymbol{x}_0\right)\right\|\right) \right)  +\frac{2 \sqrt{\kappa} T}{3}\log \left(8(3\kappa+1)^{1.5}+\frac{5760 \kappa^3(\kappa+1)^{1.5}}{T M \tilde{\epsilon}^{\prime 3}}\Delta\right)\\
& =\tilde{\mathcal{O}}\left(\sqrt{\kappa M} \epsilon^{-1.5}\right) =\tilde{\mathcal{O}}\left(\kappa^2 \sqrt{m\rho} \epsilon^{-1.5}\right).   
\end{align*}
}
\end{proof}

\section{The Details of Inexact Version of FSBA}
\label{app:ifsba}
In this section, we present the details of IFSBA method introduced in Section~\ref{sec:IFSBA}. It is worth emphasizing that IFSBA never explicitly constructs the Hessian ; all Hessian-related operations are carried out via Hessian–vector products, thereby avoiding any second-order oracle calls, matrix factorizations or inversions\cite{chen2022single}, as well as SVD for the projections\cite{huang2024optimal}.
\subsection{Construction of Matrix Chebyshev Polynomials Approximation}
We first present the details of constructing $\C_{1,t}$ and $\C_{2,t}$ such that
\begin{align*}
    \C_{1,t} \approx \left[\nabla_{y y}^2 g\left(\boldsymbol{x}, \boldsymbol{w}\right)\right]^{-1},
    \quad \C_{2,t} \approx    \left[\nabla_{y y}^2 \mathcal{L}_\lambda\left(\boldsymbol{x}, \boldsymbol{y}\right)\right]^{-1}.
\end{align*}

The following lemma
presents the upper bound of approximating the matrix inverse by Chebyshev polynomials.
\begin{lemma}[Section 9.6.1~\citet{axelsson1996iterative}] 
\label{lemma:cheby_approxiamte}
Suppose symmetric matrix $\X \in \mathbb{R}^{d \times d}$ satisfies $\mu' \I \preceq \X \preceq \ell' \I$ with $0 < \mu' \leq \ell' < 1$, then we have
\begin{equation*}
\left\| \X^{-1} - \left( \frac{c_0}{2} \I + \sum_{k=1}^{K'} c_k \T_k(\Z') \right) \right\|
\leq 
\frac{\sqrt{\ell' / \mu'} - 1}{\sqrt{\ell' \mu'}} 
\left( 1 - \frac{2}{\sqrt{\ell' / \mu'} + 1} \right)^{K'},    
\end{equation*}
where 
$\Z' = \frac{2}{\ell' - \mu'} \left( \X - \frac{\ell' + \mu'}{2} \I \right)$,
$c_k = \frac{2}{\sqrt{\ell' \mu'}} \left( \frac{\sqrt{\mu' / \ell'} - 1}{\sqrt{\mu' / \ell'} + 1} \right)^k$ for $k = 0,1,\dots,K'$,
and $\T_k(\cdot)$ are matrix Chebyshev polynomials defined by $\T_0(\Z') := \I$, $\T_1(\Z') = \Z'$, and $\T_k(\Z'): = 2 \Z' \T_{k-1}(\Z') - \T_{k-2}(\Z')$ for $k \geq 2$.
\end{lemma}
Since {$\mu \I \preceq \nabla^2_{yy}g(\fx,\fw)\preceq \ell \I$ and $\frac{\lambda \mu}{2} \I \preceq\nabla^2_{yy}\fL_{\lambda}(\fx,\fy)\preceq (1+\lambda)\ell \I$}, we constructed $\C_{1,t}$ and $\C_{2,t}$ according to
\begin{align}
\label{eq:construct_cheby}
    {\C}_{1,t} = \frac{c_{1,0}}{4\ell} \I + \frac{1}{2\ell} \sum_{k=1}^{K_1'} c_{1,k} \T_k(\Z_{1,t})~~~\text{and}~~~
   {\C}_{2,t} = \frac{c_{2,0}}{4(\lambda+1)\ell} \I + \frac{1}{2(\lambda+1)\ell} \sum_{k=1}^{K_2'} c_{2, k} \T_k(\Z_{2,t}),
\end{align}
where
\begin{align*}
 \Z_{1,t} = \frac{4\ell}{\ell - \mu} \left( \frac{1}{2\ell} \nabla_{yy}^2 g(\boldsymbol{x}_t, \boldsymbol{y}_t) - \frac{\ell + \mu}{4\ell} \I \right),~~~
\Z_{2,t} = \frac{2}{2(\lambda+1)\ell-\lambda\mu} \left( 2 \nabla_{yy}^2 \mathcal{L}_\lambda(\boldsymbol{x}_t, \boldsymbol{y}_t) - ((\lambda+1)\ell+ \frac{\lambda\mu}{2}) \I \right),
\end{align*}
and $\{c_{1,k}\}, 
\{c_{2,k}\}$ computed by
\begin{align*}
  c_{1,k}=\frac{2}{\sqrt{\ell \mu}} \left( \frac{\sqrt{\mu / \ell} - 1}{\sqrt{\mu / \ell} + 1} \right)^k~~~\text{and}~~~c_{2,k}= \frac{2}{\sqrt{(1+\lambda)\ell\lambda \mu/2}} \left( \frac{\sqrt{\frac{\lambda \mu}{2(1+\lambda)\ell} } - 1}{\sqrt{\frac{\lambda \mu}{2(1+\lambda) \ell } } + 1} \right)^k.  
\end{align*}

Then, we are able to bound the difference between $\C(\fx_t;\fy_t,\fw_t)$ and $\nabla^2 \fL_{\lambda}^*(\fx_t)$ by combining the statements of Lemma~\ref{lem:estimators_error} and Lemma~\ref{lemma:cheby_approxiamte}.
\begin{lemma} \label{lemma:cheyb_error_bound}
Using the notation of Algorithm~\ref{alg:IFSBA}, under {Assumption~\ref{assum:basic}}, we have
\begin{align*}
&\left\|{\nabla}^2 \mathcal{L}_\lambda^*\left(\boldsymbol{x_t}\right)-\C(\boldsymbol{x}_t;\boldsymbol{y}_t,\boldsymbol{w}_t)\right\|\\
&\leq 
C_1 \left\|\boldsymbol{w}_t-\boldsymbol{y}^*\left(\boldsymbol{x}_t\right) \right\| + C_2
 \left\|\boldsymbol{y}_t -\boldsymbol{y}_\lambda^*\left(\boldsymbol{x}_t\right)\right\|+ \kappa \ell \left( 1 - \frac{2}{\sqrt{\kappa} + 1} \right)^{K_{1}'}+
6
(\lambda+1)\kappa \ell \left( 1 - \frac{2}{\sqrt{3\kappa} + 1} \right)^{K_{2}'}.
\end{align*}
\end{lemma}
\subsection{Gradient-Based Subproblem Solver}
In this section, we formally present the subroutines Cubic-Solver and Final-Cubic-Solver (line 7 and line 10 in Algorithm~\ref{alg:IFSBA}) to solve the following cubic-regularized problem
\begin{equation}
\label{eq:cubic}
\boldsymbol{s}_t \approx \argmin_{\boldsymbol{s} \in \mathbb{R}^{d_x}} m_t(\boldsymbol{s}) := \boldsymbol{g}_t^\top \boldsymbol{s} + \frac{1}{2} \boldsymbol{s}^\top \C_t \boldsymbol{s} + \frac{M}{6} \|\boldsymbol{s}\|^3.
\end{equation}
We introduce the Cubic-Solver and Final-Cubic-Solver in Algorithm~\ref{alg:cubic_sub} and Algorithm~\ref{alg:final_cubic}, respectively. 
Cubic-Solver constructs gradient-based update to approximately solve~\eqref{eq:cubic} with desired accuracy in high probability.
When $\Delta_t\geq -\frac{\epsilon^3}{128M}$, we run Final-Cubic-Solver to guarantee that the output $\fx_{t+1}$ is an $(\epsilon,\OM(\sqrt{\epsilon}))$ SOSP of $\fL^*_{\lambda}(\cdot)$.

\begin{algorithm}
\caption{Cubic-Solver($\g,\H,\sigma, \mathcal{K}(\epsilon,\delta')$)}
\label{alg:cubic_sub}
\begin{algorithmic}[1]
\STATE \textbf{Input:} $\g, \H, \sigma, \mathcal{K}(\epsilon, \delta')$
\vskip 0.2cm
\IF{$\|\g\| \geq L^2 / M$}
    \STATE $R_C = -\dfrac{\g^\top \H \g}{M \|\g\|^2} + \sqrt{(\frac{\g^\top \H \g}{M \|\g\|^2})^2+\frac{2\|\g\|}{M}}$
    \STATE $\hat{\boldsymbol{s}} = -R_C \cdot \g / \|\g\|$
\ELSE
    \STATE $\boldsymbol{s}_0 = \mathbf{0},\ \eta = 1/(20L)$
    \STATE $\tilde{\g} = \g + \sigma \boldsymbol{\zeta},\ \text{where}\ \boldsymbol{\zeta} \sim \operatorname{Uniform}(\mathcal{S}^{d-1})$
    \FOR{$k = 0, 1, \cdots, \mathcal{K}(\epsilon, \delta')-1$}
        \STATE $\boldsymbol{s}_{k+1} = \boldsymbol{s}_k - \eta (\tilde{\g} + \H \boldsymbol{s}_k + \dfrac{M}{2} \|\boldsymbol{s}_k\| \boldsymbol{s}_k)$
    \ENDFOR
    \STATE $\hat{\boldsymbol{s}} = \boldsymbol{s}_{\mathcal{K}(\epsilon, \delta')}$
\ENDIF
\vskip 0.2cm
\STATE \textbf{Output:} $\hat{\boldsymbol{s}}$ and $\Delta = \g^\top \hat{\boldsymbol{s}} + \dfrac{1}{2} \hat{\boldsymbol{s}}^\top \H \hat{\boldsymbol{s}} + \dfrac{M}{6} \|\hat{\boldsymbol{s}}\|^3$
\end{algorithmic}
\end{algorithm}

\begin{algorithm}
\caption{Final-Cubic-Solver}
\label{alg:final_cubic}
\begin{algorithmic}[1]
\STATE \textbf{Input:} $\g, \H, \epsilon$
\vskip 0.2cm
\STATE $\boldsymbol{s}_0 = \mathbf{0},\ \g_0 = \g,\ \eta = 1/(20L)$
\FOR{$t = 0, 1, \cdots$}
    \IF{$\|\g_t\| \leq \epsilon / 2$}
        \STATE \textbf{break}
    \ENDIF
    \STATE $\boldsymbol{s}_{t+1} = \boldsymbol{s}_t - \eta \g_t$
    \STATE $\g_{t+1} = \g + \H \boldsymbol{s}_{t+1} + \dfrac{M}{2} \|\boldsymbol{s}_{t+1}\|\boldsymbol{s}_{t+1}$
\ENDFOR
\vskip 0.2cm
\STATE \textbf{Output:} $\hat{\boldsymbol{s}} = \boldsymbol{s}_t$
\end{algorithmic}
\end{algorithm}


\subsection{The Convergence Analysis}
We provide the convergence analysis for IFSBA (Algorithm~\ref{alg:IFSBA}), following the same assumptions and notations as those used in section~\ref{sec:FSBA}. 
We suppose $\epsilon \leq \frac{L^2}{M}$, otherwise, the second-order condition $\nabla^2 \mathcal{L}_\lambda^*(\boldsymbol{x}_t) \succeq -\sqrt{M\epsilon}\I$ always holds and we only need to use gradient methods to find first-order stationary point.

The following lemma indicates that once $\g(\fx_t;\fy_t,\fw_t)$ and $\C(\fx_t;\fy_t,\fw_t)$ approximate $\nabla\fL_{\lambda}^*(\fx_t)$ and $\nabla^2 \fL_{\lambda}^*(\fx_t)$ well and Cubic-Solver iterates with sufficient steps, then IFSBA enjoys a similar iteration complexity as FSBA with high probability.
\begin{lemma}[Theorem 3, \citet{Luo22finding}]
\label{lem:IFSBA_convergence}
Under Assumption~\ref{assum:basic}, if we
run Algorithm~\ref{alg:IFSBA} with $\delta' = \delta / T$, 
$T = \left\lceil626 \left(\fL_{\lambda}^*(\fx_0)-\min_{\fx}\fL_{\lambda}^*(\fx)\right)\sqrt{M} \, \epsilon^{-1.5} \right\rceil$, 
and suppose the iterations $K_t^1$, $K_t^2$ of AGD and the order $K_1'$, $K_2'$ of Chebyshev-Polynomials in \eqref{eq:construct_cheby} are sufficiently large such that the following condition
\begin{align}
\label{eq:inexact-cubic-iteration-cheby}
\begin{aligned}
\left\|\nabla \mathcal{L}_\lambda^*\!\left(\boldsymbol{x}_t\right)
      - \g(\fx_t;\fy_t,\fw_t)\right\|
\le{ \tilde{C}_g} \epsilon,~~~~
\left\|\nabla^2 \mathcal{L}_\lambda^*\!\left(\boldsymbol{x}_t\right)
      - \C(\fx_t;\fy_t,\fw_t)\right\|
\le \tilde{C}_H \sqrt{M\epsilon},
\end{aligned}
\end{align}
hold with $\tilde{C}_g = 1/240$, $\tilde{C}_H = 1/200$, and the hyperparameters of Cubic-Solver (Algorithm~\ref{alg:cubic_sub}) satisfies that
{\begin{align*}
 \eta\! = \!\frac{1}{20L}, \sigma\! = \!\frac{C_{\sigma}M^2\sqrt{\epsilon^3/M^3}}{4608(4L\!+\!\sqrt{M\epsilon})}, \mathcal{K}(\epsilon, \delta')\!=\! \frac{19200L}{C_{\sigma}\sqrt{M\epsilon}} \!(6\!\log(\!3\!+\!\frac{9\sqrt{d_x}}{\delta'}\!)
 \!+\! 18\log(\frac{6L}{\sqrt{M\epsilon}})\!+\! 14 \log(\frac{48(L\! + \!\tilde{C}_H \sqrt{M\epsilon})}{C_{\sigma}\sqrt{M\epsilon}}\!+\!\frac{24}{C_{\sigma}})
 )\!
\end{align*}
}
for some $C_{\sigma}>0$, then the condition $\Delta_t\geq -\frac{1}{128}\frac{\epsilon^3}{M}$ must hold within no more than 
$T = \OM\left( \kappa^{2.5} \sqrt{\bar{\ell} }\, \epsilon^{-1.5} \right)$ iterations; 
and the output $\hat{\boldsymbol{x}}$ is an $(\epsilon, \OM(\kappa^{2.5} \sqrt{\bar{{\ell}}\epsilon} ))$-SOSP of $\fL^*_{\lambda}(\cdot)$
with probability at least $1 - \delta$.
\end{lemma}

We provide the following lemma to satisfy the condition~\eqref{eq:inexact-cubic-iteration-cheby}.

\begin{lemma}
\label{lem:IFSBA_inner_complexity}
Under Assumption~\ref{assum:basic}, let $\epsilon_{H}>0$, $C_2 = \OM(\lambda\bar{\ell}\kappa^2)$, $\tilde{\epsilon} = \min\left\{ \frac{\tilde{C}_g \epsilon}{2\lambda\ell},\; \frac{\min\{ \tilde{C}_H \sqrt{M\epsilon} , \epsilon_{\mathrm{H}}L \}}{4 {C_2}} \right\}$, $R=\max(\left\|\fy^*\left(\boldsymbol{x}_0\right)\right\|, \left\|\fy_\lambda^*\left(\boldsymbol{x}_0\right)\right\|)$, and $\Delta=\varphi\left(\boldsymbol{x}_0\right) -\varphi^{*}$. 
if we run Algorithm~\ref{alg:IFSBA} with the same settings as in Lemma~\ref{lem:IFSBA_convergence} and $\lambda = \max \left\{ \bar{\ell} \kappa^2 / \Delta, \bar{\ell}\kappa^3 / \epsilon, \bar{\ell}\kappa^5 / \sqrt{M\epsilon}\right\}$, $\kappa_1 = \kappa$, $\ell_1 = \ell$, $\kappa_2 = 3\kappa$, $\ell_2 = (1+\lambda)\ell$,
the order $K_1', K_2'$ of Chebyshev Polynomials in \eqref{eq:construct_cheby} is
\begin{equation*}
K_1'=K_2' = \frac{\sqrt{3\kappa} + 1}{2} \log\left( \frac{24(\lambda+1)\kappa \ell}{ \min\{\tilde{ C}_H \sqrt{M\epsilon} , \epsilon_{\mathrm{H}} L\}} \right)   
\end{equation*}
and the number of iterations of AGD subroutine as 
\begin{align*}
 K_t^1 = K_t^2 =\begin{cases}
    & \left\lceil 2 \sqrt{3\kappa} \log \left(\frac{\sqrt{3\kappa+1}}{\tilde{\epsilon}}R\right)\right\rceil~~~~~~~~~~~~~~~~~~~~~~~~~~~~~~~~~~t=0 \\[0.3cm]
     & \left\lceil 2 \sqrt{3\kappa} \log \left(\frac{\sqrt{3\kappa+1}}{\tilde{\epsilon}}\left(\tilde{\epsilon}+4\kappa\left\|\boldsymbol{x}_t-\boldsymbol{x}_{t-1}\right\|\right)\right)\right\rceil~~~t\geq 1
 \end{cases},
\end{align*}
then the condition \eqref{eq:inexact-cubic-iteration-cheby} in Lemma~\ref{lem:IFSBA_convergence} is satisfied. 

Note that the value of $K'_1=K'_2 = O(\sqrt{\kappa})$ corresponds to the number of Hessian-vector product calls per iteration of the cubic subproblem solver (Algorithms~\ref{alg:cubic_sub} and~\ref{alg:final_cubic}). Combining Lemma~\ref{lem:IFSBA_convergence}, Lemma~\ref{lem:IFSBA_inner_complexity}, and the value of $\mathcal{K}(\epsilon, \delta')$, we obtain the main result for Algorithm~\ref{alg:IFSBA} as follows.
\end{lemma}

\begin{theorem}
\label{thm:imcn_main_complexity}
Under Assumption~\ref{assum:basic}, run Algorithm~\ref{alg:IFSBA} under the same setting as in Lemma~\ref{lem:IFSBA_convergence} and \ref{lem:IFSBA_inner_complexity}, let $M=\Omega( \bar{\rho})$, 
$T = \Theta((\varphi(\fx_0)-\varphi^*)\sqrt{M}\epsilon^{-3/2})$, then
$\hat{\fx}$ is an $((\OM(\epsilon),\OM(\kappa^{2.5}\bar{\ell}^{0.5}\epsilon^{0.5}))$-SOSP of $\varphi(\cdot)$ with probability at least $1 - \delta$. In addition, the total number of $K^1_t$, $K^2_t$ can be bounded by
$  \sum_{t=0}^{T-1}(K^1_t+K^2_t) \leq \tilde{\OM}(\kappa^3\bar{\ell}^{0.5}\epsilon^{-1.5}).$
The complexities of the gradient calls and Hessian-vector product calls can be bounded by $\tilde{\OM}(\kappa^3\bar{\ell}^{0.5}\epsilon^{-1.5})$ and ${\OM}(\kappa^{3.5}\bar{\ell}\epsilon^{-2})$, respectively.
\end{theorem}

\subsection{The Proof of Lemma~\ref{lemma:cheyb_error_bound}}
\begin{proof}
Recalling that $\mu \I \preceq \nabla_{yy}^2 g(\boldsymbol{x}_t, \boldsymbol{y}_t) \preceq \ell \I$, we estimate the inverse of the Hessian with respect to $\fy$ as
\begin{equation*}
\left( \frac{1}{2\ell} \nabla_{yy}^2 g(\boldsymbol{x}_t, \boldsymbol{y}_t) \right)^{-1}
\approx 
\frac{c_{1,0}}{2} \I + \sum_{k=1}^{K_{1}'} c_{1,k} \T_k(\Z_{1,t}).  
\end{equation*}

Lemma~\ref{lemma:cheby_approxiamte} implies
\begin{equation*}
\left\| 
\left( \frac{1}{2\ell} \nabla_{yy}^2 g(\boldsymbol{x}_t, \boldsymbol{y}_t) \right)^{-1}
- 
\left( \frac{c_{1,0}}{2} \I + \sum_{k=1}^{K_{1}'} c_{1,k} \T_k(\Z_{1,t}) \right)
\right\|
\leq 
2(\kappa - \sqrt{\kappa}) \left( 1 - \frac{2}{\sqrt{\kappa} + 1} \right)^{K_{1}'}.
\end{equation*}

Hence, we have
\begin{align*}
&\left\| 
\nabla_{xy}^2 g(\boldsymbol{x}_t, \boldsymbol{y}_t) \left( \nabla_{yy}^2 g(\boldsymbol{x}_t, \boldsymbol{y}_t) \right)^{-1} \nabla_{yx}^2 g(\boldsymbol{x}_t, \boldsymbol{y}_t)
- 
\nabla_{xy}^2 g(\boldsymbol{x}_t, \boldsymbol{y}_t) \boldsymbol{C}_{1,t} \nabla_{yx}^2 g(\boldsymbol{x}_t, \boldsymbol{y}_t)
\right\| \\
&\quad \leq 
\left\| \nabla_{xy}^2 g(\boldsymbol{x}_t, \boldsymbol{y}_t) \right\|^2
\left\| 
\left( \nabla_{yy}^2 g(\boldsymbol{x}_t, \boldsymbol{y}_t) \right)^{-1}
- 
\left( \frac{c_{1,0}}{4\ell} \I + \frac{1}{2\ell} \sum_{k=1}^{K_{1}'} c_{1,k} \T_k(\Z_{1,t}) \right)
\right\| \\
&\quad \leq 
\ell  (\kappa - \sqrt{\kappa}) \left( 1 - \frac{2}{\sqrt{\kappa} + 1} \right)^{K_{1}'} \\
&\quad \leq 
\kappa \ell \left( 1 - \frac{2}{\sqrt{\kappa} + 1} \right)^{K_{1}'}.
\end{align*}

Similarily, we have $\frac{\lambda \mu}{2} \I\preceq \nabla_{y y}^2 \mathcal{L}_\lambda\left(\boldsymbol{x}, \boldsymbol{y}\right) \preceq (1+\lambda)\ell\I$. We estimate the inverse of the Hessian with respect to $\fy$ as 
\begin{equation*}
\left[\frac{1}{2(\lambda+1)\ell}\nabla_{y y}^2 \mathcal{L}_\lambda\left(\boldsymbol{x}_t, \boldsymbol{y}_t\right)\right]^{-1}
\approx 
\frac{c_{2,0}}{2} \I + \sum_{k=1}^{K_{2}'} c_{2,k} \T_k(\Z_{2,t}).   
\end{equation*}

Lemma~\ref{lemma:cheby_approxiamte} implies
{
\begin{equation*}
\left\| \left[\frac{1}{2(\lambda+1)\ell}\nabla_{y y}^2 \mathcal{L}_\lambda\left(\boldsymbol{x}_t, \boldsymbol{y}_t\right)\right]^{-1}
- 
\left( \frac{c_{2,0}}{2} \I + \sum_{k=1}^{K_{2}'} c_{2,k} \T_k(\boldsymbol{Z}_{2,t}) \right)
\right\|
\leq 
2(3\kappa - \sqrt{3\kappa}) \left( 1 - \frac{2}{\sqrt{3\kappa} + 1} \right)^{K_{2}'}.   
\end{equation*}
}

Hence, we have

\begin{align*}
&\left\| 
\nabla_{x y}^2 \mathcal{L}_\lambda\left(\boldsymbol{x}_t, \boldsymbol{y}_t\right)\left[\nabla_{y y}^2 \mathcal{L}_\lambda\left(\boldsymbol{x}_t, \boldsymbol{y}_t\right)\right]^{-1}\nabla_{y x}^2 \mathcal{L}_\lambda\left(\boldsymbol{x}_t, \boldsymbol{y}_t\right)
- 
\nabla_{x y}^2 \mathcal{L}_\lambda\left(\boldsymbol{x}_t, \boldsymbol{y}_t\right) \boldsymbol{C}_{2,t} \nabla_{y x}^2 \mathcal{L}_\lambda\left(\boldsymbol{x}_t, \boldsymbol{y}_t\right)
\right\| \\
& \leq 
\left\| \nabla_{x y}^2 \mathcal{L}_\lambda\left(\boldsymbol{x}_t, \boldsymbol{y}_t\right)\right\|^2
\left\| 
\left( \nabla_{y y}^2 \mathcal{L}_\lambda\left(\boldsymbol{x}_t, \boldsymbol{y}_t\right) \right)^{-1}
- 
\left( \frac{c_{2,0}}{4(\lambda+1)\ell} \I + \frac{1}{2(\lambda+1)\ell} \sum_{k=1}^{K_{2}'} c_{2,k} \T_k(\Z_{2,t}) \right)
\right\| \\
& \leq 2(\lambda+1)
\ell  (3\kappa - \sqrt{3\kappa})\left( 1 - \frac{2}{\sqrt{3\kappa} + 1} \right)^{K_{2}'} \\
&\leq 6
(\lambda+1)\kappa \ell \left( 1 - \frac{2}{\sqrt{3\kappa} + 1} \right)^{K_{2}'}.
\end{align*}

Then, we have
\begin{equation*}
\left\|\H(\boldsymbol{x}_t;\boldsymbol{y}_t,\boldsymbol{w}_t)-\C(\boldsymbol{x}_t;\boldsymbol{y}_t,\boldsymbol{w}_t)\right\| \leq \kappa \ell \left( 1 - \frac{2}{\sqrt{\kappa} + 1} \right)^{K_{1}'}+
6
(\lambda+1)\kappa \ell \left( 1 - \frac{2}{\sqrt{3\kappa} + 1} \right)^{K_{2}'}  .
\end{equation*}

According to
$\left\|{\nabla}^2 \mathcal{L}_\lambda^*\left(\boldsymbol{x_t}\right)-\H(\boldsymbol{x}_t;\boldsymbol{y}_t,\boldsymbol{w}_t)\right\| \leq C_1 \left\|\boldsymbol{w}_t-\boldsymbol{y}^*\left(\boldsymbol{x}_t\right) \right\| + C_2
 \left\|\boldsymbol{y}_t -\boldsymbol{y}_\lambda^*\left(\boldsymbol{x}_t\right)\right\|$ for $C_1 := \mathcal{\OM}\left(\lambda \bar{\ell} +\bar{\ell} \kappa^2\right)$ and $C_2 := \mathcal{\OM}\left( \lambda\bar{\ell} \kappa^2 \right)$.
We obtain $\boldsymbol{w}_t \approx \fy^*(\boldsymbol{x}_t)$ and $\boldsymbol{y}_t \approx \fy_{\lambda}^*(\boldsymbol{x}_t)$ by AGD. Then we can bound the approximation error of $\C(\boldsymbol{x}_t;\boldsymbol{y}_t,\boldsymbol{w}_t)$ as follows:
\begin{align*}
&\left\|{\nabla}^2 \mathcal{L}_\lambda^*\left(\boldsymbol{x_t}\right)-\C(\boldsymbol{x}_t;\boldsymbol{y}_t,\boldsymbol{w}_t)\right\|\\
&\leq 
\left\| 
{\nabla}^2 \mathcal{L}_\lambda^*\left(\boldsymbol{x_t}\right)-\H(\boldsymbol{x}_t;\boldsymbol{y}_t,\boldsymbol{w}_t)
\right\| 
+
\left\| 
\H(\boldsymbol{x}_t;\boldsymbol{y}_t,\boldsymbol{w}_t)-\C(\boldsymbol{x}_t;\boldsymbol{y}_t,\boldsymbol{w}_t)
\right\| \\
&\leq 
C_1 \left\|\boldsymbol{w}_t-\y^*\left(\boldsymbol{x}_t\right) \right\| + C_2
 \left\|\boldsymbol{y}_t -\boldsymbol{y}_\lambda^*\left(\boldsymbol{x}_t\right)\right\|
+ \kappa \ell \left( 1 - \frac{2}{\sqrt{\kappa} + 1} \right)^{K_{1}'}+
6
(\lambda+1)\kappa \ell \left( 1 - \frac{2}{\sqrt{3\kappa} + 1} \right)^{K_{2}'} .
\end{align*}
\end{proof}

\subsection{The Proof of Lemma
~\ref{lem:IFSBA_inner_complexity}}
\begin{proof}
Since FSBA and IFSBA share the same procedure and analysis in the AGD part, the following result can be derived in the same manner as Lemma~\ref{lem:exact_enough}:
\begin{equation}
\label{eq:error_y_w_1}
\left\|\boldsymbol{y}_t-\fy_\lambda^*\left(\boldsymbol{x}_t\right)\right\| \leq \tilde{\epsilon}, \left\|\boldsymbol{w}_t-\fy^*\left(\boldsymbol{x}_t\right)\right\| \leq \tilde{\epsilon}
\end{equation} 
holds for any $t \geq 0$. 
Combining inequality \eqref{eq:error_y_w_1} with Lemma \ref{lemma:cheyb_error_bound}, together with the definition of $K_1'$ and $K_2'$, we obtain
\begin{align*}
\left\|\nabla \mathcal{L}_\lambda^*\left(\boldsymbol{x}_t\right)-\g(\fx_t;\fy_t,\fw_t)\right\|
&\leq 2 \lambda \ell \left\|\boldsymbol{y}_t - \fy_\lambda^*\left(\boldsymbol{x}_t\right)\right\|
+ \lambda \ell\left\|\boldsymbol{w}_t - \fy^*\left(\boldsymbol{x}_t\right)\right\| \leq \tilde{C}_g \epsilon, \\
\left\|\nabla^2 \mathcal{L}_\lambda^*(\boldsymbol{x}_t)-\C(\fx_t;\fy_t,\fw_t)\right\|
&\leq C_1 \left\|\boldsymbol{w}_t-\fy^*\left(\boldsymbol{x}_t\right) \right\| + C_2
 \left\|\boldsymbol{y}_t-\boldsymbol{y}_\lambda^*\left(\boldsymbol{x}_t\right)\right\|\\
&+ \kappa \ell \left( 1 - \frac{2}{\sqrt{\kappa} + 1} \right)^{K_{1}'}+
6
(\lambda+1)\kappa \ell \left( 1 - \frac{2}{\sqrt{3\kappa} + 1} \right)^{K_{2}'} \\
&\leq \min\{ \tilde{C}_H  \sqrt{M\epsilon}, \epsilon_{\mathrm{H}}L \}.
\end{align*}
From the above results, it can be observed that the condition \eqref{eq:inexact-cubic-iteration-cheby} in Lemma~\ref{lem:IFSBA_convergence} is satisfied. 
\end{proof}

\subsection{The Proof of Theorem~\ref{thm:imcn_main_complexity}}
\begin{proof}
Let $M=\Omega( \bar{\rho})$,
$T = \Theta\left((\varphi(\fx_0)-\varphi^*)\sqrt{M}\epsilon^{-3/2}\right)$ and the setting of $\lambda$,
then we can prove that the output
$\hat{\fx}$ of Algorithm~\ref{alg:IFSBA} is an $\left((\OM(\epsilon),\OM(\kappa^{2.5}\bar{\ell}^{0.5}\epsilon^{0.5})\right)$-SOSP of $\varphi(\cdot)$.

Since the algorithm~\ref{alg:IFSBA} could find an $(\epsilon, \sqrt{M \epsilon})$-SOSP of $\mathcal{L}_\lambda^*(\boldsymbol{x})$ in Lemma~\ref{lem:IFSBA_convergence},
then we have
\begin{equation*}
\left\|\nabla \mathcal{L}_\lambda^*\left(\boldsymbol{x}\right)\right\| \leq \epsilon, \quad
\nabla^2 \mathcal{L}_\lambda^*\left(\boldsymbol{x}\right) \succeq -\sqrt{M \epsilon} I. 
\end{equation*}

Following the analysis in Theorem~\ref{thm:FSBA} with the setting of $\lambda$ and Proposition~\ref{prop:proxy_property}
, we can show the following results:
\begin{equation*}
\left\|\nabla \varphi(\boldsymbol{x})\right\|
\le \mathcal{O}(\epsilon), \quad
\nabla^2 \varphi(\boldsymbol{x})
\succeq -\mathcal{O}(\sqrt{M\epsilon})\I, \quad
\mathcal{L}_\lambda^*(\boldsymbol{x}_0)
- \min_{\boldsymbol{x}} \mathcal{L}_\lambda^*(\boldsymbol{x})
= \mathcal{O}(\Delta).
\end{equation*}

Since FSBA and IFSBA share the same structure in the AGD component, we can analogously to Lemma~\ref{lem:exact_enough} establish the corresponding first-order oracle complexity:
{
\begin{align*}
& \sum_{t=0}^{T-1}(K^1_t+K^2_t) \\
= & \frac{4 \sqrt{3\kappa} T}{3}\left[\frac{3}{T} \log \left(\frac{\sqrt{3\kappa+1}}{\tilde{\epsilon}}R\right)+\log \left(8(3\kappa+1)^{1.5}+\frac{8 (4\kappa)^3(3\kappa+1)^{1.5}}{T \tilde{\epsilon}^3} \sum_{t=1}^T\left\|\boldsymbol{s}_{t-1}\right\|^3\right)\right]+2T.  
\end{align*}
}

Our Lemma~\ref{lem:IFSBA_convergence} corresponds to Theorem 3 in \citet{Luo22finding}. Under the same setting and within the proof of that theorem, the following lemma (Lemma 16 in in \citet{Luo22finding}) was employed:

Under the setting of Lemma~\ref{lem:IFSBA_convergence}, if it satisfies 
$\Delta_t \le -\frac{1}{128}\sqrt{\frac{\epsilon^3}{M}}$,
then we have
\begin{equation}
\frac{M}{256}\|\boldsymbol{s}_t\|^3 \le \mathcal{L}_\lambda^*\left(\boldsymbol{x}_t\right) - \mathcal{L}_\lambda^*\left(\boldsymbol{x}_t+\boldsymbol{s}_t\right) - \frac{1}{626}\sqrt{\frac{\epsilon^3}{M}},   
\end{equation}
with probability at least $1 - \delta'$.

Based on the above lemma, we conclude the total number of AGD calls is at most
{
\begin{align*}
&\sum_{t=0}^{T-1}(K^1_t+K^2_t) \\
& \leq \frac{2 \sqrt{3\kappa} T}{3} \left( \frac{3}{T} \log\left(\frac{\sqrt{3\kappa+1}}{\tilde{\epsilon}}R\right) \right)  +\frac{2 \sqrt{3\kappa} T}{3}\log \left(8(3\kappa+1)^{1.5}+\frac{2048 (4\kappa)^3(3\kappa+1)^{1.5}}{T M \tilde{\epsilon}^{ 3}}\Delta\right) + 2T\\
& =\tilde{\mathcal{O}}\left(\sqrt{\kappa M} \epsilon^{-1.5}\right) =
\tilde{O}\left( \kappa^3 \sqrt{\bar{\ell}} \, \epsilon^{-1.5} \right).
\end{align*}
}

The total number of Hessian-vector calls from Algorithm~\ref{alg:IFSBA} is at most
\begin{align*}
T \cdot \mathcal{K}(\epsilon, \delta') \cdot( K_1'+K_2') 
&\leq
\tilde{O}\left( \kappa^{2.5} \sqrt{\bar{\ell}} \epsilon^{-1.5} \right) \cdot \tilde{O}\left( \frac{L }{\sqrt{M\epsilon}} \right)\cdot \tilde{O}(\sqrt{\kappa}) \\
&\leq \tilde{O}\left( \kappa^{2.5} \sqrt{\bar{\ell}} \epsilon^{-1.5} \right) \cdot\tilde{O}\left( \sqrt{\frac{\bar{\ell}\kappa}{\epsilon}} \right) \cdot \tilde{O}(\sqrt{\kappa})\\
& \leq \tilde{O}\left( \bar{\ell}\kappa^{3.5} \epsilon^{-2} \right)
\end{align*}
Using Lemma~8 of \citet{tripuraneni2018stochastic}, we know the total number of Hessian-vector calls from
Algorithm~\ref{alg:final_cubic} is at most
\begin{equation*}
\tilde{\mathcal{O}}(\sqrt{\kappa}) \cdot \mathcal{O}(\frac{L^2}{M\epsilon} )
= \tilde{\mathcal{O}}\!\left(\frac{\bar{\ell} \kappa}{\epsilon}\right),   
\end{equation*}
which is not the leading term in total complexity for small $\epsilon$.
\end{proof}

\section{Experiment Details}
\label{sec:exp_detail}
We present the additional experiment details in this section.

\paragraph{Code Availability.}
The code is available at \url{https://github.com/silas-yang9/FSBA}.

\subsection{Experimental Setup for Section~\ref{sec:exp}}
Our experiments are carried out on a server equipped with an Intel Xeon Platinum 8352V CPU @ 2.10GHz, featuring 16 vCPUs and 120GB of memory. The GPU used is an NVIDIA RTX 4090 (24GB VRAM). We implement the algorithms using PyTorch 2.5.1 and Python 3.12, with GPU acceleration supported by CUDA 12.4. The operating system is Ubuntu 22.04.

\subsection{Experiment Details in Section~\ref{sec:syn}}
\label{app:expminimax}
The $w(\cdot)$ in the objective function $f(\x,\y)$ is defined to be
\begin{align*}
    w(x)= \begin{cases}\sqrt{\epsilon}(x+(L+1) \sqrt{\epsilon})^2-\frac{1}{3}(x+(L+1) \sqrt{\epsilon})^3-\frac{1}{3}(3 L+1) \epsilon^{3 / 2}, & x \leq-L \sqrt{\epsilon} ; \\ \epsilon x+\frac{\epsilon^{3 / 2}}{3}, & -L \sqrt{\epsilon}<x \leq-\sqrt{\epsilon} ; \\ -\sqrt{\epsilon} x^2-\frac{x^3}{3}, & -\sqrt{\epsilon}<x \leq 0 ; \\ -\sqrt{\epsilon} x^2+\frac{x^3}{3}, & 0<x \leq \sqrt{\epsilon} ; \\ -\epsilon x+\frac{\epsilon^{3 / 2}}{3}, & \sqrt{\epsilon}<x \leq L \sqrt{\epsilon} ; \\ \sqrt{\epsilon}(x-(L+1) \sqrt{\epsilon})^2+\frac{1}{3}(x-(L+1) \sqrt{\epsilon})^3-\frac{1}{3}(3 L+1) \epsilon^{3 / 2}, & L \sqrt{\epsilon}<x ;\end{cases} .
\end{align*}

The hyperparameters in all methods are tuned as follows.
We perform a grid search to tune the learning rates for the AGD steps, GDA, and the outer loop of PRAGDA from $\left\{c \times 10^i: c \in\{1,5\}, i \in\{1,2,3\}\right\}$. 
The momentum parameter is selected from $\{c \times 0.1: c \in\{1,2,3,4,5,6,7,8,9\}\}$. 
For LMCN, the frequency of Hessian updates $m$ is chosen from $\{1,2,3,4,5,6,7,8,9,10\}$. For LMCN, iMCN and MCN, we set the cubic regularized Newton constant $M=5$.

\begin{figure}[htbp]
    \centering
    
    \setlength{\tabcolsep}{1pt} 
    
    \begin{tabular}{cc}
        
        \includegraphics[width=0.38\linewidth]{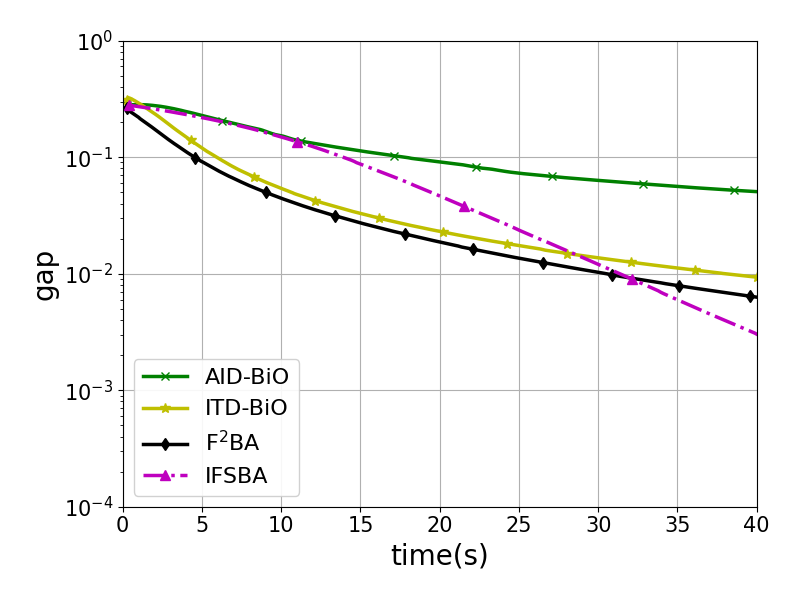} &  
        
        \includegraphics[width=0.38\linewidth]{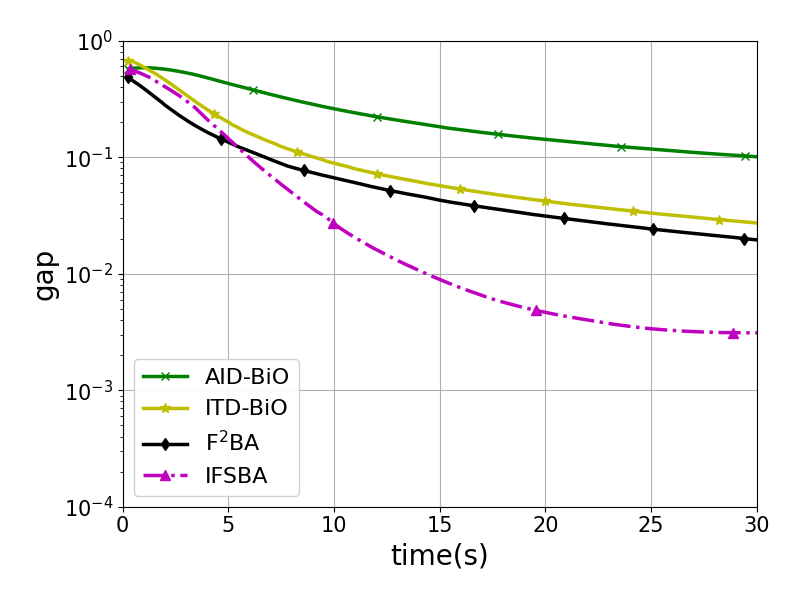} \\

        \scriptsize MNIST $p=0.25$ & 
        \scriptsize MNIST $p=0.5$
    \end{tabular}

    \caption{Comparison of various bilevel algorithms for data hypercleaning at different noise rate $p$ on "MNIST" datasets.}
    \label{fig:MNIST}
\end{figure}

\subsection{Experiment Details in Section~\ref{sec:hypercleaning}}
\label{app:expdatacleaning}
{We tune the inner-loop and outer-loop learning rate of all methods from $\left\{10^{-3}, 10^{-2}, 10^{-1}, 1, 10^{1}, 10^{2}, 10^{3}\right\}$, the iteration numbers of GD or AGD step from $\{5, 10, 30, 50\}$, and the iteration number of CG step from $\{5, 10, 30, 50\}$.
For F$^2$BA and LFSBA, we additionally tune the multiplier $\lambda$ in $\left\{1,10^1, 10^2, 10^3\right\}$. 
For LFSBA, we tune $M$ from $\left\{1, 10^1, 10^2, 10^3\right\}$ and $m$ from $\{1,5,10,100\}$. 
}

We also compare IFSBA method (Algorithm~\ref{alg:IFSBA}) with baseline methods, including ITD, AID with conjugate gradient, and F$^2$BA on "MNIST" datasets~\cite{lecun2002gradient}. We report the results on $\fD_{tr}$ with different noise rates $p  = 25\%$ and $p = 50\%$ in Figure~\ref{fig:MNIST}.

We follow the setting in Section~\ref{sec:hypercleaning} for choosing the hyperparameter. IFSBA requires additional tuning of the Cubic-Solver iterations and Matrix Chebyshev Polynomials, where the iteration steps are chosen from $\{1, 5, 10, 100\}$.

\section{Ablation Studies}

\subsection{Ablation Studies on $m$}

Our theoretical analysis suggests setting 
$m = \Theta\!\left(1 + \frac{d}{\sqrt{\kappa}}\right)$ to balance iteration complexity and per-iteration computational cost. 
To validate this choice, we conduct ablation studies on $m$ for both synthetic and data-cleaning problems. 
As shown in Figure~\ref{fig:ablation_m}, there is a clear trade-off in the choice of $m$. 
As suggested by Theorem~\ref{thm:Lazy_Hessian_calls}, the choice of $m$ introduces a trade-off between iteration complexity and Hessian-related computation. 
Smaller $m$ improves the iteration complexity due to the $m^{1/2}$ factor, whereas larger $m$ reduces the cost of Hessian updates, as captured by the $m^{-1/2}$ factor in the second-order oracle term.
In practice, moderately large values of $m$ achieve the best overall performance in terms of running time.

\begin{figure}[t]
    \centering
    \setlength{\tabcolsep}{4pt}
    \begin{tabular}{cc}
        \includegraphics[width=0.38\linewidth]{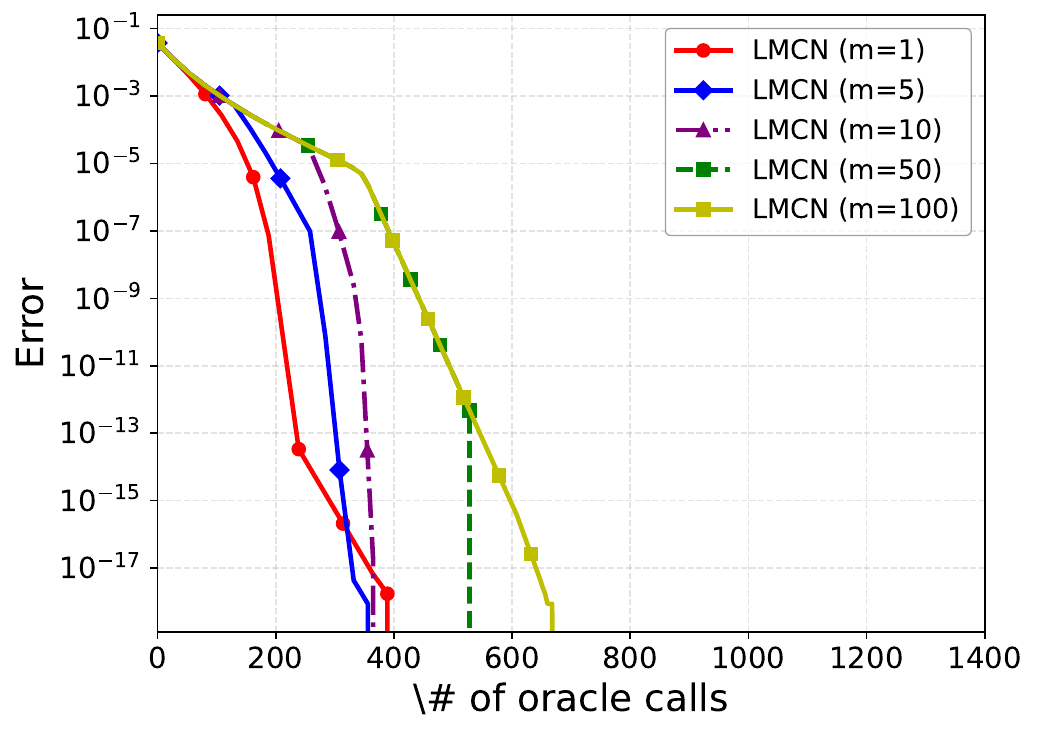} &
        \includegraphics[width=0.38\linewidth]{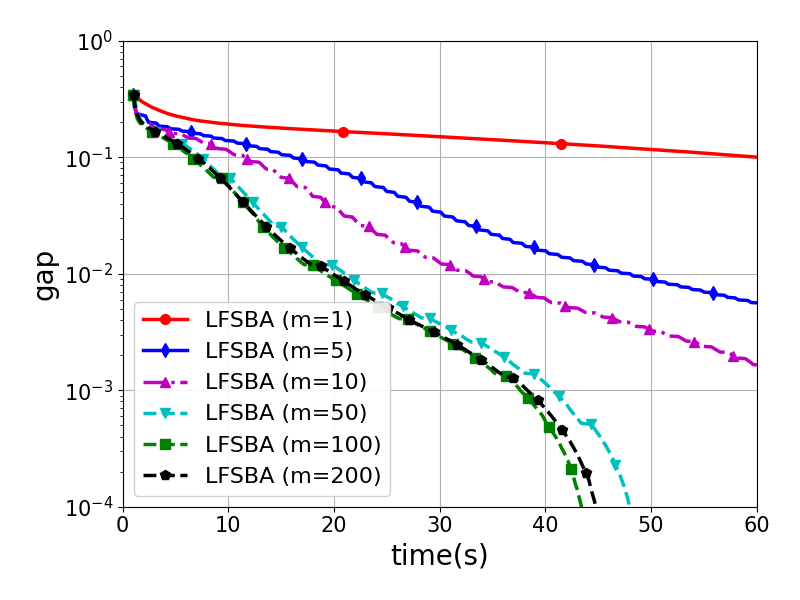} \\
        \scriptsize (a) Synthetic problem &
        \scriptsize (b) Data-cleaning problem on Austrilian dataset
    \end{tabular}
    \caption{Ablation study on the Hessian update frequency $m$.}
    \label{fig:ablation_m}
\end{figure}

\subsection{Ablation Studies on $\lambda$ and $M$}

The exact values of theoretical constants, such as the smoothness parameter $\ell$, the Hessian Lipschitz parameter $\rho$, and the condition number $\kappa$, are highly problem-dependent. 
For example, in the data hypercleaning task, assuming the logit loss is bounded by $B=\max_{i,y} L(\langle \boldsymbol{a}_i,\fy\rangle,\fy)$, the second-order derivatives of $g$ can be bounded as 
$\|\nabla_{xx}^2 g(\fx,\fy)\| \leq \frac{1}{|D_{\mathrm{tr}}|}\max_i |\sigma''(\fx_i)|L_i(\fy)=\mathcal{O}(B/|D_{\mathrm{tr}}|)$, 
$\|\nabla_{xy}^2 g(\fx,\fy)\| \leq \frac{1}{4|D_{\mathrm{tr}}|}\sqrt{\sum_{i\in D_{\mathrm{tr}}}\|\boldsymbol{a}_i\|^2}$, and 
$\|\nabla_{yy}^2 g(\fx,\fy)\| \leq \frac{1}{4|D_{\mathrm{tr}}|}\lambda_{\max}\!\left(\sum_{i\in D_{\mathrm{tr}}}\boldsymbol{a}_i\boldsymbol{a}_i^\top\right)+\mu$. 
These bounds provide estimates of the smoothness constant of $g$, and other problem-dependent constants can be bounded in a similar way. 
Nevertheless, computing tight global constants is often conservative and impractical. 
Therefore, in our implementation, we do not rely on strict worst-case bounds; instead, we treat the induced algorithmic parameters, including the cubic regularization parameter $M$ and the penalty multiplier $\lambda$, as tunable hyperparameters.

To examine the sensitivity of FSBA/LFSBA to these choices, we conduct additional ablation studies on the Australian dataset for the data hypercleaning task. 
Figures~\ref{fig:ablation_lambda} and~\ref{fig:ablation_M} report the results for different choices of $\lambda$ and $M$, respectively, while keeping the other hyperparameters fixed. 
We observe that FSBA/LFSBA consistently converges and achieves a small optimality gap even when these hyperparameters are varied over a wide range. 
Since $\lambda$ and $M$ are theoretically tied to problem-dependent constants, e.g., $M \geq c\rho$, this empirical stability demonstrates that our method is robust to the choice of these parameters in practice.

\begin{figure}[t]
    \centering
    \setlength{\tabcolsep}{3pt}
    \begin{tabular}{ccc}
        \includegraphics[width=0.32\linewidth]{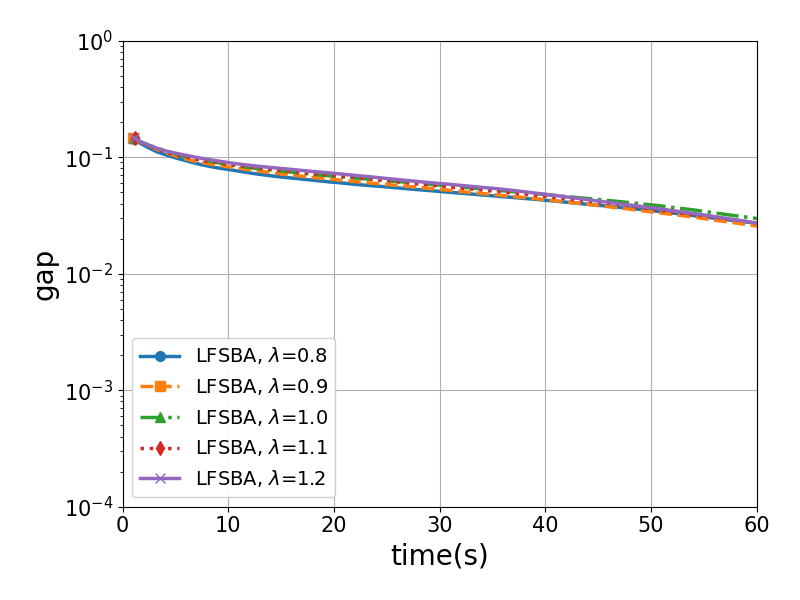} &
        \includegraphics[width=0.32\linewidth]{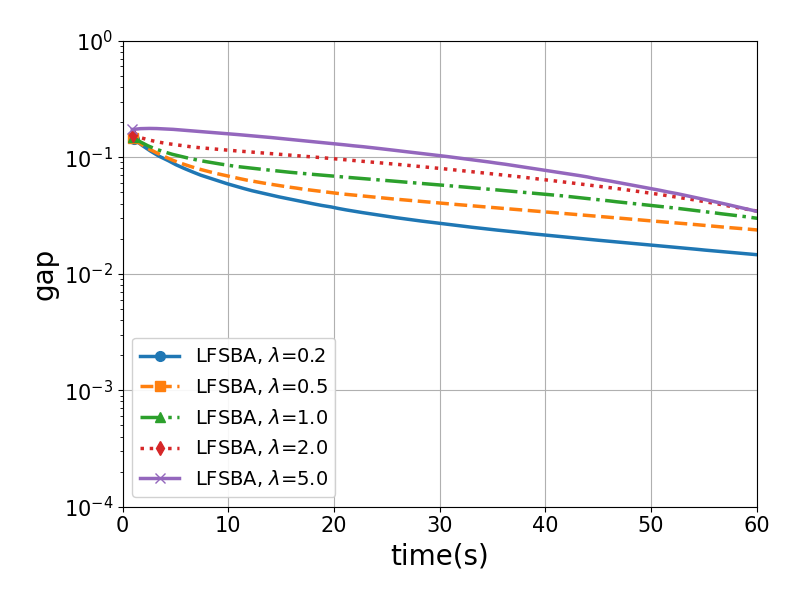} &
        \includegraphics[width=0.32\linewidth]{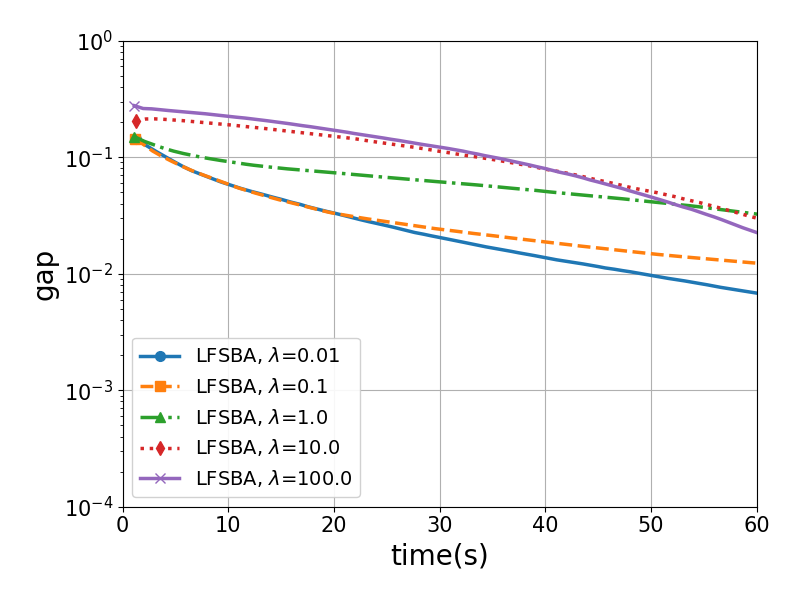} \\
        \scriptsize (a) $\lambda=1$ &
        \scriptsize (b) $\lambda=10$ &
        \scriptsize (c) $\lambda=100$
    \end{tabular}
    \caption{Ablation study on the penalty multiplier $\lambda$ for the data-cleaning task.}
    \label{fig:ablation_lambda}
\end{figure}

\begin{figure}[t]
    \centering
    \setlength{\tabcolsep}{3pt}
    \begin{tabular}{ccc}
        \includegraphics[width=0.32\linewidth]{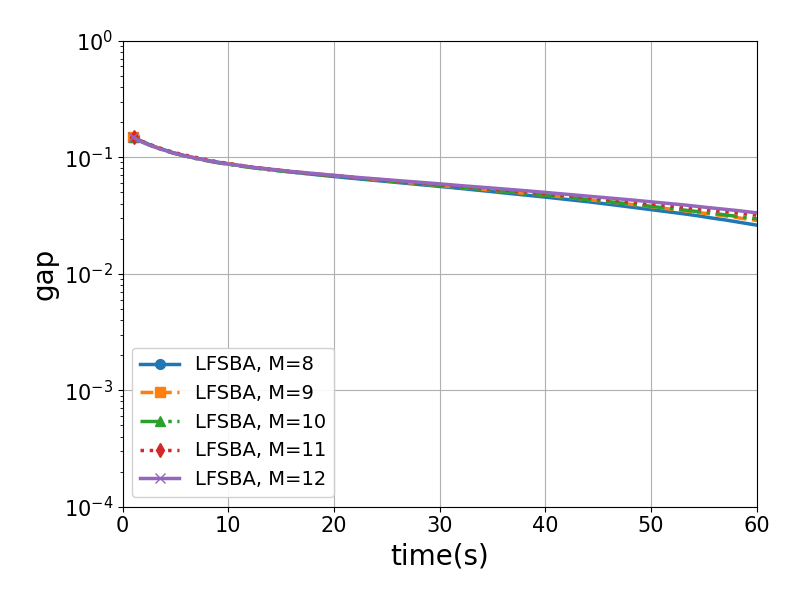} &
        \includegraphics[width=0.32\linewidth]{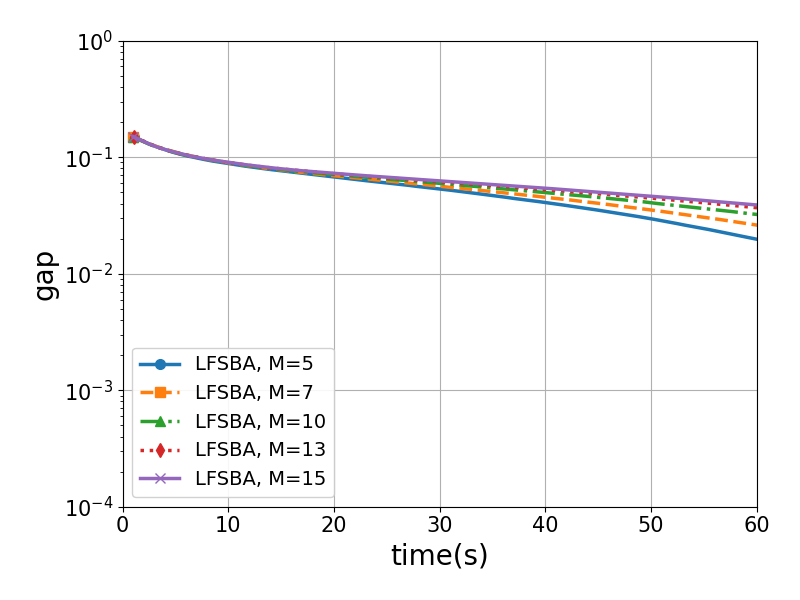} &
        \includegraphics[width=0.32\linewidth]{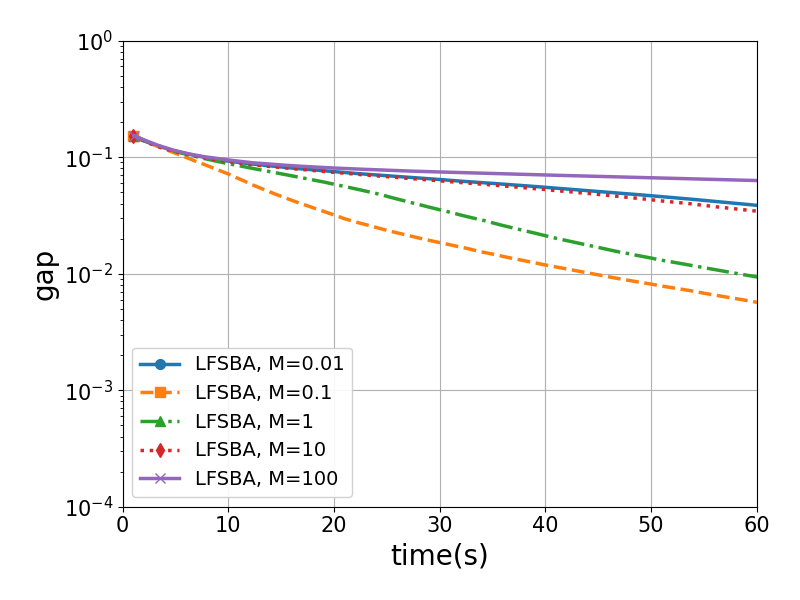} \\
        \scriptsize (a) $M=1$ &
        \scriptsize (b) $M=5$ &
        \scriptsize (c) $M=10$
    \end{tabular}
    \caption{Ablation study on the cubic regularization parameter $M$ for the data-cleaning task.}
    \label{fig:ablation_M}
\end{figure}


\end{document}